\documentclass[a4paper,twoside,11pt]{article}

\usepackage[T1]{fontenc}
\usepackage{times}
\usepackage{amsmath,amsfonts,amssymb,amsthm,euscript,pifont}
\usepackage{mathrsfs}
\usepackage{xspace}

\usepackage[colorlinks=true]{hyperref}
\usepackage{fullpage}

\newcommand{\er}{\mathbb R}
\newcommand{\nat}{\mathbb N}

\newcommand{\hyp}{({\it H})}
\newcommand{\hyplang}{({\it H$_{\it{Langevin}}$})}
\newcommand{\hyplangi}{({\it H$_{\it{Langevin}}$})-($i$)}
\newcommand{\hyplangii}{({\it H$_{\it{Langevin}}$})-($ii$)}
\newcommand{\hypedp}{({\it H$_{\rm{MVFP}}$})}
\newcommand{\hypedpi}{({\it H$_{\it{MVFP}}$})-($i$)}
\newcommand{\hypedpii}{({\it H$_{\it{MVFP}}$})-($ii$)}
\newcommand{\hypedpiii}{({\it H$_{\it{MVFP}}$})-($iii$)}
\newcommand{\hypbarriere}{({\it H$_{f_0,q}$})}

\newtheorem{theorem}{Theorem}[section]
\newtheorem{corollary}[theorem]{Corollary}
\newtheorem{definition}[theorem]{Definition}

\newtheorem{lemma}[theorem]{Lemma}
\newtheorem{proposition}[theorem]{Proposition}
\newtheorem{remark}[theorem]{Remark}

\newcommand{\Cc}{\mathcal C}
\newcommand{\Dd}{\mathcal D}

\newcommand{\Ff}{\mathcal F}

\newcommand{\Hh}{\mathcal H}

\newcommand{\Ll}{\mathcal L}

\newcommand{\Oo}{{\mathcal O}}

\newcommand{\Tt}{\mathcal T}
\newcommand{\Uu}{\mathcal U}
\newcommand{\Yy}{{\mathcal Y}}
\newcommand{\Vv}{\mathcal V}

\newcommand{\DD}{\mathbb D}
\newcommand{\EE}{\mathbb E}
\newcommand{\QQ}{\mathbb Q}
\newcommand{\NN}{\mathbb N}
\newcommand{\PP}{\mathbb P}

\usepackage{bbm}
\def\ind{{\mathbbm{1}}}

\newcommand{\trans}{\mathcal{T}}
\newcommand{\rhofp}{\rho^{\text{FP}}}
\newcommand{\okphasespace}{(\mathcal{D}\times\er^d)\cup\Sigma^{-}}
\newcommand{\okphasespacebis}{(\mathcal{D}\times\er^d){\cup{(\Sigma \setminus\Sigma^0)}}}
\newcommand{\nd}{n_{\mathcal{D}}}

\newcommand{\VVV}{{\vartheta}}

\newcommand{\SSS}{{\mathscr{S}}}
\newcommand{\TTT}{{\mathscr{T}}}

\newcommand{\VoneT}{{V_1(\omega,Q_T)}}

\numberwithin{equation}{section}

\title{Lagrangian stochastic models with specular boundary condition}

\author{Mireille Bossy\thanks{INRIA, France;
mireille.bossy@inria.fr} \and Jean-Fran\c{c}ois Jabir\thanks{CIMFAV, Universidad de Valparaíso, Chile; jean-francois.jabir@uv.cl. The second author acknowledges the CONICYT Program PAI/ACADEMIA No. 79112030, the FONDECYT Iniciaci\'on en Investigaci\'on No. 11130705 and the Programa Iniciativa Cientifica Milenio No. 130062 through the Nucleus Millennium Stochastic Models of Complex and Disordered Systems for their support.}}

\date{\today}

\begin{document}
\maketitle
\begin{abstract}
In this paper, we prove the well-posedness of a conditional McKean Lagrangian stochastic model, endowed with the specular boundary condition, and further the mean no-permeability condition, in a smooth bounded confinement domain $\Dd$. This result extends our previous work \cite{jabir-11a}, where  the confinement domain was the upper-half plane and where the specular boundary condition  has been constructed  owing to some well known results on the law of the passage times at zero of the  Brownian primitive. The extension of the construction to more general confinement domain exhibits  difficulties that we handle by combining stochastic calculus and the analysis of kinetic equations. As a prerequisite for the study of the  nonlinear case, we construct a Langevin process confined in $\overline{\Dd}$ and satisfying the specular boundary condition. We then use PDE techniques to construct the time-marginal densities of the nonlinear process from which we are able to exhibit the conditional McKean Lagrangian stochastic model.
\end{abstract}


\section{Introduction}\label{sec:Intro}
We are interested in the well-posedness of the stochastic process $((X_{t},U_{t});\,0\leq t \leq T)$, for any arbitrary finite time $T>0$, whose time-evolution is given by
\begin{equation}\label{eq:NonlinearConfinedLangevin}
\left\{
\begin{aligned}
&X_{t} =X_{0} +\int_{0}^{t} U_{s}\,ds,\\
&U_{t} = U_{0}+\int_{0}^{t}B[X_s; \rho(s)] ds + \sigma W_{t}+K_{t},\\
&K_{t}=- \sum_{0<s\leq t}
2\left(U_{s^{-}}\cdot n_{\Dd}(X_{s})\right)\nd(X_{s})
\ind_{\left\{X_{s}\in\partial\Dd\right\}},\\
&\rho(t) \mbox{ is the probability density of }(X_t,U_t)\mbox{ for all }t \in (0, T],
\end{aligned}
\right.
\end{equation}
where $(W_{t},t\geq 0)$ is a standard $\er^{d}$-Brownian motion, the diffusion $\sigma$ is a positive  constant, $\Dd$ is an open  bounded domain of $\er^d$, and $\nd$ is the outward normal unit vector of $\partial\Dd$. Eq. \eqref{eq:NonlinearConfinedLangevin} provides a  Lagrangian model describing, at each time $t$, the position $X_t$ and the velocity $U_t$ of a  particle confined within $\overline{\Dd}$.

The drift coefficient $B$  is the mapping from $\Dd\times L^{1}(\Dd\times\er^d)$ to $\er^d$ defined by
\begin{equation}\label{eq:DriftDefinitionEDP}
B[x;\psi]=\left\{
\begin{aligned}
&\frac{\displaystyle \int_{\mathbb{R}^{d}} b(v)\psi(t,x,v)dv}
{\displaystyle \int_{\mathbb{R}^{d}}\psi(t,x,v)dv}\mbox{ whenever }
\int_{\er^{d}}\psi(t,x,v)dv\neq 0,\\
&0\mbox{ otherwise,}
\end{aligned}
\right.
\end{equation}
where $b:\er^{d}\rightarrow \er^{d}$ is a given measurable function. Formally the function $(t,x) \mapsto B[x;\rho(t)]$ in \eqref{eq:NonlinearConfinedLangevin} corresponds to the conditional expectation   $(t,x) \mapsto \EE [b(U_t) | X_t = x]$ and the velocity equation in \eqref{eq:NonlinearConfinedLangevin}  rewrites
\[U_{t} = U_{0}+\int_{0}^{t}\EE[b(U_{s})|X_{s}]ds + \sigma W_{t}+K_{t}.\]
In the dynamics of $U$, the c{\`a}dl{\`a}g process $(K_{t})$  confines the component $X$ in  $\overline{\Dd}$  by reflecting the velocity of the outgoing particle. This particular confinement is linked with the specular boundary condition:
\begin{equation}\label{eq:BCond}
\gamma(\rho)(t,x,u)
= \gamma(\rho)(t,x,u-2(u\cdot n_{\Dd}(x))n_{\Dd}(x)),~dt\otimes d\sigma_{\partial \Dd}\otimes du\mbox{-a.e. on }(0,T)\times\partial\Dd\times\er^{d},
\end{equation}
where $\sigma_{\partial \Dd}$ denotes the surface measure of $\partial\Dd$ and where $\gamma(\rho)$ stands for the trace of the probability density $\rho$ on $(0,T)\times\partial\Dd\times\er^{d}$.
As already noticed in \cite[Corollary $2.4$]{jabir-11a}, under  integrability and positiveness properties on $\gamma(\rho)$, the specular condition \eqref{eq:BCond} implies the mean no-permeability boundary condition:
\begin{equation}\label{eq:nopermeability}
\frac{\displaystyle
\int_{\er^{d}} (u\cdot n_{\Dd}(x))\gamma(\rho)(t,x,u)\,du}
{\displaystyle \int_{\er^{d}}
\gamma(\rho)(t,x,u)\,du}=0,~\mbox{for}~dt\otimes d\sigma_{\partial\Dd}
\mbox{-a.e.}~(t,x)\in(0,T)\times\partial\Dd.
\end{equation}
The function $(t,x)\mapsto \dfrac{\int_{\er^{d}}(u\cdot n_{\Dd}(x))\gamma(\rho)(t,x,u)\,du}{\int_{\er^{d}} \gamma(\rho)(t,x,u)\,du}$,  on  $(0,T)\times\partial\Dd$, serves here as a formal representation of the normal component of the bulk velocity  at the boundary,
 so that \eqref{eq:nopermeability} can be seen as
\begin{equation*}
\EE[(U_{t}\cdot \nd(X_t))|X_{t} = x]=0,~\mbox{for}~dt\otimes d\sigma_{\partial\Dd}
\mbox{-a.e.}~(t,x)\in(0,T)\times\partial\Dd.
\end{equation*}
In view of \eqref{eq:nopermeability}, an appropriate notion of the trace of $\rho$ is given with the following
\begin{definition}\label{def:Trace_density_process}
Let $(\rho(t);\,t\in[0,T])$ be the time-marginal densities of a solution to \eqref{eq:NonlinearConfinedLangevin}. We say that $\gamma(\rho):(0,T)\times\partial\Dd\times\er^d\rightarrow\er$ is the trace of $(\rho(t);~t\in[0,T])$ along $(0,T)\times\partial\Dd\times\er^d$ if it is a nonnegative function $\gamma(\rho)$  satisfying,  for all $t$ in $(0,T]$, $f$ in $\Cc^{\infty}_{c}([0,T]\times\overline{\Dd}\times\er^d)$:
\begin{align}\label{Green_formula}
\begin{array}{ll}
&{\displaystyle \int_{\Sigma_{t}}
\left(u\cdot n_{\Dd}(x)\right)\gamma(\rho)(s,x,u)f(s,x,u)\,ds
\,d\sigma_{\partial \Dd}(x)\,du}\\
&{\displaystyle = -\int_{\Dd\times\er^{d}}f(t,x,u)\rho_{t}(x,u)\,dx\,du
+\int_{\Dd\times\er^{d}}f(0,x,u)\rho_{0}(x,u)\,dx\,du }\\
&\quad{\displaystyle +\int_{Q_{t}} \left(\partial_{s}f + u\cdot\nabla_x f+B[\cdot;\rho_{\cdot}]\cdot \nabla_{u}f+\frac{\sigma^2}{2}\triangle_u f\right)(s,x,u)
\rho_{s}(x,u)\,ds\,dx\,du}\\
\end{array}
\end{align}
and, for $dt\otimes d\sigma_{\partial \Dd}$ a.e. $(t,x)$ in $(0,T)\times\partial\Dd$,
\begin{subequations}
\begin{align}
&\label{nopermeabilitycondition_1}\int_{\er^{d}} |(v\cdot \nd(x))| \gamma(\rho)(t,x,v)\,dv < +\infty,\\
&\label{nopermeabilitycondition_2}\int_{\er^{d}}\gamma(\rho)(t,x,v)\,dv > 0.
\end{align}
\end{subequations}
\end{definition}
In addition to the well-posedness of \eqref{eq:NonlinearConfinedLangevin}, we prove that the solution admits a trace in a sense of Definition \ref{def:Trace_density_process} and thus satisfies the specular condition and the mean no-permeability boundary condition \eqref{eq:BCond}-\eqref{eq:nopermeability}.

Our interest in the model \eqref{eq:NonlinearConfinedLangevin} and its connection with \eqref{eq:nopermeability} arises with the modeling of boundary conditions of the Lagrangian stochastic models for turbulent flows. These models are developed in the context of Computational Fluid Dynamics (CFD) and feature a class of
stochastic differential equations with singular coefficients (we refer to Bernardin {\it et al.} \cite{jabir-10b}, Bossy {\it et al.} \cite{jabir-10a}
for an account of the various theoretical and computational issues related to these models). The design of boundary conditions for the Lagrangian stochastic models according to some Dirichlet condition or some physical wall law,  the analysis of their effects on the nonlinear dynamics and their momenta are among the current challenging issues raised by the use of Lagrangian stochastic models in CFD.

In the kinetic theory of gases, the specular boundary condition belongs to the family of the Maxwell boundary conditions which model the interaction (reflection, diffusion and absorption phenomena) between gas particles and solid surface (see Cercignani \cite{cercignani-88}).
Specifically, the specular boundary condition models the reflection of the particles at the boundary of a totally elastic wall (no loss of mass nor energy).

The intrinsic difficulty to the well-posedness of \eqref{eq:NonlinearConfinedLangevin} lies in the study of the  hitting times $\{\tau_n,n\geq 0\}$ of the particle position $(X_{t})$ on the boundary $\partial\Dd$, defined by
\begin{equation*}
 \begin{cases}
\tau_{n}=
\inf\{{ \tau_{n-1} < t \leq T} ; ~X_{t}\in\partial\Dd\},\mbox{ for}~n \geq 1,\\
\tau_{0}=0,
\end{cases}
\end{equation*}
which must tend to infinity to ensure that $(K_t)$ is well defined (with the convention that $\inf\{\emptyset\} = + \infty$).
By Girsanov Theorem, it is not difficult to see that the sequence $\{\tau_{n};\,n\geq 0\}$ is related to the attaining times of the primitive of the Brownian motion on a smooth  surface.

In the previous work \cite{jabir-11a}, we established the  well-posedness of  \eqref{eq:NonlinearConfinedLangevin} in the case where $\Dd$ is the upper half-plane $\er^{d-1}\times(0,+\infty)$. In this situation, only one component of the process is confined in $[0,+\infty)$, and our construction of the confined process mainly relies on the explicit distribution of the zero-sets of the primitive of one dimensional Brownian motion given in  McKean \cite{McKean1963} and Lachal \cite{Lachal1997}. To the best of our knowledge, similar results on these attaining times have only been extended in the case of bounded interval.
Note also that in the case treated in \cite{jabir-11a}, the existence problem of trace functions in the sense of Definition \ref{def:Trace_density_process} is solved thanks to the explicit construction of the confined linear Langevin process.

Here some new  difficulties are enhanced by the boundary reflection generalized to any smooth bounded domain $\Dd$.
Those difficulties  appear first in the construction of the confined linear  Langevin process (see Eq. \eqref{eq:LinearConfinedLangevin} that corresponds to \eqref{eq:NonlinearConfinedLangevin} with $b=0$), next in the treatment of the McKean nonlinearity in \eqref{eq:NonlinearConfinedLangevin} and in the verification of the mean-no-permeability condition.

The approach that we propose in this paper strongly mixes stochastic analysis with PDE analysis.

\subsection{Main result}

From now on, we implicitly assume  that $\sigma$ is positive and that $\Dd$ is an open bounded domain in $\er^d$. In addition, the set of hypotheses for the main theorem below is denoted by \hyp. In this set we distinguish \hyplang, the hypotheses for the construction of the linear Langevin  process, and \hypedp\, the hypotheses for the well-posedness of the nonlinear 
Vlasov-Fokker-Planck equation related to \eqref{eq:NonlinearConfinedLangevin} (see Eq. \eqref{eq:ConditionalMcK-VFP}), as follows:
\begin{description}
\item{\hyplangi} The initial condition $(X_{0},U_{0})$ is assumed to be distributed according to a given initial law $\mu_{0}$ having its support in $\Dd\times\er^{d}$ and such that  $\int_{\Dd\times\er^{d}}\left(|x|^{2}+|u|^{2}\right)\mu_{0}(dx,du)<+\infty$.
\item{\hyplangii} The boundary ${\partial \Dd}$ is a compact $\Cc^3$ submanifold of $\er^d$. \end{description}
\begin{description}
\item{\hypedpi} $b:\er^{d}\rightarrow \er^{d}$ {is a bounded measurable function}.
\item{\hypedpii} The initial law $\mu_{0}$ has a density $\rho_{0}$ in the weighted  space $L^{2}(\omega,\Dd\times\er^{d})$ with $\omega(u):=(1+|u|^{2})^{\frac{\alpha}{2}}$ for some $\alpha>d \vee 2$ (see the Notation subsection below for a precise definition).
\item{\hypedpiii} There exist two measurable functions $\underline{P}_{0}$, $\overline{P}_{0}:\er^{+}\longrightarrow \er^{+}$ such that
\begin{align*}
&0 < {\displaystyle \underline{P}_{0}(|u|)\leq\rho_{0}(x,u)\leq \overline{P}_{0}(|u|),~\mbox{a.e. on}~\Dd\times\er^{d};}
\\
\mbox{and } \quad&{\displaystyle \int_{\er^{d}} (1+|u|)\omega(u)\overline{P}_{0}^2(|u|)du<+\infty.}
\end{align*}
\end{description}
Let us precise the notion of solution that we consider for \eqref{eq:NonlinearConfinedLangevin}. A probability measure $\QQ$, in the sample space $\TTT:=\Cc([0,T];\overline{\Dd})\times\DD([0,T];\er^{d})$ with canonical process $(x(t),u(t); t\in[0,T])$,  is a solution in law to \eqref{eq:NonlinearConfinedLangevin} if for all $t\in[0,T]$, $\QQ\circ(x(t),u(t))^{-1}$ admits a density function $\rho(t)$ with $\rho(0)=\rho_{0}$ and there exists an $\er^{d}$-Brownian motion $(w(t);\,t\geq 0)$ under $\QQ$, such that $\QQ$-a.s.
\begin{equation*}
\left\{
\begin{aligned}
x(t) &= x(0) + \int_{0}^{t} u(s)\,ds,\\
u(t) &= u(0) + \int_{0}^{t}B[x(s);\rho(s)]\,ds+\sigma w(t) -\sum_{\substack{0 < s \leq t}} 2\left(u(s^{-})\cdot n_{\Dd}(x(s))\right)n_{\Dd}(x(s))
\ind_{\displaystyle \left\{x(s)\in \partial \Dd\right\}}.
\end{aligned}
\right.
\end{equation*}
We further introduce the set
\begin{align*}
\Pi_{\omega}&:=\left\{\QQ\mbox{, probability measure on}\,\TTT\,\mbox{s.t.,}\,\mbox{for all}\,t\in[0,T],\,\QQ\circ(x(t),u(t))^{-1}\in L^{2}(\omega;\Dd\times\er^{d})\right\}.
\end{align*}
\begin{theorem}\label{thm:main}
Under \hyp, there exists a unique solution in law to \eqref{eq:NonlinearConfinedLangevin} in $\Pi_{\omega}$.

Moreover the set of time-marginal densities $(\rho(t,\cdot),t\in[0,T])$ is in $\VoneT$ and admits a trace $\gamma(\rho)$ in the sense of Definition \ref{def:Trace_density_process} which satisfies the no-permeability boundary condition \eqref{eq:nopermeability}.
\end{theorem}
The precise definition of the weighted Sobolev space $\VoneT$ is given in the Notation subsection below.

The rest of the paper is devoted to the proof of Theorem \ref{thm:main}.
In Section \ref{sec:Linearcase}, we set the {\it linear} basis of our approach: we construct the solution to the confined linear Langevin process (solution to \eqref{eq:LinearConfinedLangevin}) and we study the property of its semi-group. This latter will rely on a Feynman-Kac interpretation of the semi-group and the analysis of  the  boundary value problem
\begin{eqnarray}
\left\{
\begin{aligned}
&\partial_{t} f(t,x,u) - (u\cdot \nabla_{x}f(t,x,u)) -\frac{\sigma^2}{2}\triangle_{u}f(t,x,u)= 0,~\forall~(t,x,u)\in(0,T]\times\Dd \times\mathbb{R}^{d},\\
&\lim_{t\rightarrow 0^{+}}f(t,x,u) = f_{0}(x,u),~\forall~(x,u)\in\Dd \times\mathbb{R}^{d},\\
&f(t,x,u) = q(t,x,u),~\forall~(t,x,u)\in\Sigma^{+}_{T},
\end{aligned}\right.
\end{eqnarray}
for which we prove the existence of a smooth solution, continuous at the boundary (see Theorem \ref{thm:L2VFPDirichlet}).

In Section \ref{sec:PDEApproach}, using a PDE approach, we construct a set of density functions  $(\rho(t);t\in[0,T])$ that features the time-marginal densities of a solution in law to \eqref{eq:NonlinearConfinedLangevin}. More precisely, we construct a weak solution to the following  nonlinear McKean-Vlasov-Fokker-Planck equation with specular boundary condition:
\begin{align*}
\left\{\begin{array}{l}
\partial_{t} \rho + (u\cdot\nabla_{x}\rho) + (B[\cdot~;\rho]\cdot\nabla_{u}\rho)
-\frac{\sigma^{2}}{2}\triangle_{u} \rho = 0,~\mbox {on}~(0,T]\times\Dd \times\mathbb{R}^{d},\\
\rho(0,x,u) = \rho_{0}(x,u),~\mbox{ on}~\Dd \times\mathbb{R}^{d},\\
\gamma(\rho)(t,x,u) = \gamma(\rho)(t,x,u-2(u\cdot \nd(x))\nd(x)),~\mbox{ on}~(0,T)\times\partial\Dd \times\mathbb{R}^{d},
\end{array}\right.
\end{align*}
where $\gamma(\rho)$ stands for the trace of $\rho$ in the sense of Definition \ref{def:Trace_density_process} (see Theorem   \ref{thm:Existence_result} for the existence result). In particular, the verification of the properties \eqref{nopermeabilitycondition_1} and \eqref{nopermeabilitycondition_2} is obtained thank to the construction of Maxwellian bounds for the solution to the nonlinear PDE and its trace at the boundary.
Starting from this solution, we set a drift $B(t,x) = B[x;\rho(t)]$ from \eqref{eq:DriftDefinitionEDP} and   we construct a process candidate to be a solution of \eqref{eq:NonlinearConfinedLangevin} using {a } change of probability measure from the confined Langevin law constructed in Section  \ref{sec:Linearcase}.  We  achieve the proof of Theorem \ref{thm:main}  in Section \ref{sec:LawReconstruction},  by proving that the resulting set of time-marginal densities coincides with the solution to the McKean-Vlasov-Fokker-Planck equation considered in Section \ref{sec:PDEApproach}. We also prove the uniqueness in law for the solution of
\eqref{eq:NonlinearConfinedLangevin} in $\Pi_{\omega}$.
\subsection{Notation}
For all $t\in(0,T]$, we introduce the time-phase space
\begin{equation*}
Q_{t}:=(0,t)\times\Dd\times\er^{d},
\end{equation*}
and the boundary sets:
\begin{align*}
\begin{array}{lll}
&\Sigma^{+}:=\left\{(x,u)\in\partial\Dd \times\er^{d}~\mbox{s.t.}~(u\cdot \nd(x))>0\right\},&\quad\Sigma^{+}_{t}:=(0,t)\times\Sigma^{+},\\
&\Sigma^{-}:=\left\{(x,u)\in\partial\Dd\times\er^{d}~\mbox{s.t.}~(u\cdot \nd(x))<0\right\},&\quad\Sigma^{-}_{t}:=(0,t)\times\Sigma^{-},\\
&\Sigma^{0}:=\left\{(x,u)\in\partial\Dd\times\er^{d}~\mbox{s.t.}~(u\cdot \nd(x))=0\right\},&\quad\Sigma^{0}_{t}:=(0,t)\times\Sigma^{0},
\end{array}
\end{align*}
and further $\Sigma_{T}:=\Sigma^{+}_{T}\cup \Sigma^{0}_{T}\cup \Sigma^{-}_{T} = (0,T)\times\partial\Dd\times\er^{d}$.
Denoting by $d\sigma_{\partial \Dd}$ the surface measure on $\partial\Dd$, we introduce the product measure on $\Sigma_{T}$:
\[d\lambda_{\Sigma_{T}}:=dt\otimes d\sigma_{\partial \Dd}(x)\otimes du.\]
We set the Sobolev space
\begin{equation*}
\Hh(Q_{t})=L^{2}((0,t)\times\Dd;H^{1}(\er^{d}))
\end{equation*}
equipped with the norm $\Vert\,\Vert_{\Hh(Q_{t})}$ defined by
\begin{equation*}
\|\phi\|^{2}_{\Hh(Q_{t})} = \|\phi\|^{2}_{L^{2}(Q_{t})} +
\|\nabla_{u}\phi\|^{2}_{L^{2}(Q_{t})}.
\end{equation*}
We denote by $\Hh'(Q_{t})$, the dual space of $\Hh(Q_{t})$, and by
$\left(~,~\right)_{\Hh'(Q_{t}),\Hh(Q_{t})}$, the inner product between $\Hh'(Q_{t})$ and $\Hh(Q_{t})$.

We define the weighted Lebesgue space
\begin{equation*}
L^{2}(\omega,Q_t):=\left\{\psi:Q_t\rightarrow\er~ \mbox{ s.t } \sqrt{\omega}\psi\in L^{2}(Q_t)\right\},
\end{equation*}
with the weight function $u\mapsto \omega(u)$ on the velocity variable
\begin{equation}
\label{def:Poids}
\omega(u):=(1+|u|^{2})^{\frac{\alpha}{2}},\,\mbox{  for }\,\alpha>d\vee 2.
\end{equation}
We endow $L^{2}(\omega,Q_t)$ with the norm $\Vert\,\Vert_{L^{2}(\omega,Q_t)}$ defined by
$\|\phi\|^{2}_{L^{2}(\omega,Q_t)} = \|\sqrt{\omega} \phi\|_{L^2(Q_t)}$.

We introduce the weighted Sobolev space
\begin{equation*}
\Hh(\omega,Q_t):=\left\{\psi\in L^{2}(\omega,Q_t)~ \mbox{ s.t } \left|\nabla_{u}\psi\right|\in L^{2}(\omega,Q_t)\right\}
\end{equation*}
with the norm
$\|\,\|_{\Hh(\omega,Q_t)}$ defined by
\begin{equation*}
\|\phi\|^{2}_{\Hh(\omega,Q_t)} = \|\phi\|^{2}_{L^{2}(\omega,Q_t)} +
\|\nabla_{u}\phi\|^{2}_{L^{2}(\omega,Q_t)}.
\end{equation*}
Finally, we define the set
\begin{equation*}
\VoneT=\Cc\left([0,T]; L^{2}(\omega,\Dd\times\er^{d})\right)\cap \Hh(\omega,Q_T),
\end{equation*}
equipped with the norm
\begin{equation*}
\Vert\phi\Vert^{2}_{\VoneT} = \max_{t\in\left[0,T\right]}\left\{\int_{\Dd \times\er^{d}}
\omega(u)\left|\phi(t,x,u)\right|^{2}\,dx\,du\right\} + \int_{Q_{T}} \omega(u)|\nabla_{u} \phi(t,x,u)|^{2}\,dt\,dx\,du.
\end{equation*}
We further introduce the spaces
\begin{align*}
&L^{2}(\Sigma^{\pm}_{T})
=\big\{\psi:\Sigma^{\pm}_{T}\rightarrow \er \mbox{ s.t }
\int_{\Sigma^{\pm}_{T}} |(u\cdot \nd(x))|\left|\psi(t,x,u)\right|^{2}
d\lambda_{\Sigma_{T}}(t,x,u)<+\infty \big\},\\
&L^{2}(\omega,\Sigma^{\pm}_{T})=\big\{\psi:\Sigma^{\pm}_{T}\rightarrow \er \mbox{ s.t }
\int_{\Sigma^{\pm}_{T}} \omega(u)|(u\cdot \nd(x))|\left|\psi(t,x,u)\right|^{2}
d\lambda_{\Sigma_{T}}(t,x,u)<+\infty \big\},
\end{align*}
equipped with their respective norms
\begin{equation*}
\begin{aligned}
&\Vert \psi\Vert_{L^{2}(\Sigma^{\pm}_{T})}
=\sqrt{\int_{\Sigma^{\pm}_{T}} |(u\cdot \nd(x))|~\left|\psi(t,x,u)\right|^{2}
d\lambda_{\Sigma_{T}}(t,x,u)},\\
&\Vert \psi\Vert_{L^{2}(\omega,\Sigma^{\pm}_{T})}
=\sqrt{\int_{\Sigma^{\pm}_{T}} \omega(u)|(u\cdot \nd(x))|~\left|\psi(t,x,u)\right|^{2}
d\lambda_{\Sigma_{T}}(t,x,u)}.
\end{aligned}
\end{equation*}

\section{Preliminaries: the confined Langevin process}\label{sec:Linearcase}
In this section, we prove the well-posedness of the  confined linear Langevin  equation, namely there exists a unique solution, defined on a filtered probability space
$(\Omega,\Ff,(\Ff_t;\,t\in[0,T]),\PP)$ endowed with an $\er^d$-Brownian motion $W$, to
\begin{equation}\label{eq:LinearConfinedLangevin}
\left\{
\begin{aligned}
X_{t} &= x_{0} + \int_{0}^{t} U_{s}\,ds,\\
U_{t} &= u_{0} + \sigma W_{t}+K_{t},\\
K_{t}&= - 2\sum_{\substack{0 \leq s \leq t}}\left(U_{s^{-}}\cdot n_{\Dd}(X_{s})\right)n_{\Dd}(X_{s})
\ind_{\displaystyle \left\{X_{s}\in \partial \Dd\right\}},~\forall t\in[0,T],
\end{aligned}\right.
\end{equation}for any $(x_0,u_0) \in (\Dd\times\er^{d})\cup(\Sigma\setminus\Sigma^{0})$.
We further investigate some properties  of its semigroup (notably the $L^{p}$-stability).

\subsection{Well-posedness of \eqref{eq:LinearConfinedLangevin}}
We focus our well-posedness result to  the case where $(x_0,u_0) \in  \okphasespace$,  which is the situation where either the particle starts inside $\Dd$ or  starts at  boundary with an  ingoing velocity. We naturally  extend the solution  to the case where $(x_0,u_0)\in \Sigma^{+}$ (namely the situation of an initial outgoing velocity) by defining (with the flow notation) the solution of \eqref{eq:LinearConfinedLangevin} starting from $(x_0,u_0)\in \Sigma^{+}$ by
\begin{align*}
\left((X_t,U_t)^{0,x_0,u_0};\,t\in[0,T]\right) =
\left((X_t,U_t)^{0,x_0,u_0-2(u_0\cdot\nd(x_0))\nd(x_0)};\,t\in[0,T]\right).
\end{align*}

The construction presented hereafter takes advantage of the regularity of $\partial \Dd$ to locally straighten the boundary, in the same manner than the construction of  the diffracted process across a submanifold in \cite{BoChMaTa-10}. This allows us to adapt the one dimensional construction proposed in \cite{jabir-11a} and based on the explicit law of the sequence of passage times at zero of the 1D-Brownian motion primitive (see \cite{McKean1963}, \cite{Lachal1997}). Our main result is the following:
\begin{theorem}\label{thm:LinearSolutionByStraightening}
Under \hyplangii, for any $(x_0,u_0) \in \okphasespace$, there exists a  weak solution to \eqref{eq:LinearConfinedLangevin}. Moreover, the sequence of hitting times
\[ \tau_n = \inf\{ \tau_{n-1} < t\leq T ~;~ X_t \in \partial \Dd\},\mbox{  for } n\geq1,\quad \tau_0 = 0, \]
is well defined and grows to infinity. The pathwise uniqueness holds for the solution of \eqref{eq:LinearConfinedLangevin}.
\end{theorem}
For the sake of completeness, we recall some results related to the local straightening of the boundary $\partial\Dd$  as given in \cite{BoChMaTa-10}.  Since  $\partial\Dd$ is $\Cc^3$, one can construct a $\Cc^{2}_b$-mapping $\pi$  from a neighborhood $\mathcal{N}$ of $\partial\Dd$ to $\partial\Dd$ such that
$$ |x-\pi(x)| = d(x,\partial\Dd),\quad\forall x\in\mathcal{N}, $$
where $d(x,\partial\Dd)$ denotes the distance between $x$ and the set $\partial\Dd$. Note that, reducing ${\cal N}$ if necessary, we can always assume that $\pi$ is $\Cc^2_b(\overline{{\cal N}})$. For all $x\in{\cal N}$, we set
\begin{equation}
\label{eq:signed-distance}
\varsigma(x):=\big((x-\pi(x))\cdot \nd(\pi(x))\big),
\end{equation}
so that  $\varsigma(x)$ is the signed distance to $\partial\Dd$ (positive in $\er^{d}\setminus\overline{\Dd}$, negative in $\Dd$) and is of
class $\Cc^2_b({\cal N})$.  We still denote by
$\varsigma$ a $\mathcal{C}^2(\er^d)$-extension of this function to the whole Euclidean space.  It is well-known (see e.g.~\cite[p.\:355]{GilbTrud-83}) that
\begin{equation}  \label{eq:nabla-d}
\nabla\varsigma(x)=\nd(\pi(x)),\quad  \forall x\in{\cal N}.
\end{equation}
\begin{proposition} [see \cite{BoChMaTa-10}, Proposition 2.1]\label{prop:straightening}
Under \hyplangii, there exist a family of bounded open subsets of ${\cal N}$, $\{\Uu_1,\ldots,\Uu_{M-1}\}$
  such that $\partial\Dd\subset\cup_{i=1}^{M-1}\Uu_i$, and a family of {$\er^{d}$-valued }functions
  $\{\psi_1,\ldots,\psi_{M-1}\}$ such that, for all $1\leq i\leq M-1$, $\psi_i=(\psi_i^{(1)},\ldots,\psi_i^{(d)})$ is
  a $\Cc^2_b$ diffeomorphism from $\Uu_i$ to $\psi_i(\Uu_i)$, admitting a $\Cc^2_b$ extension
  on $\overline{\Uu_i}$ and satisfying for all $x\in\overline{\Uu_i}$
\begin{eqnarray}
  \left\{\begin{array}{rl}
    \psi_i^{(d)}(x) =&\varsigma (x), \label{eq:straightening} \\
    \nabla\psi_i^{(k)}(x)\cdot \nd(\pi(x))=&0,\quad \forall k\in\{1,2,\ldots,d-1\}, \label{eq:ortho} \\
    \displaystyle\frac{\partial \psi_i^{-1}}{\partial x_d}(\psi_i(x))=&\nd(\pi(x)). \label{eq:inverse-psi}
    \end{array}\right.
    \end{eqnarray}
\end{proposition}
Note that by~(\ref{eq:straightening}), $\psi_i(\Uu_i\cap\partial\Dd)\subset\er^{d-1}\times\{0\}$, which
justifies the term ``local straightening'' of the boundary.

Let $\Uu_M$ be an open subset of $\er^d$ such that $\partial\Dd\cap\Uu_M=\emptyset$  and $\cup_{i=1}^M\Uu_i=\er^d$, and set ${\psi}_M(x):=x$ on $\Uu_M$.
\begin{proof}[Proof of Theorem \ref{thm:LinearSolutionByStraightening}]
For any $(x,u) \in \okphasespace$, we consider the flow of processes $((x_t,u_t)^{s,x,u};\,{s}\leq t\leq T)$ in $\er^{2d}$, defined by
\begin{align}\label{eq:freelangevin}
\left\{\begin{array}{l}
x_t^{s,x,u} =x + \int^t_{{s}} u_{r}^{s,x,u} dr, \\
u_t^{s,x,u} = u + \sigma( W_t-W_s).
\end{array}\right.
\end{align}
For notation convenience,  we set $(x_t,u_t):=(x_t,u_t)^{{0,x_{0},u_{0}}}$ for $(x_0,u_0) \in \okphasespace$, and assume that the process is constructed on the canonical filtered probability space
$(\Omega,\Ff,(\Ff_{t};\,t\in[0,T]),\PP_{x_0,u_0})$. We introduce the index $i_1$ that corresponds to the smallest index of the open subsets for which ${x_{0}}$ is the most ``deeply'' contained:
\begin{align*}
i_{1}&=\min\left\{i \in \{1,\ldots, M\}; \,d(x_{0},\er^d\setminus\Uu_{i})=\max_{1\leq m\leq M}d(x_{0},\er^d\setminus\Uu_{m})\right\}.
\end{align*}
We consider also the exit time
\begin{equation*}
\zeta_{1}=\inf\left\{t\geq 0\,;\,x_{t}\notin\Uu_{i_{1}}\right\}.
\end{equation*}
If $i_1=M$, we set for all $t\leq \zeta_{1}$,
\begin{align*}
(X_t,{U}_t) = (x_t,u_t).
\end{align*}
Else, as the diffeomorphism $\psi_{i_1}$ is $\Cc^2_b$ on $\mathcal {\overline{U}}_{i_1}$
and  satisfies \eqref{eq:straightening}, we can  apply the It\^o formula between 0 and $t \leq \zeta_1$ to the vector process  $(Y_{t},V_{t})=((Y^{(k)}_{t},V^{(k)}_{t});\,k=1,...,d)$ given by
\begin{align*}
(Y^{(k)}_t, V^{(k)}_t) := \Big( \psi^{(k)}_{i_1}(x_t), (\nabla \psi^{(k)}_{i_1}(x_t) \cdot u_t)\Big)
=\Big( \psi^{(k)}_{i_1}(x_t), \sum_{l=1}^{d}\partial_{x_{l}} \psi^{(k)}_{i_1}(x_t) u^{(l)}_t\Big).
\end{align*}
For all $0\leq t\leq \zeta_1$,  we obtain that $(Y, V)$ is a solution to the following system of SDE: for $k=1,\ldots,d$,
\begin{equation}
\label{eq:StraightedLangevin}
\left\{
\begin{aligned}
Y^{(k)}_{t} = & {\displaystyle \psi_{i_1}({x_{0}}) + \int_{0}^{t} V^{(k)}_{s} ds,}\\
V^{(k)}_{t} =& \left(\nabla \psi^{(k)}_{i_1}({x_{0}}) \cdot {u_0}\right)
 + \sigma \int_0^t \sum_{l=1}^d  \left(\displaystyle{\ind_{\{Y_s\in \psi_{i_1}(\mathcal {\overline{U}}_{i_1})\}}} \partial_{x_{l}}\psi^{(k)}_{i_1}(\psi^{-1}_{i_1}(Y_s))  + \displaystyle{\frac{1}{\sqrt{d}}\ind_{\{Y_s\notin \psi_{i_1}(\mathcal {U}_{i_1})\}}}\right)
 dW^{(l)}_s, \\
& + \int_0^t
\displaystyle{\ind_{\{Y_s\in \psi_{i_1}(\mathcal {U}_{i_1})\}}}
\sum_{1\leq l,n\leq d}\partial^2_{x_l,x_n}\psi^{(k)}_{i_1}(\psi^{-1}_{i_1}(Y_s))(\nabla \psi^{-1}_{i_1}(Y_s) V_s)^{(l)}(\nabla \psi^{-1}_{i_1}(Y_s) V_s)^{(n)}\; ds.
\end{aligned}\right.
\end{equation}
The SDE above has a non-homogeneous diffusion coefficient and a drift coefficient with quadratic growth.
Nevertheless, since $\max_{t\in[0,T]}|u_{t}|^{2}$ is finite $\PP$-a.s., the same holds true for $\max_{t\in[0,T]}|V_{t\wedge \zeta_1}|^{2}$   so that the solution does not explode at finite time and is pathwise unique.

Note that, from  \eqref{eq:straightening} and \eqref{eq:nabla-d}, $|\nabla \psi_{i_1}^{(d)}(x)| =1$ on $\Uu_{i_1}$,   so that the stochastic integral in $V^{(d)}_{t}$ is a local  martingale with a quadratic variation given by
\begin{equation*}
\Big \langle \sum_{l=1}^d  \left(\displaystyle{\ind_{\{Y_s\in \psi_{i_1}(\mathcal {\overline{U}}_{i_1})\}}} \partial_{x_{l}}\psi^{(d)}_{i_1}(\psi^{-1}_{i_1}(Y_s))  + \displaystyle{\frac{1}{\sqrt{d}}\ind_{\{Y_s\notin \psi_{i_1}(\mathcal {U}_{i_1})\}}}\right)
 dW^{(l)}_s \Big \rangle_{t} = t.
\end{equation*}
By Levy characterization, it follows that
\begin{align*}
\widetilde{W}^{(d)}_t  : = \int_0^t \sum_{l=1}^d  \left(\displaystyle{\ind_{\{Y_s\in \psi_{i_1}(\mathcal {\overline{U}}_{i_1})\}}} \partial_{x_{l}}\psi^{(d)}_{i_1}(\psi^{-1}_{i_1}(Y_s))  + \displaystyle{\frac{1}{\sqrt{d}}\ind_{\{Y_s\notin \psi_{i_1}(\mathcal {U}_{i_1})\}}}\right)
 dW^{(l)}_s
\end{align*}
is a standard Brownian motion in $\er$. We also set  $$\widetilde{W}^{(k)}_t  : =\int_0^t \sum_{l=1}^d ({\ind_{\{Y_s\in \psi_{i_1}(\mathcal {\overline{U}}_{i_1})\}}} \partial_{x_{l}}\psi^{(k)}_{i_1}(\psi^{-1}_{i_1}(Y_s))  + {\frac{1}{\sqrt{d}}\ind_{\{Y_s\notin \psi_{i_1}(\mathcal {U}_{i_1})\}}}) dW^{(l)}_s,$$
for $k=1, \ldots, d-1$. Now, from the identity $\psi^{-1}(\psi(x)) = x$, we easily derive that 
$$\sum_{i=1}^d \partial_{x_i}{\left(\psi^{-1}\right)^{(k)}}(\psi(x))\partial_{x_l}\psi^{(i)}(x){=\delta_{kl}},$$
 {where $\delta_{kl}$ is the Kronecker delta}, and
\[\left(\psi_{i_1}^{-1}(Y_t), \nabla\psi^{-1}_{i_1}(Y_t) V_t\right) = (x_t,u_t).\]
Then for any component $k$, the drift term of  $V^{(k)}_{t}$ writes
\begin{equation*}
\int_0^t \displaystyle{\ind_{\{Y_s\in \psi_{i_1}(\mathcal {\overline{U}}_{i_1})\}}}\sum_{1\leq l,n\leq d}\partial^2_{x_l,x_n}\psi^{(k)}_{i_1}(x_s) u_s^{(l)}u_s^{(n)} ds
\end{equation*}
and is such  that
\begin{align*}
&\EE \left[\int_0^t \left(\displaystyle{\ind_{\{Y_s\in \psi_{i_1}(\mathcal {\overline{U}}_{i_1})\}}}\sum_{1\leq l,n\leq d}\partial^2_{x_l,x_n}\psi^{(k)}_{i_1}(x_s)u_s^{(l)}u_s^{(n)}\right)^2 ds\right]\\
& \quad \leq \sup_{1\leq l,n\leq d}\|\partial^2_{x_l,x_n}\psi^{(k)}_{i_1}\|_{L^\infty(\Uu_{i_1})} \EE\left[ \int_0^t\left(u_s^{(l)}u_s^{(n)}\right)^2 ds\right] < +\infty.
\end{align*}
Consequently (see Lipster-Shiryaev \cite[Theorem 7.4]{LipShi-77}), the law of $((Y_{t},V_{t});\,t\in[0,T])$
is absolutely continuous w.r.t. the law of $((\Yy_{t},\Vv_{t});\,t\in[0,T])$, solution to
\begin{equation}\label{eq:equivLangevinStraightened}
\left\{
\begin{aligned}
\Yy^{(k)}_{t} &= {\displaystyle \psi^{(k)}_{i_1}(x_{0}) + \int_{0}^{t} \Vv^{(k)}_{s} ds,}\\
\Vv^{(k)}_{t} &= \left(\sum_{l=1}^{d}\partial_{x_{l}} \psi^{(k)}_{i_1}({x_0})   u^{(l)}_{0}\right) + \sigma \widetilde{W}^{(k)}_t.
\end{aligned}
\right.
\end{equation}
In particular, $\Yy^{(d)}$ is the primitive of the Brownian motion $\widetilde{W}^{(d)}$ and   McKean \cite{McKean1963} has shown that if $(\Yy^d_{0},\Vv^d_{0})\neq (0,0)$ then, $\PP$-almost surely, the path  $t \mapsto (\Yy^{(d)}_{t}, \Vv^{(d)}_{t})$ never crosses $(0,0)$.
Thus the sequence of passage times at zero of $(\Yy^{(d)}_{t})$ tends to infinity, as $(x_0,u_0) \in \okphasespace$, and  the same holds true for the sequence of passage times at zero $(\beta^{1}_{n},n\geq 1, \beta^1_0 = 0)$ of $(Y^{(d)}_t)$ as well as for the sequence of hitting times of the boundary $\partial \Dd$ by  $(x_t)$.
In particular
\[\beta^{1}_1 = \inf\{t>0; \Yy^{(d)}_{t} =0\} = \inf\{t>0; Y^{(d)}_{t}=0\} = \inf\{t>0; x_t \in \partial \Dd\}.\]
Now we set
\begin{align*}
(X_t,U_t) = (x_t,u_t),  \mbox{ for all } 0\leq t<  \beta^{1}_1\wedge \zeta_1.
\end{align*}
Suppose that $\beta^{1}_{1}<\zeta_1$. At time $\beta^{1}_{1}$, as $x_{\beta^{1}_{1}}\in \partial \Dd$ with $(u_{\beta^{1}_{1}}\cdot \nd(x_{\beta^{1}_{1}})) >0$ $\PP$-a.s., we reflect the velocity as follows:
\begin{eqnarray*}
(X_{\beta^{1}_{1}},U_{\beta^{1}_{1}}) =
\left(x_{\beta^{1}_{1}},u_{{\beta^{1}_{1}}^-}
-2(u_{{\beta^{1}_{1}}^-}\cdot\nd(x_{\beta^{1}_{1}}))\nd(x_{\beta^{1}_{1}})\right).
\end{eqnarray*}
We resume the first step of our construction: we set $\theta_0 = \zeta_0 = \beta^0_1 = 0$ and  we have defined
\begin{align*}
i_{1}&=\min\left\{i \in \{1,\ldots, M\}; \;d(X_{\theta_0},\er^d\setminus\Uu_{i})=\max_{1\leq m\leq M}d(X_{\theta_0},\er^d\setminus\Uu_{m})\right\},\\
\zeta_{1}  &=\inf\{t>  \theta_0;\;x_{t}\notin\Uu_{i_{1}}\}, \\
\beta^{1}_1 &=\inf\{t >  \theta_0;\;x_t \in \partial \Dd\}.
\end{align*}
We set $\theta_1 = \beta^{1}_{1}\wedge \zeta_1$ and
\begin{align*}
(X_{t},U_{t})  & =  (x_{t},u_{t}), \mbox{ for all } \theta_0\leq  t < \theta_1, \\
\mbox{and } \quad(X_{\theta_1},U_{\theta_1}) & = \left(x_{\theta_1},u_{\theta_1^{-} }
-2( u_{\theta_1^{-} }\cdot\nd(x_{\theta_1}))\nd(x_{\theta_1}) \ind_{\{X_{\theta_1}\in \partial \Dd\}}\right).
\end{align*}
We iterate the construction as follows: assume that we have constructed  the process $(X_t,U_t)$ on $[0,\theta_{n}\wedge T]$.  We  define
\begin{align}
i_{n+1}&=\min\left\{i \in \{1,\ldots, M\}; \;d(X_{\theta_n},\er^d\setminus\Uu_{i})=\max_{1\leq m\leq M}d(X_{\theta_n},\er^d\setminus\Uu_{m})\right\}, \label{def:i_n+1}\\
\zeta_{n+1}  &=\inf\{t > \theta_n;\;x_{t}\notin\Uu_{i_{n+1}}\}, \label{def:zeta_n+1}\\ \beta^{n+1}_{1} & = \inf\{t > \theta_n;\;x_t \in \partial \Dd\} \label{def:beta_n+1},
\end{align}
where now  $(x_t,u_t)$ denotes  the  solution of  \eqref{eq:freelangevin} starting at $(\theta_n,X_{\theta_n},U_{\theta_n})$.
We set
\begin{align}
\theta_{n+1} & = \beta^{n+1}_1\wedge \zeta_{n+1}, \label{def:theta_n}\\
(X_{t},U_{t})  & =  (x_{t},u_{t})^{\theta_n,X_{\theta_n},U_{\theta_n}},\mbox{ for all } \theta_{n}\leq  t < \theta_{n+1}, \label{def:process-stepn+1}\\
\mbox{and} \quad(X_{\theta_{n+1}},U_{\theta_{n+1}}) & = \left(x_{\theta_{n+1}},u_{\theta_{n+1}^{-} }
-2( u_{\theta_{n+1}^{-} }\cdot\nd(x_{\theta_{n+1}}))\nd(x_{\theta_{n+1}}) \ind_{\{X_{\theta_{n+1}}\in \partial \Dd\}} \right). \label{def:reflection-stepn+1}
\end{align}
Note that the index $n$ above is always greater or equal to  the number of times $(X_t, t\leq \theta_n)$ hits the boundary.
By construction $(X_{t},U_{t}) $ satisfies  \eqref{eq:LinearConfinedLangevin} and the sequence of hitting times $\{\tau_m;  m\geq 0\}$ is well defined as each $\tau_m$ corresponds to some $\beta^{n}_1$.
We conclude on the existence of a solution to \eqref{eq:LinearConfinedLangevin} with Lemma \ref{lem:AsymptoticsExitTimesOfStraightening} proved below.

The pathwise uniqueness result is consequence of the well-posedness of the hitting time sequence $\{\tau_m,m\geq 0\}$: consider $({X},{U})$ and $(\widetilde{X},\widetilde{U})$, two  solutions to~\eqref{eq:LinearConfinedLangevin} defined on the same probability space, endowed with the same Brownian motion. Let us denote by $\widetilde{\tau}_{1}$  the first hitting time of $\widetilde{X}$,  we observe that
$\widetilde{\tau}_{1}\wedge T =\tau_{1}\wedge T$ due to the continuity of $X$ and $\widetilde{X}$. It follows that
$U_{\tau_{1}\wedge T}=\widetilde{U}_{\tau_{1}\wedge T}$, so that
$(X,U)$ and $(\widetilde{X},\widetilde{U})$ are equal up to $\tau_{1}\wedge T$.
By induction, one checks that this assertion holds true up to $\tau_{n}\wedge T$ for all $n\in\nat$. As $\tau_{n}$ tends to $+\infty$ $\PP_{x_0,u_0}$-a.s., $(X,U)$ and $(\widetilde{X},\widetilde{U})$ are equal on $[0,T]$.

\end{proof}
\begin{lemma}\label{lem:AsymptoticsExitTimesOfStraightening} For any $(x_0,u_0) \in \okphasespace$,  $\PP_{x_0,u_0}$-a.s., the sequence $\{\theta_{n};n\in\mathbb{N}\}$ given in \eqref{def:theta_n} grows to infinity as $n$ tends to $+\infty$.
\end{lemma}

We also emphasize that, starting from $(x_0,u_0) \in \okphasespace$, $\PP_{x_0,u_0}$-almost surely any path of  the free process $(x_t,u_t)$  solution of \eqref{eq:freelangevin} never crosses  $\Sigma^0$. Indeed, the probability for the solution of \eqref{eq:freelangevin} to    cross $\Sigma^0$ is dominated by the probability that a piece of straightened path (starting in $\Dd\times\er^{d}$) crosses $(\er^{d-1}\times\{0\})^2$ which is nil by the McKean result.

\begin{proof}[Proof of Lemma \ref{lem:AsymptoticsExitTimesOfStraightening}]

We already know that the sequence $\{\beta_1^n;\,n\in\nat\}$ is well defined and grows to infinity. So we only have to prove that $\{\zeta_n;\,n\in\nat\}$ grows to infinity. We simplify the presentation of the proof by considering the sequence of stopping times  $\{\zeta_n;\,n\in\nat\}$ omitting the intercalation with $\{\beta_1^n;\,n\in\nat\}$:
\begin{align}
i_{n+1}&=\min\left\{i \in \{1,\ldots, M\}; \;d(x_{\zeta_n},\er^d\setminus\Uu_{i})=\max_{1\leq m\leq M}d(x_{\zeta_n},\er^d\setminus\Uu_{m})\right\},\label{def:i_n2}\\
\zeta_{n+1}  &=\inf\{t\geq \zeta_n\,;\,x_{t}\notin\Uu_{i_{n+1}}\},\;\; \zeta_0 =0 \label{def:zeta_n2}
\end{align}
where $((x_t,u_t);\,t\in[0,+\infty))$ is the solution of \eqref{eq:freelangevin} starting at $(0,x_0,u_0)$.

For any $i\in\{1,\ldots,M\}$ and any $x\in\partial{\Uu}_i$,
$\max_{j\not=i}d(x,\er^{d}\setminus\Uu_j)$ is strictly  positive and continuous with respect to $x\in\partial\Uu_i$. Since ${\Uu}_i$ is bounded for $1\leq i\leq M-1$ and $\partial{\cal
  U}_M\subset\bigcup_{i=1}^{M-1}\Uu_i$, the set $\partial\Uu_i$ is compact for any $i\in\{1,\ldots,M\}$. Hence, we can define the strictly positive constant $\gamma_0$ as the minimal distance  that allows the process $(x_t)$ to go from one $\Uu_i$ to another $\Uu_j$,
\begin{align*}
\gamma_0:=\min_{1\leq i\leq M}\:\inf_{x\in\partial{\Uu}_i}\:\max_{j\not=i}d(x,\er^d\setminus{\Uu}_j)>0.
\end{align*}
The idea is to prove that, almost surely, $\zeta_k-\zeta_{k-1}\geq \Tt$ infinitely often w.r.t. $k\geq 1$ for $\Tt$ small enough. We fix a constant $\Tt>0$ and  we consider $A_k = \{ \sup_{0\leq t \leq \Tt}|x_{\zeta_{k-1} + t} - x_{\zeta_{k-1}}| \leq \gamma_0 \}$, such that
$$
\PP(A_k^c\mid{\cal F}_{\zeta_{k-1}}) \leq \frac{\Tt}{\gamma_0}\EE \left[\sup_{0\leq t \leq T} | u_t|\right].
$$
Since $A_k\in{\cal F}_{\zeta_k}$ for all $k\geq 1$,  for all $m<n$,
\begin{align*}
\PP\Big(\bigcap_{m\leq k\leq n}A_k^c\Big)
&=\EE\Big[\ind_{\{\bigcap_{m\leq k\leq n-1}A_k^c\}}
\PP(A_n^c\mid{\cal F}_{\zeta_{n-1}})\Big]\\
& \leq\PP\Big(\bigcap_{m\leq k \leq n-1}A_k^c\Big)\left(\frac{\Tt}{\gamma_0}\EE \left[\sup_{0\leq t \leq T} | u_t|\right]\right)\leq\ldots\leq\left( \frac{\Tt}{\gamma_0}\EE \left[\sup_{0\leq t \leq T} | u_t|\right]\right)^{n-m+1}.
\end{align*}
Therefore, choosing $\Tt$ such that $\frac{\Tt}{\gamma_0}\EE [\sup_{0\leq t \leq T} | u_0 + \sigma W_t|]<1$, it follows that
$$
\lim_{n\rightarrow+\infty}\PP\Big(\bigcap_{m\leq k\leq n}A_k^c\Big)=0,\quad\forall m\geq 1.
$$
This entails that the events $A_k$ occur infinitely often $\PP$-a.s. We thus found $\Tt>0$ such that the events $\{\zeta_k-\zeta_{k-1}\geq \Tt\}$  occur a.s. infinitely often.
\end{proof}

The following lemma is used  in Section \ref{sec:LawReconstruction} to achieve  the construction of the nonlinear process.
\begin{lemma}\label{lem:abs-cont-wrt-surface-mesure} Let $(X,U)$ be the confined Langevin process solution to \eqref{eq:LinearConfinedLangevin}  and $\{\tau_{n};  n\in \nat\}$ be the related sequence of hitting times defined in Theorem \ref{thm:LinearSolutionByStraightening}. Then, for all $(x_0,u_0) \in \okphasespacebis$,  for all integer $K>1$, the measure  $\sum_{n=0}^{K}\PP_{x_0,u_0}(\tau_{n}\in dt,X_{\tau_{n}}\in dx,U_{\tau_{n}}\in du)$ is absolutely continuous w.r.t the  measure $\lambda_{\Sigma_{T}}(dt,dx,du) =dt\otimes d\sigma_{\partial \Dd}(x)\otimes du$.
\end{lemma}
\begin{proof}
Owing to the strong Markov property of $(X,U)$, it is sufficient to prove  that for all $(x_0,u_0)\in\Sigma^{-}$, the probability $\PP_{x_0,u_0}\circ(\tau_{1},X_{\tau_{1}},U_{\tau_{1}})^{-1}$ is absolutely continuous w.r.t. $\lambda_{\Sigma_{T}}$. Using the same notation than in the proof of Theorem \ref{thm:LinearSolutionByStraightening}, for $(x_0,u_0)$ corresponding to some $(x_{\zeta_n},u_{\zeta_n})$ in the iterative construction (\eqref{def:i_n+1}-\eqref{def:reflection-stepn+1}),  it is sufficient to prove that  $\PP_{x_0,u_0}\circ(\beta^{1}_{1},x_{\beta^{1}_{1}},u_{\beta^{1}_{1}})^{-1}$ and $\lambda_{\Sigma_{T}}$ are equivalent, where $((x_t,u_t);\, t\in[0,T])$ is the solution of \eqref{eq:freelangevin}. Let $i_1$ be the index of the subset $\Uu_{i_1}$ such that $x^{x_0,u_0}_{\beta^{1}_{1}} \in \Uu_{i_1}\cap \partial\Dd$, $\psi_{i_1}(x_t) = Y_t$,
$(\nabla\psi_{i_1}(x_t)\cdot u_{t})=V_{t}$ and $\beta^{1}_{1}=\inf\{t>0; Y^{(d)}_{t}=0\}$ (with $((Y_t,V_t);\,t\in[0,T])$ solution to \eqref{eq:StraightedLangevin}). Then for any measurable test function $f$,
\begin{equation*}
\EE_{\PP_{x_0,u_0}}\left[f(\beta^{1}_{1},x_{\beta^{1}_{1}},u_{\beta^{1}_{1}})\right] =
\EE_{\PP_{y_0,v_0}}\left[f(\beta^{1}_{1},\psi^{-1}_{i_1}(Y_{\beta^{1}_{1}}),(\nabla\psi^{-1}_{i_{1}}(Y_{\beta^{1}_{1}})\cdot V_{\beta^{1}_{1}}))\right],
\end{equation*}
where $(y_{0},v_{0}):=(\psi^{-1}_{i_{1}}(x_{0}),(\nabla\psi^{-1}_{i_{1}}\cdot u_{0}))$. At this point, and owing to the equivalence between the laws of $((Y_{t},V_{t});\,t\in[0,T])$ and  $((\Yy_{t},\Vv_{t});\,t\in[0,T])$, solution to \eqref{eq:equivLangevinStraightened}, we are reduce to prove that the law of $(\beta^{1}_{1},\psi^{-1}_{i}(\Yy_{\beta^{1}_{1}}),(\nabla\psi^{-1}_{i_{1}}(\Yy_{\beta^{1}_{1}})\cdot V_{\beta^{1}_{1}}))$ is absolutely continuous w.r.t. $\lambda_{\Sigma_{T}}$. Let us first recall that the joint law of $(\beta^{1}_{1},\Yy^{(d)}_{\beta^{1}_{1}},\Vv^{(d)}_{\beta^{1}_{1}})$ is explicitly known (see \cite[Theorem 1]{Lachal1997}) and is absolutely continuous w.r.t. the product measure $dt\otimes du^{(d)}$. Furthermore,  $\beta^{1}_{1}$ is independent of the $(d-1)$-first components of  $(\Yy_t,\Vv_{t}, t\in[0,T])$. Hence we remark that the law of $(\beta^{1}_{1},\Yy_{\beta^{1}_{1}},\Vv_{\beta^{1}_{1}})$ is absolutely continuous w.r.t. $dt\otimes dy'\otimes du$, where $dy'$ is the surface measure of $\partial (\er^{d-1}\times \er^+)$.

Let us next recall the characterization of the surface measure $\sigma_{\partial\Dd}$ (see e.g  \cite[Chapter 5]{Adams-1975}). As $\partial \Dd$ is $\Cc^3$, it can be (locally) represented as the graph of a $\Cc^3$ function:  for all $x\in\partial\Dd$, there exists an open neighborhood  $\Uu_x\subset \er^d$ of $x$ and a $\Cc^3_b(\er^{d-1})$ function $\phi_x$, such that for all $y\in\partial \Dd  \cap \overline{\Uu_x}$, $y^{(d)} = \phi_x(y^{(1)},\ldots,y^{(d-1)})$  and
 $\Dd \cap \Uu_x = \{ y = (y',y^{(d)}) \in \Uu_x, \text{ s.t } y^{(d)} < \phi_x(y')\}$. Hence, for all $\Cc^{\infty}$-function $g$ with compact support in  $\partial \Dd  \cap \overline{\Uu_x}$, we have that
\[\int_{\partial \Dd} g(x) d\sigma_{\partial\Dd}(x)  =  \int_{\er^{d-1} } g((y',\phi_x(y'))) \sqrt{ 1 + | \nabla \phi_x(y')|^2} dy'.\]
Owing to this characterization and the preceding remark on the law of $(\beta^{1}_{1},\Yy_{\beta^{1}_{1}},\Vv_{\beta^{1}_{1}})$, we conclude on the result.
\end{proof}

\subsection{On the semigroup of the confined Langevin process}
In this section we investigate some estimates related to the semigroup associated  to the solution of the SDE  \eqref{eq:LinearConfinedLangevin}; namely, for a nonnegative test function $\psi \in \Cc^{\infty}_c(\Dd\times\er^d)$, for all $(x,u)\in\okphasespacebis$, we define
\begin{equation}
\label{eq:EqMildKernel}
\Gamma^{\psi}(t,x,u):=\EE_{\PP}\left[\psi(X^{x,u}_{t},U^{x,u}_{t})\right],
\end{equation}
where $((X^{x,u}_{t},U^{x,u}_{t});\,t\in[0,T])$ is the solution of \eqref{eq:LinearConfinedLangevin} starting from $(0,x,u)$ and  $((X^{s,x,u}_{t},U^{s,x,u}_{t});\,t\in[0,T])$ is the solution of \eqref{eq:LinearConfinedLangevin} starting from $(s,x,u)$.

As in \cite{jabir-11a}, this semigroup is of particular interest for the construction of the solution to the nonlinear equation \eqref{eq:NonlinearConfinedLangevin}. Due to the pathwise uniqueness of the confined Langevin process, one has that, for all $0\leq s\leq t\leq T$,
\begin{equation}
\Gamma^{\psi}(t-s,x,u)=\EE_{\PP}\left[\psi(X^{s,x,u}_{t},U^{s,x,u}_{t})\right],
\end{equation}
so that the  estimates hereafter can be extended to the semigroup transitions of the process.

We consider also the semigroup related to the stopped process:
\begin{align*}
\Gamma^{\psi}_{n}(t,x,u)=\EE_{\PP}\left[\psi(X^{x,u}_{t\wedge\tau^{x,u}_{n}},U^{x,u}_{t\wedge\tau^{x,u}_{n}})\right],
\end{align*}
where $\{\tau^{x,u}_{n};n\in\nat\}$ is  the sequence of hitting times defined in Theorem \ref{thm:LinearSolutionByStraightening} and $\Gamma^{\psi}_{0}(t,x,u)=\psi(x,u)$.
The estimates on $\{\Gamma^{\psi}_{n};\,n\geq 1\}$ and $\Gamma^{\psi}$ rely on the following PDE result, the proof of which is postponed in the next Subsection \ref{sec:PDEStudy}.
\begin{theorem}\label{thm:L2VFPDirichlet}
Assume \hyplang. Given two nonnegative functions $f_{0}\in L^{2}(\Dd\times\er^{d})\cap\Cc_{b}(\Dd\times\er^{d})$ and
$q\in L^{2}(\Sigma^{+}_{T})\cap\Cc_{b}(\Sigma^{+}_{T})$, there exists a unique nonnegative function $f\in \Cc^{1,1,2}_{b}(Q_{T})\cap \Cc((0,T]\times (\overline{\Dd}\times \er^d\setminus\Sigma^{0}))\cap L^{2}((0,T)\times\Dd;H^{1}(\er^{d}))$ which is a solution to
\begin{equation}\label{eq:VFPDirichlet}
\left\{
\begin{aligned}
&\partial_{t} f(t,x,u) - (u\cdot\nabla_x f(t,x,u))-\frac{\sigma^{2}}{2}\triangle_u f(t,x,u) = 0,~\textrm{ for all}~(t,x,u)\in Q_{T},\\
&f(0,x,u) = f_{0}(x,u),~\textrm{ for all}~(x,u)\in\Dd \times\er^{d},\\
&f(t,x,u) = q(t,x,u),~\textrm{ for all}~
(t,x,u)\in\Sigma^{+}_{T}.
\end{aligned}
\right.
\end{equation}
In addition, for $(x^{x,u}_{t},u^{x,u}_{t};\,t\in[0,T])$ solution to \eqref{eq:freelangevin} starting from $(x,u)\in\Dd\times\er^{d}$ at $t=0$ and $\beta^{x,u}:=\inf\{t>0\,;\,x^{x,u}_{t}\in\partial\Dd\}$, we have
\begin{equation}
\label{eq:FeynmanKacInterpretation}
f(t,x,u)=\EE_{\PP}
\left[f_0(x^{x,u}_{t},u^{x,u}_{t})
\ind_{\displaystyle\{t\leq\beta^{x,u}\}}\right]
+\EE_{\PP}\left[q(t-\beta^{x,u},
x^{x,u}_{\beta^{x,u}},
u^{x,u}_{\beta^{x,u}})
\ind_{\displaystyle\{t>\beta^{x,u}\}}\right].
\end{equation}
Furthermore, for all $t\in(0,T)$, $f$ satisfies the energy equality:
\begin{equation}\label{estim:GeneralEnergyEquality}
\Vert f(t)\Vert^{2}_{L^{2}(\Dd\times\er^{d})}+\Vert f\Vert^{2}_{L^{2}(\Sigma^{-}_{t})}+\sigma^{2}\Vert\nabla_{u}f\Vert^{2}_{L^{2}(Q_{t})}=
\Vert f_0\Vert^{2}_{L^{2}(\Dd\times\er^{d})}+\Vert q\Vert^{2}_{L^{2}(\Sigma^{+}_{t})},
\end{equation}
and the $L^{p}$-estimate: for $1\leq p<+\infty$,
\begin{equation}
\label{estim:GeneralLpEquality}
\Vert f(t)\Vert^{p}_{L^{p}(\Dd\times\er^{d})}+
\Vert f\Vert^{p}_{L^{p}(\Sigma^{-}_{t})}\leq \Vert f_{0}\Vert^{p}_{L^{p}(\Dd\times\er^{d})}+\Vert q\Vert^{p}_{L^{p}(\Sigma^{+}_{t})}.
\end{equation}
\end{theorem}
From Theorem \ref{thm:L2VFPDirichlet}, we deduce the following result for $\{\Gamma^{\psi}_{n}, n\geq 1\}$:
\begin{corollary}\label{coro:KernelSmoothness}
Assume \hyplang. Then, for all nonnegative $\psi\in\Cc_{c}(\Dd\times\er^d)$ and all $n\in \nat^*$, $\Gamma^{\psi}_{n}$  is a nonnegative function in $\Cc^{1,1,2}_{b}(Q_{T})\cap \Cc(\overline{Q_{T}}\setminus\Sigma^{0})$ and satisfies the PDE
\begin{equation}\label{eq:backwardVFP}
\left\{\begin{aligned}
&\partial_{t}\Gamma^{\psi}_{n}(t,x,u)-(u\cdot \nabla_{x}\Gamma^{\psi}_{n}(t,x,u))-\frac{\sigma^{2}}{2}\triangle_{u}\Gamma^{\psi}_{n}(t,x,u)=0,~\textrm{ for all}~(t,x,u) \in Q_{T},\\
&\Gamma^{\psi}_{n}( 0,x,u)=\psi(x,u),~\textrm{ for all}~(x,u)\in\Dd\times\er^{d},\\
&\Gamma^{\psi}_{n}(t,x,u)=\Gamma^{\psi}_{n-1}(t,x,u-2(u\cdot\nd(x))\nd(x)),~
\textrm{ for all}~(t,x,u)\in\Sigma^{+}_{T}.
\end{aligned}\right.
\end{equation}
In addition, the set $\{\Gamma^{\psi}_{n},n\geq 1, \Gamma^{\psi}_{0} = \psi\}$ belongs to $L^{2}((0,T)\times\Dd;H^1(\er^d))$ and satisfies the energy equality
\begin{align}\label{ineq:ApproxSobolevEstimate_L2}
\Vert \Gamma^{\psi}_{n}(t)\Vert^{2}_{L^{2}(\Dd\times\er^{d})}+\sigma^{2}\Vert\nabla_{u}\Gamma^{\psi}_{n}\Vert^{2}_{L^{2}(Q_{t})}+\Vert \Gamma^{\psi}_{n}\Vert^{2}_{L^{2}(\Sigma^{-}_{t})}=
\Vert \psi\Vert^{2}_{L^{2}(\Dd\times\er^{d})}+\Vert \Gamma^{\psi}_{n-1}\Vert^{2}_{L^{2}(\Sigma^{-}_{t})}.
\end{align}
\end{corollary}
\begin{proof}
For $n> 1$, let us assume that $\Gamma^{\psi}_{n-1}\in \Cc(\overline{Q_{T}}\setminus\Sigma^{0})$ with
$\Gamma^{\psi}_{n-1}|_{\Sigma^{-}_{T}}\in L^2(\Sigma^{-}_{T})$. Since  $x\mapsto n_{\Dd}(x)$ is continuous on $\partial\Dd$,
the function $(t,x,u)\mapsto \Gamma^{\psi}_{n-1}(t,x,u-2(u\cdot n_{\Dd}(x))n_{\Dd}(x))$ is in $\Cc_{b}(\Sigma^{+}_{T})$. Furthermore, $\Gamma^{\psi}_{n-1}(t,x,u-2(u\cdot n_{\Dd}(x))n_{\Dd}(x))|_{\Sigma^{+}_{T}}$  is in $L^{2}(\Sigma^{+}_{T})$ since, by using the change of variables $u\mapsto \widehat{u} := u-2(u\cdot n_{\Dd}(x))n_{\Dd}(x)$ for fixed $x\in\partial\Dd$,  we have
\begin{equation}\label{prfst:KernelSmoothness}
\begin{aligned}
&\int_{\Sigma^{+}_{T}}|(u\cdot n_{\Dd}(x))|\left(\Gamma^{\psi}_{n-1}(t,x,u-2(u\cdot n_{\Dd}(x))n_{\Dd}(x))\right)^{2}d\lambda_{\Sigma_{T}}(t,x,u)\\
&=\int_{\Sigma^{-}_{T}}|(u\cdot n_{\Dd}(x))|\left(\Gamma^{\psi}_{n-1}(t,x,u)\right)^{2}d\lambda_{\Sigma_{T}}(t,x,u)=\Vert\Gamma^{\psi}_{n-1}
\Vert^{2}_{L^{2}(\Sigma^{-}_{T})}<+\infty.
\end{aligned}
\end{equation}

From the strong Markov property of the solution of \eqref{eq:LinearConfinedLangevin}, we get that for all $(x,u)\in (\Dd\times \er^d)\cup\Sigma^+$,
\begin{align}\label{eq:Markov}
\EE [\psi(X^{x,u}_{t\wedge\tau_n^{x,u}},U^{x,u}_{t\wedge\tau_n^{x,u}}) \ind_{\{\tau_1^{x,u} <  t\}} ]  = \EE [\Gamma^{\psi}_{n-1}(t-\tau_1^{x,u},X^{x,u}_{\tau_1^{x,u}}, U^{x,u}_{\tau_1^{x,u}}) \ind_{\{\tau_1^{x,u} <  t\}}].
\end{align}
Considering a sequence $(x_m,u_m, m\in \nat)$ in $\Dd\times\er^d$ converging to $(x,u)\in \Sigma^+$, and $t>0$, we have, for $m$ large enough
\begin{align*}
({\tau_1^{x_m,u_m}},X^{x_m,u_m}_t,U^{x_m,u_m}_t, t< {\tau_1^{x_m,u_m}}) & = (\beta^{x_m,u_m}, x^{x_m,u_m}_t, u^{x_m,u_m}_t, t < \beta^{x_m,u_m})\\
(X^{x_m,u_m}_{\tau_1^{x_m,u_m}},U^{x_m,u_m}_{\tau_1^{x_m,u_m}} ) & = (x^{x_m,u_m}_{\beta^{x_m,u_m}}, \widehat{u}^{x_m,u_m}_{\beta^{x_m,u_m}}).
\end{align*}
Hence, from the continuity of $(y,v)\mapsto (\beta^{y,v},x^{y,v}_{t},u^{y,v}_{t})$ proved with Proposition
\ref{prop:ContU2IrrelevantBound},
\begin{align*}
\lim_{m\rightarrow +\infty } ( \tau_1^{x_m,u_m},X^{x_m,u_m}_{t\wedge\tau_1^{x_m,u_m}} , U^{x_m,u_m}_{t\wedge\tau_1^{x_m,u_m}}) = (0, x, u - 2(u\cdot \nd(x))\nd(x) ).
\end{align*}
The right-hand side of \eqref{eq:Markov} is then continuous on $(\Dd\times \er^d)\cup\Sigma^+$, as well as
\begin{align*}
\EE [\psi(X^{x,u}_{t\wedge\tau_n^{x,u}},U^{x,u}_{t\wedge\tau_n^{x,u}}) \ind_{\{\tau_1^{x,u} \geq t\}} ]  = \EE [\psi(X^{x,u}_{t\wedge\tau_1^{x,u}},U^{x,u}_{t\wedge\tau_1^{x,u}}) \ind_{\{\tau_1^{x,u} \geq t\}}].
\end{align*}
Moreover,  for $(t,x,u) \in \Sigma^+_T$,
\begin{align*}
\Gamma^{\psi}_{n}(t,x,u) & = \lim_{m\rightarrow +\infty} \left\{ \EE [\psi(X^{x_m,u_m}_{t\wedge\tau_n^{x_m,u_m}},U^{x_m,u_m}_{t\wedge\tau_n^{x_m,u_m}}) \ind_{\{\tau_1^{x_m,u_m} < t\}} ] +
\EE [\psi(X^{x_m,u_m}_{t\wedge\tau_n^{x_m,u_m}},U^{x_m,u_m}_{t\wedge\tau_n^{x_m,u_m}}) \ind_{\{\tau_1^{x_m,u_m} \geq t\}} ]\right\} \\
& = \Gamma^{\psi}_{n-1}(t,x,u-2(u\cdot \nd(x))\nd(x)).
\end{align*}

Now Theorem \ref{thm:L2VFPDirichlet} ensures that there exists a classical solution $f_{n}$ to \eqref{eq:VFPDirichlet} for $f_{0}=\psi$ and $q(t,x,u)=\Gamma^{\psi}_{n-1}(t,x,u-2(u\cdot n_{\Dd}(x))n_{\Dd}(x))$ on $\Sigma^{+}_{T}$. According to \eqref{eq:FeynmanKacInterpretation}, we have, for $(t,x,u)\in Q_T$
\begin{align*}
f_{n}(t,x,u) &=\EE_{\PP}\left[\psi(x^{x,u}_{t},u^{x,u}_{t})\ind_{\displaystyle\{t\leq\beta^{x,u}\}}\right]\\
&\quad +\EE_{\PP}\left[\Gamma^{\psi}_{n-1}\left(t-\beta^{x,u},
x^{x,u}_{\beta^{x,u}},
u^{x,u}_{\beta^{x,u}}-2(u^{x,u}_{\beta^{x,u}}\cdot n_{\Dd}(x^{x,u}_{\beta^{x,u}}))n_{\Dd}(x^{x,u}_{\beta^{x,u}})\right)
\ind_{\displaystyle\{t>\beta^{x,u}\}}\right].
\end{align*}
One can observe that
\begin{equation*}
\EE_{\PP}\left[\psi(x^{x,u}_{t},u^{x,u}_{t})\ind_{\displaystyle\{t\leq\beta^{x,u}\}}\right]
=\EE_{\PP}\left[\psi(X^{x,u}_{t},U^{x,u}_{t})\ind_{\displaystyle\{t\leq\tau_{1}^{x,u}\}}\right]
=\EE_{\PP}\left[\psi(X^{x,u}_{t\wedge \tau^{x,u}_{n}},U^{x,u}_{t\wedge \tau^{x,u}_{n}})\ind_{\displaystyle\{t\leq\tau_{1}^{x,u}\}}\right]
\end{equation*}
and that
\begin{align*}
&\EE_{\PP}\left[\Gamma^{\psi}_{n-1}\left(t-\beta^{x,u},
x^{x,u}_{\beta^{x,u}},
u^{x,u}_{\beta^{x,u}}-2(u^{x,u}_{\beta^{x,u}}\cdot n_{\Dd}(x^{x,u}_{\beta^{x,u}}))n_{\Dd}(x^{x,u}_{\beta^{x,u}})\right)
\ind_{\displaystyle\{t>\beta^{x,u}\}}\right]\\
&=\EE_{\PP}\left[\Gamma^{\psi}_{n-1}(t-\tau_{1}^{x,u},
X^{x,u}_{\tau_{1}^{x,u}},
U^{x,u}_{\tau_{1}^{x,u}})\ind_{\displaystyle\{t>\tau_{1}^{x,u}\}}\right]\\
&=\EE_{\PP}\left[\psi(X^{x,u}_{t\wedge \tau^{x,u}_{n}},U^{x,u}_{t\wedge \tau^{x,u}_{n}})\ind_{\displaystyle\{t>\tau_{1}^{x,u}\}}\right],
\end{align*}
where the second equality follows from the strong Markov property of $(X^{x,u}_{t},U^{x,u}_{t})$. Therefore
\begin{equation*}
f_{n}(t,x,u) =\EE_{\PP}\left[\psi(X^{x,u}_{t\wedge \tau^{x,u}_{n}},U^{x,u}_{t\wedge \tau^{x,u}_{n}})\right]=\Gamma^{\psi}_{n}(t,x,u),
\end{equation*}
from which we deduce that $\Gamma^{\psi}_{n}\in\Cc^{1,1,2}_{b}(Q_{T})\cap \Cc(\overline{Q_{T}}\setminus\Sigma^{0}_{T})\cap L^{2}((0,T)\times\Dd;H^{1}(\er^{d}))$ is a solution to \eqref{eq:backwardVFP} with $\Gamma^{\psi}_{n}|_{\Sigma^{-}_{T}}\in L^{2}(\Sigma^{-}_{T})$. Moreover, according to \eqref{prfst:KernelSmoothness},  for all $t\in(0,T)$,
\begin{align*}
& \Vert \Gamma^{\psi}_{n}(t)\Vert^{2}_{L^{2}(\Dd\times\er^{d})}+
\Vert \Gamma^{\psi}_{n}\Vert^{2}_{L^{2}(\Sigma^{-}_{t})}+\sigma^{2}\Vert \nabla_{u}\Gamma^{\psi}_{n}\Vert^{2}_{L^{2}(Q_{t})}\\
&= \Vert \psi\Vert^{2}_{L^{2}(\Dd\times\er^{d})}+\Vert q\Vert^{2}_{L^{2}(\Sigma^{+}_{t})}=\Vert \psi\Vert^{2}_{L^{2}(\Dd\times\er^{d})}+\Vert \Gamma^{\psi}_{n-1}\Vert^{2}_{L^{2}(\Sigma^{-}_{t})}.
\end{align*}
For $n=1$, setting $f_{0}=\psi$ and $q=\psi|_{\Sigma^{+}_{T}}=0$ (since $\psi$ has its support in the interior of $\Dd\times\er^{d}$), one can check that $\Gamma^{\psi}_{1}\in \Cc^{1,1,2}_{b}(Q_{T})\cap \Cc(\overline{Q_{T}}\setminus\Sigma^{0}_{T})\cap L^{2}((0,T)\times\Dd;H^{1}(\er^{d}))$ satisfies \eqref{eq:backwardVFP} and \eqref{ineq:ApproxSobolevEstimate_L2}.
By induction, we end the proof.
\end{proof}
\begin{corollary}\label{coro:KernelEstimate}
Assume  \hyplang. For all nonnegative $\psi\in\Cc_{c}(\Dd\times\er^{d})$,  $\Gamma^{\psi}$ is a nonnegative function that belongs to $L^{2}((0,T)\times\Dd;H^1(\er^d))$  and satisfies  the energy equality:
\begin{align}\label{ineq:SobolevEstimate_L2}
\Vert \Gamma^{\psi}(t)\Vert^{2}_{L^{2}(\Dd\times\er^{d})}+\sigma^{2}\Vert\nabla_{u}\Gamma^{\psi}\Vert^{2}_{L^{2}(Q_{t})}=
\Vert \psi\Vert^{2}_{L^{2}(\Dd\times\er^{d})},~\forall\,t\in(0,T).
\end{align}
Furthermore, $\Gamma^{\psi}(t)$ is solution in the sense of distributions of
\begin{equation}\label{eq:backwardVFP-bis}
\left\{\begin{aligned}
&\partial_{t}\Gamma^{\psi}-(u\cdot \nabla_{x}\Gamma^{\psi})-\frac{\sigma^{2}}{2}\triangle_{u}\Gamma^{\psi}=0,~\mbox{on }Q_{T},\\
&\Gamma^{\psi}(0,x,u)=\psi(x,u),~\mbox{on }\Dd\times\er^{d},\\
&\Gamma^{\psi}(t,x,u)=\Gamma^{\psi}(t,x,u-2(u\cdot n_{\Dd}(x))n_{\Dd}(x)),~\mbox{on }\Sigma^{+}_{T}.
\end{aligned}\right.
\end{equation}
\end{corollary}
\begin{proof}
We first observe that since $\psi|_{\partial\Dd\times\er^{d}}=0$,
\begin{equation*}
\Gamma^{\psi}_{n}(t,x,u)=\EE_{\PP}\left[\psi(X^{x,u}_{t\wedge\tau^{x,u}_{n}},U^{x,u}_{t\wedge\tau^{x,u}_{n}})\right]=\EE_{\PP}\left[\psi(X^{x,u}_{t},
U^{x,u}_{t})\ind_{\{\tau^{x,u}_{n}\geq t\}}\right].
\end{equation*}
Next, there exists a nonnegative function $\beta\in L^{2}(\er)$ such that $\beta(|u|)=1$ on the support of $\psi$ and $\psi\leq C \beta(|u|)$, with  $C:=\sup_{(x,u)\in\Dd\times\er^d} \psi(x,u)$. Then
\begin{equation*}
0 \leq \Gamma^{\psi}_{n}(t,x,u) \leq C\EE_{\PP}\left[\beta(|U^{x,u}_{t}|)
\ind_{\{\tau^{x,u}_{n}\geq t\}}\right].
\end{equation*}
As $\EE_{\PP}[\beta(|U^{x,u}_{t}|)]$ is equal to the convolution product $(G(\sigma t)*\beta)(|u|)$,   where $G$ denotes the heat kernel on $\er^d$, we obtain
\begin{equation}\label{estim:KernelMaxwellBound}
0 \leq \Gamma^{\psi}_{n}(t,x,u) \leq C (G(\sigma t)*\beta)(|u|),~\mbox{on }Q_{T}.
\end{equation}
Owing to the continuity of $\Gamma^{\psi}_{n}$,  from the interior of $Q_{T}$ to its boundary, \eqref{estim:KernelMaxwellBound} still holds true along $\Sigma^{\pm}_{T}$.

Let us now recall that, for a.e. $(x,u)\in\okphasespacebis$, $\PP_{(x,u)}$-a.s. $\tau_{n}$ grows to $\infty$ as $n$ increases, and then
\begin{equation}\label{eq:KernelsConvergence}
\lim_{n\rightarrow +\infty}\Gamma^{\psi}_{n}(t,x,u)=\Gamma^{\psi}(t,x,u), \mbox{ for a.e. }(t,x,u)\in Q_{T},\,
\lambda_{\Sigma_{T}}\mbox{-a.e. }(t,x,u)\in \Sigma_{T}\setminus\Sigma^{0}_{T}.
\end{equation}
Indeed,
\begin{align*}
\left|\Gamma^{\psi}_{n}(t,x,u)-\Gamma^{\psi}(t,x,u)\right|
=\left|\EE_{\PP}\left[\psi(X^{x,u}_{t},U^{x,u}_{t})\ind_{\{\tau^{x,u}_{n}\leq t\}}\right]\right|
\leq \Vert \psi\Vert_{\infty}\PP(\tau^{x,u}_{n}\leq t).
\end{align*}
In particular, \eqref{estim:KernelMaxwellBound}  is also true for $\Gamma^{\psi}$.
We conclude by the Lebesgue Dominated Convergence Theorem that $\Gamma^{\psi}_{n}$ converges to $\Gamma^{\psi}$ in $L^2(\Dd\times \er^d)$.

Next we deduce that the norms involving $\Gamma^{\psi}_{n}$ in the left-hand side of  \eqref{ineq:ApproxSobolevEstimate_L2}  are finite for all $t$, uniformly in $n$ (as the right-hand side of \eqref{ineq:ApproxSobolevEstimate_L2}  is bounded uniformly in $n$ by the Maxwellian bound \eqref{estim:KernelMaxwellBound}).
Therefore, the estimate \eqref{ineq:ApproxSobolevEstimate_L2} is  also true for $\Gamma^{\psi}$ (see e.g. \cite{brezis-11}), and $\Gamma^{\psi}_{n}$ converges to $\Gamma^{\psi}$ in $L^{2}((0,T)\times\Dd;H^1(\er^d))$.
\end{proof}
\begin{corollary}[$L^{p}$-estimates]\label{coro:KernelLpEstimate}
Given $\psi\in\Cc_{c}(\Dd\times\er^d)$  nonnegative, the kernels $\{\Gamma_{n}^{\psi}, n\geq 1\}$ and $\Gamma^{\psi}$  satisfy for all $p\in [1,+\infty)$
\begin{align}
&\Vert \Gamma^{\psi}_{n}(t)\Vert^{p}_{L^{p}(\Dd\times\er^{d})}+
\Vert \Gamma^{\psi}_{n}\Vert^{p}_{L^{p}(\Sigma^{-}_{t})}
\leq \Vert \psi\Vert^{p}_{L^{p}(\Dd\times\er^{d})}+\Vert \Gamma^{\psi}_{n-1}\Vert^{p}_{L^{p}(\Sigma^{-}_{t})},\quad\forall\,t\in(0,T),\label{ineq:ApproxLpConservation}\\
&\Vert \Gamma^{\psi}(t)\Vert^{p}_{L^{p}(\Dd\times\er^{d})}\leq
\Vert \psi\Vert^{p}_{L^{p}(\Dd\times\er^{d})},\quad\forall\,t\in(0,T)\label{ineq:LpConservation}.
\end{align}
\end{corollary}
\begin{proof}
Applying estimate \eqref{estim:GeneralLpEquality} to  the solution of \eqref{eq:backwardVFP}, it follows that for all $t\in(0,T)$,
\begin{equation*}
\begin{aligned}
\Vert \Gamma^{\psi}_{n}(t)\Vert^{p}_{L^{p}(\Dd\times\er^{d})}+
\Vert \Gamma^{\psi}_{n}\Vert^{p}_{L^{p}(\Sigma^{-}_{t})}
\leq \Vert \psi\Vert^{p}_{L^{p}(\Dd\times\er^{d})}+\Vert \Gamma^{\psi}_{n}\Vert^{p}_{L^{p}(\Sigma^{+}_{t})}.
\end{aligned}
\end{equation*}
Since $\Vert \Gamma^{\psi}_{n}\Vert^{p}_{L^{p}(\Sigma^{+}_{t})}=\Vert \Gamma^{\psi}_{n-1}\Vert^{p}_{L^{p}(\Sigma^{-}_{t})}$ we deduce \eqref{ineq:ApproxLpConservation}.

Using the convergence of $\Gamma^{\psi}_{n}$ to $\Gamma^{\psi}$ and the uniform bounds  \eqref{estim:KernelMaxwellBound} on $\Sigma_T^-$, we also deduce \eqref{ineq:LpConservation}.
\end{proof}
\subsection{On the boundary value problem \eqref{eq:VFPDirichlet}} \label{sec:PDEStudy}
In this section, we prove Theorem \ref{thm:L2VFPDirichlet}. We consider the inputs $(f_0,q)$ and assume the following
\[
\mbox{\hypbarriere:  }\quad\quad
f_{0}\in L^{2}(\Dd\times\er^{d})\cap\Cc_{b}(\Dd\times\er^{d}) \mbox{ and }
q\in L^{2}(\Sigma^{+}_{T})\cap\Cc_{b}(\Sigma^{+}_{T})\mbox{ are nonnegative functions.}  \]

The main difficulty in the well-posedness of the boundary value problem \eqref{eq:VFPDirichlet} lies in the degeneracy of the diffusion operator and in the fact that we want to obtain the continuity of $f$ up to and along $\Sigma_{T}\setminus\Sigma^{0}_{T}$.  Such a problem has been addressed in Fichera \cite{Fich-1960} for second order differential operators of the form
\begin{equation*}
\Ll(f)(z)=\mbox{Trace}(a(z)\nabla^{2} f(z))+(b(z)\cdot \nabla f(z))+c(z)f(z)-h(z),~z\in\er^N
\end{equation*}
 where $\VVV$ is some smooth bounded open domain of $\er^N$ and, for all $z\in\VVV$,  $a(z)$ is only assumed to be a positive semi-definite matrix,  that is $(\xi \cdot a(z)\xi)\geq 0$, for all $\xi\in\er^{N}$.  Consider PDE of the form
\begin{equation}\label{elliptic_case:interior}
\Ll(f)=0,\,\,\mbox{on } \VVV
\end{equation}
submitted to some Dirichlet boundary condition. Denoting by $\nu(z)$ the unit outward normal vector to  $\partial\VVV$, the boundary $\partial\VVV$ may be split into four parts: the so-called \textit{non-characteristic part} $\Sigma_{3}:=\{z\in\partial \VVV;~(\nu(z)\cdot a(z)\nu(z))> 0\}$,
the {\it relevant} part $\Sigma_{2}:=\{z\in\partial \VVV/\Sigma_{3};~(b(z)\cdot \nu(z)) +\mbox{Trace}(a(z)\nabla\nu(z))> 0\}$,
the {\it irrelevant} part
 $\Sigma_{1}:=\{z\in\partial \VVV/\Sigma_{3};~(b(z)\cdot \nu(z)) +\mbox{Trace}(a(z)\nabla\nu(z)) < 0\}$
 and the {\it sticking}  part
 $\Sigma_{0}:=\{z\in\partial \VVV/\Sigma_{3};~(b(z)\cdot \nu(z))+\mbox{Trace}(a(z)\nabla\nu(z))=0 \}$. The term
\textit{relevant} refers to the boundary part where the boundary condition has to be specified:
\begin{equation}
\label{elliptic_case:boundary}
f=g~\mbox{on }\,\Sigma_{2,3}=\Sigma_{2}\cup\Sigma_{3}.
\end{equation}
The existence of solutions $f$ in $\Cc(\VVV\cup\Sigma_{2,3})$ to \eqref{elliptic_case:interior}-\eqref{elliptic_case:boundary} has been studied by several authors, among them Kohn and Nirenberg \cite{KohNir-67}, Ole{\u\i}nik \cite{Olei-67}, Bony \cite{Bony-69}, and also Manfredini \cite{Manfredini-97} in the context of ultra-parabolic equations with Dirichlet boundary condition along the position$\times$velocity domain (the velocity space is assumed to be  bounded).
Stochastic interpretation of \eqref{elliptic_case:interior}-\eqref{elliptic_case:boundary} has been studied in Stroock and Varadhan \cite{StroockVaradhan1979}, Fre{\u\i}dlin \cite{Freidl-68}, and Friedman \cite{AFriedman1976}. However to the best of our knowledge, the regularity of $f$ along $\Sigma_{1}$ has not been considered outside a few works. We shall mention the works of  Ole{\u\i}nik and Radkevi{\v{c}} \cite{OleiRadk-73} and Taira \cite{Taira-93} who have shown the well-posedness of analytic solutions (on $\VVV$) to the  elliptic equation \eqref{elliptic_case:interior} with the homogeneous boundary condition
\begin{equation*}
f=0,~\mbox{on }\,\Sigma_{2,3},
\end{equation*}
under the particular assumption that the sets $(\Sigma_{i},i=0,1,2,3)$ are closed  and that $\overline{\Sigma_{2}\cup\Sigma_{3}}$ and $\overline{\Sigma_{0}\cup\Sigma_{1}}$ are disjoint. Note that such assumption does not hold in the situation of kinetic equations. In that situation,  existence of weak solution is well known (see, e.g., Degond \cite{degond-86}, Carrillo \cite{carrillo-98}). In particular,  Carrillo \cite{carrillo-98} considers the situation where \eqref{elliptic_case:boundary} is the specular boundary condition \eqref{eq:BCond} and,
establishes the existence of trace functions and a Green identity related to the transport operator $\Tt = \partial_{t}+(u\cdot \nabla_{x})$.

As a preliminary for the proof of Theorem \ref{thm:L2VFPDirichlet}, let us recall a well-known existence result for equation  \eqref{eq:VFPDirichlet}.
\begin{proposition}[Carrillo \cite{carrillo-98}, Theorem $2.2$, Proposition $2.4$  and Lemma $3.4$]\label{propo:VFPDirWellposed}
Given two nonnegative functions $f_{0}\in L^{2}(\Dd\times\er^d)$ and
$q\in L^{2}(\Sigma^{+}_{T})$, there exists a unique nonnegative function $f$ in
$\Cc([0,T];L^{2}(\Dd\times\er^{d}))\cap \Hh(Q_T)$ admitting a nonnegative trace $\gamma(f)\in L^{2}(\Sigma_{T})$ along the boundary $\Sigma_{T}$, satisfying equation \eqref{eq:VFPDirichlet} in the sense that
\begin{equation}\label{eq:VFPDirichletBis}
\begin{aligned}
&\partial_{t}f-(u\cdot \nabla_{x}f) -\frac{\sigma^{2}}{2}\triangle_{u}f=0,~\mbox{ in }\,\Hh'(Q_{T}),\\
&f(t=0,x,u)=f_{0}(x,u),\,\mbox{ on }\,\Dd\times\er^{d},\\
&\gamma(f)(t,x,u)=q(t,x,u),\,\mbox{ on }\,\Sigma^{+}_{T}.
\end{aligned}
\end{equation}
In particular, for all $t\in(0,T)$,
\begin{equation}\label{eq:VFPDirichletEnergy}
\begin{aligned}
\Vert f(t)\Vert^{2}_{L^{2}(\Dd\times\er^{d})}+
\sigma^{2}\Vert\nabla_{u}f\Vert^{2}_{L^{2}(Q_{t})}+
\Vert \gamma(f)\Vert^{2}_{L^{2}(\Sigma^{-}_{t})}
= \Vert f_{0}\Vert^{2}_{L^{2}(\Dd\times\er^{d})}+\Vert q\Vert^{2}_{L^{2}(\Sigma^{+}_{t})}.
\end{aligned}
\end{equation}
If, in addition $f_0\in L^{p}(\Dd\times\er^{d}), q\in L^{p}(\Sigma^{+}_{T})$ for $p\in[1,+\infty)$, then for all $t\in(0,T)$,
\begin{equation*}
\Vert f(t)\Vert^{p}_{L^{p}(\Dd\times\er^{d})}+
\Vert f\Vert^{p}_{L^{p}(\Sigma^{-}_{t})}
+{\sigma^{2}p(p-1)\int_{Q_t}|\nabla_{u}f|^2f^{p-2}}
\leq \Vert f_{0}\Vert^{p}_{L^{p}(\Dd\times\er^{d})}+\Vert q\Vert^{p}_{L^{p}(\Sigma^{+}_{t})}.
\end{equation*}
\end{proposition}
\begin{remark}\label{rem:LinkBetweenEquationsFormulation} Eq. \eqref{eq:VFPDirichletBis} provides a variational formulation of the abstract Cauchy problem \eqref{eq:VFPDirichlet} in the sense that
if $f$ and $\gamma(f)$ satisfy Eq. \eqref{eq:VFPDirichletBis} then, for all $t\in[0,T]$, for all $\psi\in\Cc^{\infty}_{b}(\overline{Q}_{t})$, we have
\begin{equation}\label{eq:VFPDirichletBisBis}
\begin{aligned}
&\int_{Q_{t}}f(s,x,u)\left(\partial_{s}\psi-(u\cdot \nabla_{x}\psi)-\frac{\sigma^{2}}{2}\triangle_{u}\psi\right)(s,x,u)\,ds\,dx\,du\\
&=\int_{\Dd\times\er^{d}}\big[\psi(s,x,u) f(s,x,u)\big]_{s=0}^{s=t}\,dx\,du
-\int_{\Sigma^{-}_{t}}(u\cdot \nd(x))\gamma(f)(s,x,u)\psi(s,x,u)\,d\lambda_{\Sigma_T}(s,x,u)\\
&\quad-\int_{\Sigma^{+}_{t}}(u\cdot \nd(x))q(s,x,u)\psi(s,x,u)\,d\lambda_{\Sigma_T}(s,x,u),
\end{aligned}
\end{equation}
which expresses that $f$ is a solution to \eqref{eq:VFPDirichlet} in the sense of distributions. Let us further notice that the trace function $\gamma(f)$ in $L^{2}(\Sigma_{T})$ is characterized by
the Green formula related to the transport operator $\partial_{t}+(u\cdot \nabla_{x})$ (we refer to Subsection \ref{subsec:LinearVFPequation} for more details).
\end{remark}
Considering the solution $f$ in $\Cc([0,T];L^{2}(\Dd\times\er^{d}))\cap \Hh(Q_T)$ of \eqref{eq:VFPDirichletBis}, given by Proposition \ref{propo:VFPDirWellposed}, we show its interior regularity and its continuity up to and  along $\Sigma_{T}\setminus\Sigma^{0}_T$. The proof of the following proposition is postponed to  Appendix \ref{sec:appendix-shauder}.
\begin{proposition}[Interior regularity]\label{prop:Schauder_estimate} Under \hypbarriere, the unique solution $f$ of \eqref{eq:VFPDirichletBis} belongs to $\Cc^{1,1,2} (Q_T)$.
\end{proposition}

\begin{proposition}[Continuity up to $\Sigma^{+}_{T}$]\label{prop:ContUp2RelevantBoundary}
Assume \hyplangii~and \hypbarriere.  Let $f\in\Cc^{1,1,2}(Q_{T})\cap
\Cc([0,T];L^{2}(\Dd\times\er^{d}))\cap \Hh(Q_T)$ be the solution to \eqref{eq:VFPDirichletBis} with inputs $(f_{0},q)$. Then $f$ is continuous up to $\Sigma^{+}_{T}$.
\end{proposition}
\begin{proof} To show the continuity up to the boundary $\Sigma^{+}_{T}$, we follow the classical method of local barrier functions (see e.g. \cite{GilbTrud-83}). Let $(t_0,x_0,u_0)\in \Sigma^{+}_{T}$. Since $q$ is continuous in $\Sigma^{+}_{T}$, we can assume that for any $\epsilon>0$, there exists a neighborhood $\Oo_{t_{0},x_{0},u_{0}}$ of $(t_{0},x_{0},u_{0})$ such that
\begin{equation*}
q(t_{0},x_{0},u_{0})-\epsilon\leq q(t,x,u)\leq q(t_{0},x_{0},u_{0})+\epsilon,~~\forall (t,x,u)\in\Oo_{t_{0},x_{0},u_{0}}\cap\Sigma^{+}_{T}.
\end{equation*}
In addition, since $(u_0\cdot n_{\Dd}(x_0)) >0$, by  reducing $\Oo_{t_0,x_0,u_0}$, we can  assume that $\varsigma$ (the signed distance to $\partial\Dd$ given in \eqref{eq:signed-distance})  is in $\Cc^{2}(\overline{\Oo_{t_0,x_0,u_0}})$ and that $(u\cdot \nabla\varsigma(x))>\eta$ for all $(t,x,u)\in\mathcal{O}_{t_{0},x_{0},u_{0}}$, for some positive $\eta$\footnote{For instance, set $\Dd^\mu := \{x\in \Dd; -\varsigma (x)\leq \mu\}$. Assuming that $\varsigma\in\Cc^{2}$ for all $x\in \Dd^\mu$ for some $\mu>0$, by setting $C_0:=(u_{0}\cdot n_{\Dd}(x_{0})) >0$, one can choose $\delta'\in (0,\mu)$ so that $\delta'\sup_{x\in \Dd^\mu}|\nabla \varsigma(x)| + (\delta' +|u_{0}|)\sup_{x\in\Dd^\mu}|\nabla^2 \varsigma(x)|<C_0/2$.
Therefore, we have
 $(u\cdot n_{\Dd}(x))\geq (u_{0}\cdot n_{\Dd}(x_{0}))-\left|(u_{0}\cdot n_{\Dd}(x_{0}))-(u\cdot \nabla\varsigma(x))\right|\geq C_0-\delta'\sup_{\Dd^\mu}|\nabla\varsigma|+(\delta'+|u_{0}|)\left(\sup_{\Dd^\mu}|\nabla^2\varsigma|\right) >C_0/2 := \eta$.}.
Consequently, by setting $\varrho(x)= -\varsigma(x)$ and
\begin{equation*}
L:=\partial_{t}-(u\cdot\nabla_{x})-\frac{\sigma^{2}}{2}\triangle_{u},
\end{equation*}
we observe that, for all $(t,x,u)\in \Oo_{t_0,x_0,u_0}$,
\begin{equation}\label{proof:barrier-stp1}
L(\varrho)(t,x,u)
=-(u\cdot \nabla\varrho(x))=(u\cdot \nabla\varsigma(x))>\eta> 0.
\end{equation}
Reducing again $\mathcal{O}_{t_{0},x_{0},u_{0}}$, we can assume that $\mathcal{O}_{t_{0},x_{0},u_{0}}$ has the form $(t_{0}-\delta,t_{0}+\delta)\times B_{x_{0}}(\delta')\times B_{u_{0}}(\delta')$ (where $B_{x_{0}}(\delta')$ [resp. $B_{x_{0}}(\delta')$] is the ball centered in $x_{0}$ [resp. $u_{0}$] of radius $\delta'$) for some positive constants $\delta,\delta'>0$ such that $0\leq t_{0}-\delta<t_{0}+\delta\leq T$.

We can construct the barrier functions {related to} $(t_{0},x_{0},u_{0})\in\Sigma^{+}_{T}$ {with}
\begin{equation}\label{VFP_Dirichlet:BarrierFunc}
\begin{aligned}
\overline{\omega}_{\epsilon}(t,x,u)&=q(t_{0},x_{0},u_{0})+\epsilon+k_{\epsilon}{\psi_{x_{0}}(x)}+K_{\epsilon}\varrho(x),\\
\underline{\omega}_{\epsilon}(t,x,u)&=q(t_{0},x_{0},u_{0})-\epsilon-k_{\epsilon}{\psi_{x_{0}}(x)}-K_{\epsilon}\varrho(x).
\end{aligned}
\end{equation}
where $\psi_{x_{0}}(x)=(x-x_0)^{2}$ and where the parameters $k_{\epsilon},\,K_{\epsilon}\in\er^{+}$ are chosen large enough so that, for $M^{+}$ [resp. $M^{-}$] an upper-bound [resp. lower-bound] of $f$ on $\partial \mathcal{O}_{t_{0},x_{0},u_{0}}\cap Q_{T}$, we have
\begin{align*}
&k_{\epsilon}\inf_{\mathcal{O}_{t_{0},x_{0},u_{0}}\cap Q_{T}}L(\psi_{x_{0}})+K_{\epsilon}\inf_{\mathcal{O}_{t_{0},x_{0},u_{0}}\cap Q_{T}}L(\varrho)\geq 0,\\
&k_{\epsilon}\inf_{\partial\mathcal{O}_{t_{0},x_{0},u_{0}}\cap Q_{T}}\psi_{x_{0}}+K_{\epsilon}\inf_{\partial\mathcal{O}_{t_{0},x_{0},u_{0}}\cap Q_{T}}\varrho\geq \left(M^{+}-(q(t_{0},x_{0},u_{0})+\epsilon)\right) \vee \left(q(t_{0},x_{0},u_{0})-\epsilon\right).
\end{align*}
For example, setting $\eta:=\inf_{\overline{\mathcal{O}_{t_{0},x_{0},u_{0}}\cap Q_{T}}}(u \cdot \nabla\varsigma(x))$, as  $ \inf_{\partial\mathcal{O}_{t_{0},x_{0},u_{0}}\cap Q_{T}}\psi_{x_{0}} = \delta'^2$, one can choose  $k_{\epsilon}$ and $K_{\epsilon}$ such that
\begin{align*}
-(\delta')^{2}k_{\epsilon} +K_{\epsilon}\eta =0,\quad \quad k_{\epsilon}(\delta')^{2}
=\left(M^{+}-q(t_{0},x_{0},u_{0})\right) \vee q(t_{0},x_{0},u_{0}).
\end{align*}
Thus, $\overline{\omega}_{\epsilon}$ and $\underline{\omega}_{\epsilon}$ satisfy the  properties
\begin{equation*}
\textit{(P)-}\,\,\left\{
\begin{aligned}
&(a)\,\,\overline{\omega}_{\epsilon}(t,x,u)\geq q(t,x,u)\geq\underline{\omega}_{\epsilon}(t,x,u)\,\,\mbox{  for all }\,(t,x,u)\in \mathcal{O}_{t_{0},x_{0},u_{0}}\cap (0,T)\times\partial\Dd\times\er^{d},\\
&(b)\,\,L(\overline{\omega}_{\epsilon})\geq 0 \geq L(\underline{\omega}_{\epsilon})\,\,\mbox{  for all }\,(t,x,u)\in \mathcal{O}_{t_{0},x_{0},u_{0}}\cap Q_{T},\\
&(c)\,\,\overline{\omega}_{\epsilon}(t,x,u)\geq M^{+},\,\mbox{  and }\,\underline{\omega}_{\epsilon}(t,x,u)\leq M^{-},\,\,\mbox{  for all }\,(t,x,u)\in \partial\mathcal{O}_{t_{0},x_{0},u_{0}}\cap Q_{T},\\
&(d)\,\,\lim_{\epsilon\rightarrow 0^{+}}\overline{\omega}_{\epsilon}(t_{0},x_{0},u_{0})=\lim_{\epsilon\rightarrow 0^{+}}\underline{\omega}_{\epsilon}(t_{0},x_{0},u_{0})=q(t_{0},x_{0},u_{0}).
\end{aligned}
\right.
\end{equation*}
Now we shall prove that, for $f$ the solution to \eqref{eq:VFPDirichlet}, $\underline{\omega}_{\epsilon}\leq f^{}\leq \overline{\omega}_{\epsilon}$ on $\mathcal{O}_{t_{0},x_{0},u_{0}}\cap Q_{T}$. Owing to the property {\it{(P)-}}$(d)$, this allows to conclude that $f^{}(t,x,u)$ tends to $q(t_{0},x_{0},u_{0})$ as $(t,x,u)$ tends to $(t_{0},x_{0},u_{0})$, for all $(t_{0},x_{0},u_{0})$ of $\Sigma^{+}_{T}$.

For the local comparison between $\overline{\omega}_{\epsilon}$ and $f$, let us consider the positive part $(f-\overline{\omega}_{\epsilon})^{+}$ of $f-\overline{\omega}_{\epsilon}$.
Let $\eta_{0}$ denote some nonnegative cut-off function defined in a neighborhood of $(t_{0},x_{0},u_{0})$ such that $\eta_{0}(t,x,u)=0$ for all $(t,x,u)\in \partial \mathcal{O}_{t_{0},x_{0},u_{0}}$, and let $\beta$ be a real parameter that we will specify later.  The function $\triangle_{u}|(f^{}-\overline{\omega}_{\epsilon})^{+}|^{2}$ is well defined a.e. on $Q_{T}$ since, using Theorem \ref{thm:SobolevRegNegativePart} (see e.g  Tartar \cite{tartar-78}), one
can check that  $\triangle_{u}|(f^{}-\overline{\omega}_{\epsilon})^{+}|^{2}=2\nabla_{u}\cdot ((f^{}-\overline{\omega}_{\epsilon})^{+}\nabla_{u}(f^{}-\overline{\omega}_{\epsilon}))=2  ((f^{}-\overline{\omega}_{\epsilon})^{+}\triangle_{u}(f^{}-\overline{\omega}_{\epsilon}))+2\left|\nabla_{u}(f^{}-\overline{\omega}_{\epsilon})\right|^{2}\ind_{\{f^{}>\overline{\omega}_{\epsilon}\}}$.
We shall observe that
\begin{align*}
&L(\eta_{0}\exp{\{\beta t\}}\left|(f^{}-\overline{\omega}_{\epsilon})^{+}\right|^{2})\\
&=\left|(f^{}-\overline{\omega}_{\epsilon})^{+}\right|^{2}L(\eta_{0}\exp{\{\beta t\}})+\eta_{0}\exp{\{\beta t\}}L(\left|(f^{}-\overline{\omega}_{\epsilon})^{+}\right|^{2})-
\sigma^{2}\exp{\{\beta t\}}\left(\nabla_{u}\eta_{0}\cdot \nabla_u\left|(f^{}-\overline{\omega}_{\epsilon})^{+}\right|^{2}\right).
\end{align*}
The property {\it{(P)-}}$(b)$ ensures that
\begin{equation*}
L(\left|(f-\overline{\omega}_{\epsilon})^{+}\right|^{2})=(f-\overline{\omega}_{\epsilon})^{+}L(f-\overline{\omega}_{\epsilon})-\frac{\sigma^{2}}{2}
\left|\nabla_{u}(f^{}-\overline{\omega}_{\epsilon})\right|^{2}\ind_{\{f^{}>\overline{\omega}_{\epsilon}\}}\leq -\frac{\sigma^{2}}{2}
\left|\nabla_{u}(f^{}-\overline{\omega}_{\epsilon})\right|^{2}\ind_{\{f^{}>\overline{\omega}_{\epsilon}\}}\leq 0,
\end{equation*}
and thus
\begin{align*}
& L(\eta_{0}\exp{\{\beta t\}}\left|(f^{}-\overline{\omega}_{\epsilon})^{+}\right|^{2})\\
& \leq \exp{\{\beta t\}}\left|(f^{}-\overline{\omega}_{\epsilon})^{+}\right|^{2}\left\{L(\eta_{0})+\beta\eta_{0}\right\}-\sigma^{2}\exp{\{\beta t\}}\left(\nabla_{u}\eta_{0}\cdot \nabla_{u} \left|\left(f^{}-\overline{\omega}_{\epsilon}\right)^{+}\right|^{2}\right).
\end{align*}
We integrate the two sides above over  $\mathcal{O}_{t_{0},x_{0},u_{0}}\cap Q_{T}$. Since $\eta_{0}=0$ on $\partial \mathcal{O}_{t_{0},x_{0},u_{0}}$,
an integration by parts yields
\begin{align*}
&\int_{\mathcal{O}_{t_{0},x_{0},u_{0}}\cap Q_{T}}L(\eta_{0}\exp{\{\beta t\}}\left|(f^{}-\overline{\omega}_{\epsilon})^{+}\right|^{2})(t,x,u)\\
&=-\int_{\mathcal{O}_{t_{0},x_{0},u_{0}}\cap \Sigma_{T}}(u\cdot n_{\Dd}(x))\eta_{0}(t,x,u)\exp{\{\beta t\}}|(\gamma(f^{})-\overline{\omega}_{\epsilon})^{+}(t,x,u)|^{2}d\lambda_{\Sigma_{T}}(t,x,u).
\end{align*}
Hence, we have obtained that
\begin{align*}
&-\int_{\mathcal{O}_{t_{0},x_{0},u_{0}}\cap\Sigma_{T} }
(u\cdot n_{\Dd}(x))\eta_{0}(t,x,u)\exp{\{\beta t\}}|(\gamma(f^{})-\overline{\omega}_{\epsilon})^{+}(t,x,u)|^{2}d\lambda_{\Sigma_{T}}(t,x,u)\\
&\leq \int_{\mathcal{O}_{t_{0},x_{0},u_{0}}\cap Q_{T}}\left(L(\eta_{0})+\beta\eta_{0}+\sigma^{2}\triangle_{u}\eta_{0}\right)\exp{\{\beta t\}}|(f^{}-\overline{\omega}_{\epsilon})^{+}|^{2},
\end{align*}
or equivalently, since $\eta_0=0$ on $\Sigma^{-}_{T}$, (as $(x_0,u_0)\in\Sigma^+$)
\begin{align*}
&-\int_{\mathcal{O}_{t_{0},x_{0},u_{0}}\cap\Sigma^{+}_{T}}(u\cdot n_{\Dd}(x))\eta_{0}(t,x,u)\exp{\{\beta t\}}|(q-\overline{\omega}_{\epsilon})^{+}(t,x,u)|^{2}d\lambda_{\Sigma_{T}}(t,x,u)\\
&\leq \int_{\mathcal{O}_{t_{0},x_{0},u_{0}}\cap Q_{T}}\left(L(\eta_{0})+\beta\eta_{0}+\sigma^{2}\triangle_{u}\eta_{0}\right)\exp{\{\beta t\}}|(f^{}-\overline{\omega}_{\epsilon})^{+}|^{2}
\end{align*}
From {\it{(P)-}}$(a)$ and {\it{(P)-}}$(c)$, the integral along $\mathcal{O}_{t_{0},x_{0},u_{0}}\cap \Sigma^{+}_{T}$ is  nonnegative. By choosing $\eta_{0}$ and $\beta\in \er$ such that
\begin{equation*}
L(\eta_{0})+\beta\eta_{0} +\sigma^{2}\triangle_{u}\eta_{0} < 0,
\end{equation*}
we conclude that $f^{}\leq \overline{\omega}_{\epsilon}$ on $\mathcal{O}_{t_{0},x_{0},u_{0}}$. Similar arguments entail that $\underline{\omega}_{\epsilon}\leq f^{}$.
\end{proof}
\paragraph{Feynman-Kac representation and continuity up to and along $\Sigma^{-}_{T}$. }
We prove the Feynman-Kac representation \eqref{eq:FeynmanKacInterpretation} by replicating the arguments of Friedman
\cite[Chapter 5, Theorem 5.2]{AFriedman1975}: for $(y,v)\in\Dd\times\er^{d}$ fixed, let $((x^{y,v}_{t},u^{y,v}_{t});\,t\in[0,T])$ satisfy 
\begin{equation*}
\left\{
\begin{aligned}
&x^{y,v}_t = y + \int_0^t u^{y,v}_s ds,\\
&u^{y,v}_t = v + \sigma W_t,
\end{aligned}
\right.
\end{equation*}
where $(W_{t};\,t\geq 0)$ is an $\er^{d}$-valued Brownian motion defined on some  probability space $(\Omega,\Ff,\PP)$. Set
$\beta^{y,v}_{\delta}:=\inf\{t>0\,;\,d(x^{y,v}_{t},\partial\Dd)\leq \delta\}$. Since $f$ is smooth in the interior of $Q_{T}$ and satisfies \eqref{eq:VFPDirichlet}, applying It\^{o}'s formula to $f(t-s,x^{y,v}_{s\wedge\beta^{y,v}_{\delta}},u^{y,v}_{s\wedge\beta^{y,v}_{\delta}})$, for $s\in[0,t]$, yields
\begin{equation*}
f(t,y,v)
=\EE_{\PP}
\left[f_0(x^{y,v}_{t},u^{y,v}_{t})
\ind_{\displaystyle\{t\leq\beta^{y,v}_{\delta}\}}\right]
+\EE_{\PP}\left[f(t-\beta^{y,v}_{\delta},
x^{y,v}_{\beta^{y,v}_{\delta}},
u^{y,v}_{\beta^{y,v}_{\delta}})
\ind_{\displaystyle\{t>\beta^{y,v}_{\delta}\}}\right].
\end{equation*}
Since $\PP$-a.s., $\beta^{y,v}_{\delta}$ tends to $\beta^{y,v} = \inf\{t>0\,;\,d(x^{y,v}_{t},\partial\Dd) =0\}$, as $\delta$ tends to $0$, and thanks to Proposition \ref{prop:ContUp2RelevantBoundary}, one obtains   \eqref{eq:FeynmanKacInterpretation}.
\begin{proposition}\label{prop:ContU2IrrelevantBound}
Assume \hypbarriere. Let $f^{}\in\Cc^{1,1,2}(Q_{T})\cap\Cc(Q_{T}\cup\Sigma^{+}_{T})$ be the solution to \eqref{eq:VFPDirichlet}.
Then $f^{}$ is continuous along and up to $\Sigma^{-}_{T}$.
\end{proposition}
\begin{proof}
According to \eqref{eq:FeynmanKacInterpretation} and since $f_{0}$ and $q$ are continuous, the continuity of $f$ up to $\Sigma^{-}_{T}$ will follow from the continuity of $(y,v)\mapsto (\beta^{y,v},x^{y,v}_{t},u^{y,v}_{t})$. $\PP$-almost surely, for all $t\geq 0$, the flow $(y,v)\mapsto (x^{y,v}_{t},u^{y,v}_{t})$ is continuous on $\er^{d}\times\er^{d}$.
As $(y,v)\notin\Sigma^0 \cup \Sigma^+$, we have
$\beta^{y,v} = \tau^{y,v} := \inf\{t>0;\,x^{y,v}_t\notin \overline{\Dd}\}$. To prove that $(y,v)\mapsto \tau^{y,v}$ is continuous up to $\Sigma^{-}$,  we follow the general proof of the continuity of exit time related to a flow of continuous processes given in Proposition $6.3$ in Darling and Pardoux \cite{DarlingPardoux-97}. First, replicating the argument of the authors, one can show that, for all $(y_{m},v_{m})\in \Dd\times\er^{d}$ such that {$\lim_{m\rightarrow +\infty}(y_{m},v_{m})=(y,v)\in\Sigma^-$,}
\begin{equation*}
\limsup_{m\rightarrow +\infty}\tau^{y_{m},v_{m}}\leq\tau^{y,v}.
\end{equation*}
Next, it is sufficient to check that
\begin{equation*}
\tau^{y,v}\leq\liminf_{m\rightarrow +\infty}\tau^{y_{m},v_{m}}.
\end{equation*}
Following \cite{DarlingPardoux-97}, we may observe that, as in the proof of Theorem \ref{thm:LinearSolutionByStraightening} , for a.e. $(y,v)\in\Dd\times\er^{d}\cup\Sigma^{-}$, the path  $t\mapsto(x^{y,v}_t,u^{y,v}_t)$ never hits $\Sigma^{0}{\cup \Sigma^{-}}$, and, {since  $\PP$-a.s.
$(t,y,v)\mapsto (x^{y,v}_{t},u^{y,v}_{t})$ is continuous on $[0,+\infty)\times\overline{\Dd}\times\er^{d}$, one can check that}
\begin{equation*}
\overline{\{(x^{y_{m},v_{m}}_{\tau^{y_{m},v_{m}}},
u^{y_{m},v_{m}}_{\tau^{y_{m},v_{m}}});\,m\in\NN\}}\subset\Sigma^{+},
\end{equation*}
{and that $(x^{y,v}_{\liminf_{m\rightarrow +\infty}\tau^{y_{m},v_{m}}},
u^{y,v}_{\liminf_{m\rightarrow +\infty}\tau^{y_{m},v_{m}}})\in \Sigma^{+}$.
Since $\tau^{y,v}=\inf\{t>0;\,(x^{y,v}_{t},u^{y,v}_{t})\in\Sigma^{+}\}$,}  we deduce that
$\tau^{y,v}\in [0,\liminf_{m\rightarrow +\infty}\tau^{y_{m},v_{m}}]$.
\end{proof}
This ends the proof of Theorem \ref{thm:L2VFPDirichlet}.
\section{On the conditional
McKean-Vlasov-Fokker-Planck equation}\label{sec:PDEApproach}
In this section, we construct a probability density function satisfying (in the sense of distribution):
\begin{subequations}\label{eq:ConditionalMcK-VFP}
\begin{align}
&\label{eq:VFPnonlineaire_interieur}\partial_{t} \rho + (u\cdot\nabla_{x}\rho) + (B[\cdot~;\rho]\cdot\nabla_{u}\rho)
-\frac{\sigma^{2}}{2}\triangle_{u} \rho = 0,\,\mbox{ on}~(0,T)\times\Dd \times\mathbb{R}^{d},
\\
&\label{eq:VFPnonlineaire_initial}\rho(0,x,u) = \rho_{0}(x,u),\,\mbox{ on}~\Dd \times\mathbb{R}^{d},
\\
\label{eq:specularcondition}
&\gamma(\rho)(t,x,u) = \gamma(\rho)(t,x,u-2(u\cdot n_{\Dd }(x))n_{\Dd }(x)),\,\mbox{ on}~(0,T)\times\partial\Dd \times\mathbb{R}^{d},
\end{align}
\end{subequations}
where $B$ is defined as in \eqref{eq:DriftDefinitionEDP}.
Clearly \eqref{eq:ConditionalMcK-VFP} is the equation of the time marginal law of $((X_t,U_t);\,t\in[0,T])$ solution to  \eqref{eq:NonlinearConfinedLangevin}. In particular  \eqref{eq:specularcondition} takes into account the specular reflection resulting from the confinement component $(K_t;\,t\in[0,T])$ in \eqref{eq:NonlinearConfinedLangevin}.

Throughout this section, we  refer to equation \eqref{eq:ConditionalMcK-VFP}  as the conditional
McKean-Vlasov-Fokker-Planck equation. Furthermore, for notation convenience, we  denote by $\Tt$ the transport operator in \eqref{eq:VFPnonlineaire_interieur}, namely  for all test function $\psi$ on $Q_{T}$,
\begin{equation}
\label{Def:OperateurTransport}
\trans(\psi) = \partial_{t}\psi + (u\cdot\nabla_{x}\psi).
\end{equation}

As mentioned in Subsection \ref{sec:PDEStudy}, the well-posedness of the linear Vlasov-Fokker-Planck equation and the related trace problem has
been well studied in the literature of kinetic equation (we particularly refer to Degond \cite{degond-86}, Degond and Mas-Gallic \cite{degond-masgallic-85}, Carrillo \cite{carrillo-98} and Mischler \cite{mischler-10}). More recently, in his study of geometric Kramers-Fokker-Planck operators with boundary
conditions \cite{Nier-14}, Nier showed how to associate a boundary condition operator with the linear Vlasov-Fokker-Planck equation, corresponding to a Langevin stochastic dynamics with jump process at the boundary.
This methodology may offer some other perspectives to extend the construction of the nonlinear Langevin dynamics with a general class of boundary conditions. 

For the study of Eq. \eqref{eq:ConditionalMcK-VFP}, the two main difficulties are  in the fractional form  of $B[\cdot~;\rho]$,
 and in the verification of the properties \eqref{nopermeabilitycondition_1} and \eqref{nopermeabilitycondition_2} of the trace function $\gamma(\rho)$
 (see Definition \ref{def:Trace_density_process}). For this purpose, starting from the assumption \hypedpiii, we prove  the existence of a solution to   \eqref{eq:ConditionalMcK-VFP}, as well as the existence of related Maxwellian  upper and 
 lower-bounds. These Maxwellian bounds are of the following form.
\begin{definition}\label{Def:MaxwellBounds}
For given $a\in\er,~\mu >0,~P_{0}\in L^{1}(\er^{d})$, such that $P_{0}\geq 0$ on $\er^{d}$,  a Maxwellian distribution with parameters
$(a,\mu,P_{0})$  is a  function $P :\er^{+}\times\er^{d} \rightarrow \er^{+}$ such that
\begin{equation}\label{defMaxwell}
P(t,u) = \exp\{a t\}\left[m(t,u)\right]^{\mu},
\end{equation}
where $m:\er^{+}\times\er^{d}\rightarrow \er^+$  is defined by $m(t,u) = (G(\sigma^2t) * P^{\frac{1}{\mu}}_{0}) (u)$, with $G(t,u) = \left(\frac{1}{2\pi t}\right)^{\frac{d}{2}}\exp\{\frac{ -|u|^{2}}{2t}\}$.
\end{definition}
\begin{remark}\label{Cond_spec_pour_maxwellienne}
Let $\mathit{p}$ be a Maxwellian distribution with parameters $\left(a,\mu,\mathit{p}_{0}
\right)$. If $\mathit{p}_{0}(u)=\mathit{p}_{0}(|u|)$ then, the Maxwellian distribution is invariant for specular reflection. More precisely, for all vector $\vec{n}\in \er^{d}$ such that $\|\vec{n}\|=1$,
\[\mathit{p}(t,u-2 (u\cdot \vec{n})\vec{n})=\mathit{p}(t,u),\,\mbox{ for a.e. }~(t,u)\in(0,+\infty)\times\er^{d}.\]
\end{remark}
This section is now devoted to the proof of the following existence result.
\begin{theorem}\label{thm:Existence_result} Under \hypedp, there exists a function $\rho\in \VoneT$, and there exist $\gamma^{+}(\rho)$, $\gamma^{-}(\rho)$
defined on $\Sigma^{+}_{T}$ and $\Sigma^{-}_{T}$ respectively, with $\gamma^{\pm}(\rho)\in L^{2}(\omega,\Sigma^{\pm}_{T})$,  such that, for all $t\in (0,T]$, for all $\psi\in\Cc^{\infty}_{c}(\overline{Q_{t}})$,
\begin{equation}
\label{VFPnonlineaire_variationnel}
\begin{aligned}
&\int_{Q_{t}}\left(\rho\trans(\psi)
+\psi \left(B[\cdot~;\rho]\cdot\nabla_{u}\rho\right)
+\frac{\sigma^{2}}{2}
\left(\nabla_{u}\psi\cdot\nabla_{u}\rho\right)\right)(s,x,u)\,ds\,dx\,du\\
&=\int_{\Dd\times\er^{d}}\rho(t,x,u)\psi(t,x,u)\,dx\,du -
\int_{\Dd\times\er^{d}} \rho_{0}(x,u)\psi(0,x,u)\,dx\,du\\
&\quad+\int_{\Sigma^{+}_{t}}(u\cdot n_{\Dd}(x))\gamma^{+}(\rho)(s,x,u)\psi(s,x,u)d\lambda_{\Sigma_{T}}(s,x,u)\\
&\quad+\int_{\Sigma^{-}_{t}}(u\cdot n_{\Dd}(x))
\gamma^{+}(\rho)(s,x,u-2(u\cdot n_{\Dd }(x))n_{\Dd }(x))\psi(s,x,u)d\lambda_{\Sigma_{T}}(s,x,u).
\end{aligned}
\end{equation}
In addition, there exist a couple of Maxwellian distributions $\left(\overline{P},~\underline{P}\right)$ such that
\begin{align}\label{EncadrMaxwell}
\begin{aligned}
&{\displaystyle\underline{P}\leq \rho \leq \overline{P},~\mbox{a.e. on }Q_{T},}\\
&{\displaystyle\underline{P}\leq \gamma^{\pm}(\rho) \leq \overline{P},~\lambda_{\Sigma_{T}}\mbox{-a.e. on }\Sigma^{\pm}_{T},}
\end{aligned}
\end{align}
$\overline{P}$ and $\underline{P}$ satisfy the specular boundary condition \eqref{eq:specularcondition}, and for all $t\in(0,T]$,
\begin{subequations}\label{Maxwell_bord}
\begin{align}
&\sup_{t\in(0,T)}\int_{\er^{d}}(1+|u|)\omega(u)\left(\overline{P}(t,u)\right)^{2}\,du<+\infty\label{MaxwellUpperBound},\\
&\inf_{t\in(0,T)}\int_{\er^{d}}\underline{P}(t,u)\,du>0\label{MaxwellLowerBound}.
\end{align}
\end{subequations}
\end{theorem}
\begin{remark}
As $\Dd$ is a bounded domain, the lower-bound in  \eqref{EncadrMaxwell} is well defined in $L^2(\omega, Q_T)$ (since $\|\underline{P}\|_{ L^2(\omega, Q_T)}=\sqrt{|\Dd|}\|\underline{P}\|_{ L^2(\omega, (0,T)\times\er^d)}$),  and does not contradict the fact that $\rho\in \VoneT$. In the case of an unbounded domain $\Dd$,  the Maxwellian bounds should involve some Maxwellian distributions in the space variable also.
\end{remark}
The main steps of the proof of Theorem \ref{thm:Existence_result} are the following. First,  in the next Subsection \ref{subsec:LinearVFPequation}, we consider a linear version of equation \eqref{eq:ConditionalMcK-VFP} where a Dirichlet condition is imposed on $\Sigma^{-}_{T}$, and where the drift coefficient is given in $L^{\infty}((0,T)\times\Dd;\er^{d})$.
Under \hypedpii~ and \hypedpiii, the problem is well-posed in $\VoneT$ (see Lemma \ref{lem:ExistVFPlineaire}). As a preliminary step, we also highlight some meaningful properties on the transport operator $\Tt$ and the Green identity related to the Vlasov-Fokker-Planck equation in the weighted spaces $\VoneT$.  Then, in Subsection \ref{subsec:LinearPDE},  we show the existence of Maxwellian bounds satisfying the requirements
\eqref{Maxwell_bord} (see Proposition \ref{prop:EncadrLineaire}) and which are identified as super-solution and sub-solution for the linear problem. Next in
Subsection \ref{subsec:NonLinMKFP}, by means of fixed point methods, we successively construct a solution to the equation with the specular boundary condition
(see  Proposition \ref{prop:VFPlineairespec}) and with the nonlinear term $B[\cdot~;\rho]$ (see Proposition \ref{prop:FixedPts}).

Let us also emphasize that the weight $\omega(u)$ defined in \eqref{def:Poids}  is useful  here to preserve the probabilistic interpretation of \eqref{eq:ConditionalMcK-VFP} while working in $L^{2}$-space. Later it also allows a fixed point argument. Let us remark the following properties:
\begin{lemma}\label{lem:major_poids}
For the weight function $u\mapsto\omega(u) = (1+|u|^2)^{\alpha/2}$, $\alpha>d\vee 2$, for all $u$ and $u'$ in $\er^d$,
\begin{eqnarray*}
\begin{array}{lll}
&(i)  &~\omega(u+u')\leq 2^{\frac{\alpha}{2}}\left(\omega(u)+\omega(u')\right),\\
&(ii)& ~(u\cdot\nabla\omega(u))\geq 0,~\mbox{and } ~|\nabla\omega(u)|\leq
\alpha\omega(u),\\
&&~ |(u\cdot \sqrt{\omega(u)})|
\leq \frac{\alpha}{4}\sqrt{\omega(u)},\\
&&~|\nabla\sqrt{\omega(u)}|\leq \frac{\alpha}{{2}}\sqrt{\omega(u)},\\
&&~ \triangle\omega(u)\leq \alpha (\alpha -2 + d) \omega(u),\\
&(iii)&~ \int_{\er^{d}}\frac{\displaystyle du}{\displaystyle \omega(u)}<+\infty.
\end{array}
\end{eqnarray*}
\end{lemma}
\begin{proof}
The assertions $(i)$ and $(ii)$ are directly deduced from the calculations:
\begin{align*}
&\omega(u+u') = \left(1+\left|u+u'\right|^{2}\right)^{\frac{\alpha}{2}}\leq \left(1 + 2|u|^{2}+2|u'|^{2}\right)^{\frac{\alpha}{2}},\\
&\left(u\cdot \nabla\omega(u)\right)
={\frac{\alpha}{2}} |u|^{2}(1+|u|^{2})^{\frac{\alpha}{2}-1},\\
&\left|\nabla\omega\right|(u)=\alpha\left|u\right|\left(1+|u|^{2}\right)^{\frac{\alpha}{2}-1}\leq
\alpha \omega(u),\\
&\left|\nabla\sqrt{\omega(u)}\right|=
\frac{\alpha}{{2}}|u|\left(1+|u|^{2}\right)^{\frac{\alpha}{4}-1},\\
&\triangle\omega(u)=\alpha d \frac{\omega(u)}{(1+|u|^{2})}+2\alpha\left(
\frac{\alpha}{2}-1\right)\omega(u)\frac{|u|^{2}}{(1+|u|^{2})^{2}}.
\end{align*}
For $(iii)$, by a change of variable in the polar coordinates, we have
\begin{equation*}
\int_{\er^{d}}\frac{du}{\omega(u)} = |S_{d-1}|\int_{\er^{+}}\left(1+r^{2}\right)^{-\frac{\alpha}{2}}r^{d-1}\,dr\leq|S_{d-1}|\int_{\er^{+}}\left(1+|r|^{2}\right)^{-\frac{\alpha+d-1}{2}}\,dr
\end{equation*}
where $|S_{d-1}|$ is the Lebesgue measure of the unit sphere of $\er^{d}$. Since the right member is finite for $\alpha>d\vee 2$, $(iii)$ follows.
\end{proof}
\subsection{On the linear Vlasov-Fokker-Planck equation}\label{subsec:LinearVFPequation}
In this section, we set up the framework for the  proof of Theorem \ref{thm:Existence_result}, based on the
existence result of the linear Vlasov-Fokker-Planck equation and the associated spaces.

First we give some properties of the operator $\Tt$  defined in \eqref{Def:OperateurTransport}, that were initially stated in \cite{carrillo-98}, inspired from  ideas in \cite{degond-masgallic-85} and \cite{degond-86}.
For all $t\in(0,T]$, we consider the space
\begin{equation*}
\mathcal{Y}(Q_{t})=\left\{\phi\in \mathcal{H}(Q_{t})\mbox{ s.t. }\trans(\phi)
\in\mathcal{H}'(Q_{t})\right\},
\end{equation*}
equipped with the norm
\begin{equation*}
\|\phi\|^{2}_{\mathcal{Y}(Q_{t})} = \|\phi\|^{2}_{\mathcal{H}(Q_{t})} +
\|\trans(\phi)\|^{2}_{\mathcal{H}'(Q_{t})}.
\end{equation*}
We consider also the subset $\widetilde{\mathcal{Y}}(Q_{t})$  of all elements of $\mathcal{C}^{1,1,2}_{c}(\overline{Q_{t}})$, 
vanishing at the neighborhood of the boundaries $\{0\}\times\partial\Dd \times\mathbb{R}^{d}$,
$\{t\}\times\partial\Dd \times\mathbb{R}^{d}$ and $\Sigma^{0}_{t}$. We recall the following.
\begin{lemma}[Carrillo \cite{carrillo-98}, Lemma $2.3$ and its proof]\label{lem:Density}

For all $\psi\in\mathcal{Y}(Q_{t})$,
there exists a sequence $\left\{\psi_{n};\,n\in\mathbb{N}\right\}$ of
$\widetilde{\mathcal{Y}}(Q_{t})$ such that $\psi_{n}$ tends to $\psi$ for the norm $\Vert \, \Vert_{\mathcal{Y}(Q_{t})}$ when $n$ tends to $+\infty$. Moreover,  $\psi$ has trace values $\gamma^{+}(\psi)\in L^{2}(\Sigma^{+}_{T})$ (resp. $\gamma^{-}(\psi)\in L^{2}(\Sigma^{-}_{T})$) on $\Sigma^{+}_{T}$
(resp. on $\Sigma^{-}_{T}$) given by
\begin{align*}
\gamma^{\pm}(\psi) = \lim_{\substack{n\rightarrow+\infty}}\psi_{n}, ~\mbox{ in }\,L^{2}(\Sigma^{\pm}_{t})
\end{align*}
and, for all $t \in\left(0,T\right]$, $\psi(t,\cdot)$ belongs to
$L^{2}(\Dd \times\mathbb{R}^{d})$, with
\begin{align*}
\psi(t,\cdot)= \lim_{\substack{n\rightarrow+\infty}}\psi_{n}(t,\cdot),~\mbox{ in }\,L^{2}(\Dd\times\er^d).
\end{align*}
\end{lemma}
With the help of Lemma \ref{lem:Density}, we adapt the Green formula stated in \cite[Lemma 2.3]{carrillo-98} to the weighted spaces considered in this section.
\begin{lemma}\label{lem:Derive_Carrillo}
Let $\psi\in \Hh(\omega,Q_T)$ be such that $\trans(\sqrt{\omega}\psi)\in\Hh'(Q_{T})$.
Then $\psi$ has traces
$\gamma^{\pm}(\psi)\in L^{2}(\omega,\Sigma^{\pm}_{T})$, and
$\psi(t,\cdot) \in L^{2}(\omega,\Dd\times\er^{d})$. Moreover, for all $\phi$ in $\mathcal{Y}(Q_{T})$, $t\in[0,T]$,
\begin{equation}\label{Greenformula}
\begin{aligned}
&\left(\trans(\psi),\phi\right)_{\mathcal{H}'(Q_{t}),\mathcal{H}(Q_{t})} {+}
\left(\trans(\phi),\psi\right)_{\mathcal{H}'(Q_{t}),\mathcal{H}(Q_{t})}\\
&= \int_{\Dd \times\er^{d}} \psi(t,x,u)\phi(t,x,u)\,dx\,du
- \int_{\Dd \times\mathbb{R}^{d}} \psi(0,x,u)\phi(0,x,u)\,dx\,du\\
&\quad+\int_{\Sigma^{+}_{t}} (u\cdot n_{\Dd }(x))\gamma^{+}(\psi)(s,x,u) \gamma^{+}(\phi)(s,x,u) d\lambda_{\Sigma_{t}}(s,x,u)\\
&\quad+\int_{\Sigma^{-}_{t}} (u\cdot n_{\Dd }(x))\gamma^{-}(\psi)(s,x,u) \gamma^{-}(\phi)(s,x,u) d\lambda_{\Sigma_{t}}(s,x,u).
\end{aligned}
\end{equation}
\end{lemma}
\begin{proof}
In \cite[Lemma 2.3]{carrillo-98}, the existence of the trace values (in time and space) and the Green formula \eqref{Greenformula} are obtained for all $\psi \in \Hh(Q_T)$ such that $\trans(\psi) \in \Hh'(Q_T)$.
Here, we have  that $\trans(\sqrt{\omega}\psi) = \sqrt{\omega}\trans(\psi)$ and, for all $\phi\in \Hh(Q_T)$,
\begin{align*}
|\left(\Tt(\psi),\phi\right)_{\Hh'(Q_{T}),\Hh(Q_{T})}|=|(\Tt(\sqrt{\omega}\psi),\frac{\phi}{\sqrt{\omega}})_{\Hh'(Q_{T}),\Hh(Q_{T})}|&\leq \Vert\Tt(\sqrt{\omega}\psi)\Vert_{\Hh'(Q_{T})} \Vert \frac{\phi}{\sqrt{\omega}}\Vert_{\Hh(Q_{T})}\\
&\leq \frac{\alpha}{2}\Vert\Tt(\sqrt{\omega}\psi)\Vert_{\Hh'(Q_{T})} \Vert \phi\Vert_{\Hh(Q_{T})}.
\end{align*}
We deduce that $\trans(\psi)$ is in $\Hh'(Q_T)$,  and hence
 we deduce the Green formula.
\end{proof}

For given $q$, $B$ and $g$, let us consider the linear Vlasov-Fokker-Planck equation:
\begin{subequations}\label{eq:ToyEquation}
\begin{align}
&\Tt(f)=\frac{\sigma^{2}}{2}\triangle_{u}f-\left(\nabla_{u}\cdot B f\right)+g,\,\mbox{ in }\Hh'(Q_{T}),\label{eq:ToyEquationInterior}\\
&f(0,x,u)=\rho_{0}(x,u),\,\mbox{ on }\,\Dd\times\er^{d}\label{eq:ToyEquationInitial},\\
&\gamma^{-}(f)(t,x,u)=q(t,x,u),\,\mbox{ on }\,\Sigma^{-}_{T}\label{eq:ToyEquationBoundary}.
\end{align}
\end{subequations}
\begin{lemma}\label{lem:ExistVFPlineaire} Assume \hypedpii. Then, given
 $B \in L^{\infty}((0,T)\times\Dd;\er^{d})$,
$q\in L^{2}(\omega,\Sigma^{-}_{T})$, and $g\in L^{2}(\omega,Q_{T})$, there exists a unique solution $f$ in $\VoneT$  to \eqref{eq:ToyEquation}. In addition, this solution
admits trace functions $\gamma^{\pm}(f)$ in $L^{2}(\omega,\Sigma^{\pm}_{T})$ and, for all $t\in(0,T]$,
\begin{equation}\label{eq:LinearGreenIdentity}
\begin{aligned}
&\Vert f(t)\Vert^{2}_{L^{2}(\omega,\Dd\times\er^{d})}+ {\sigma^{2}}\Vert \nabla_{u}f\Vert^{2}_{L^{2}(\omega,Q_{t})}
+\Vert \gamma^{+}(f)\Vert^{2}_{L^{2}(\omega,\Sigma^{+}_{t})}\\
&=\Vert \rho_{0}\Vert^{2}_{L^{2}(\omega,\Dd\times\er^{d})}+\Vert q\Vert^{2}_{L^{2}(\omega,\Sigma^{-}_{t})}{+}\int_{Q_{t}}\left\{\frac{\sigma^{2}}{2}\triangle\omega+(\nabla_{u}\omega\cdot B)\right\}|f|^{2}+{2}\int_{Q_{t}}\omega g f.
\end{aligned}
\end{equation}

When $g=0$,  if $\rho_{0}$ and $q$ are nonnegative, then  $f$ and $\gamma^{+}(f)$ are nonnegative.
\end{lemma}
\begin{proof}

In our situation of weighted spaces, it is easy to deduce from the original proof of Carrillo \cite{carrillo-98}, that there exists a unique solution $f\in\Hh(\omega,Q_{T})$ to \eqref{eq:ToyEquation} and that $\sqrt{\omega} f\in \mathcal{Y}(Q_T)$ (for the sake of completeness, the proof of this well-posedness result is given in the Appendix  \ref{subsec:Annexe_EDP_4}).

Then, using Lemma \ref{lem:Derive_Carrillo}, one can take $\phi=\psi=\sqrt{\omega}f$ in the Green formula \eqref{Greenformula} and combined with \eqref{eq:ToyEquationInterior}, we obtain that:
\begin{align*}
&2(\Tt(\sqrt{\omega}f),\sqrt{\omega}f)_{\Hh'(Q_{t}),\Hh(Q_{t})}\\
&=\Vert f(t)\Vert^{2}_{L^{2}(\omega,\Dd\times\er^{d})}-\Vert f(0)\Vert^{2}_{L^{2}(\omega,\Dd\times\er^{d})}+
\Vert \gamma^{+}(f)\Vert^{2}_{L^{2}(\omega,\Sigma^{+}_{T})}-\Vert \gamma^{-}(f)\Vert^{2}_{L^{2}(\omega,\Sigma^{-}_{T})}\\
&=-{\sigma^{2}}\Vert \nabla_{u}f\Vert^{2}_{L^{2}(\omega,Q_{t})}{+}\int_{Q_{t}}(\frac{\sigma^{2}}{2}\triangle\omega+(\nabla_{u}\omega\cdot B))|f|^{2}+{2}\int_{Q_{t}}\omega g f.
\end{align*}
Using \eqref{eq:ToyEquationInitial} and \eqref{eq:ToyEquationBoundary}, we deduce \eqref{eq:LinearGreenIdentity}.

In order to conclude that $f\in \VoneT$, it remains to show the continuity of the mapping $t\mapsto f(t)$ in $L^{2}(\omega,\Dd\times\er^{d})$. Let us start by establishing the right continuity.
For $0<h<T$, we define  $f_{h}(t,x,u):=f(t+h,x,u)$  on $Q_{T-h}$. Observe that, for any $\chi:t\in[0,+\infty)\mapsto\chi(t)\in [0,+\infty)$ in $\Cc^{1}_{b}([0,+\infty))$, by \eqref{eq:ToyEquationInterior}, \eqref{eq:ToyEquationInitial}, and \eqref{eq:ToyEquationBoundary}, one can check that
$\chi(f_{h}-f)\in \Hh(\omega,Q_{t-h})$, $\trans(\chi(f_{h}-f))\in \Hh'(Q_{T-h})$ and
\begin{align*}
&\Tt(\chi(f_{h}-f))=\frac{\sigma^{2}}{2}\triangle_{u}\chi(f_{h}-f)-\left(\nabla_{u}\cdot B \chi(f_{h}-f)\right)+\chi(g_{h}-g)+\chi'(f_{h}-f),\,\mbox{ in }\,\Hh'(Q_{T-h}),\\
&(\chi(f_{h}-f))(0,x,u)=\chi(0)(f(h,x,u)-f_{0}(x,u)),\,\mbox{  on }\,\Dd\times\er^{d},\\
&\gamma^{-}(\chi(f_{h}-f))(t,x,u)=\chi(t)(q_{h}-q)(t,x,u),\,\mbox{  on }\,\Sigma^{-}_{T-h}.
\end{align*}
Since $\chi(f_{h}-f)\in\Hh(\omega,Q_{T-h})$, using \eqref{eq:LinearGreenIdentity}, one obtains that, for all $t\in(0,T)$
\begin{equation}
\begin{aligned}
&\Vert\chi(t)(f_{h}-f)(t) \Vert^{2}_{L^{2}(\omega,\Dd\times\er^{d})}-\Vert \chi(0)(f(h)-f_{0}) \Vert^{2}_{L^{2}(\omega,\Dd\times\er^{d})}\\
& +
{\sigma^{2}}\Vert \nabla_{u} \chi(f_{h}-f) \Vert^{2}_{L^{2}(\omega,Q_{t+h})}+\Vert\chi(\gamma^{+}(f_{h})-\gamma^{+}(f))\Vert^{2}_{L^{2}(\omega,\Sigma^{+}_{t+h})}\\
&\quad =\Vert\chi(q_{h}-q)\Vert^{2}_{L^{2}(\omega,\Sigma^{-}_{t-h})}+\int_{Q_{t+h}}(\frac{\sigma^{2}}{2}\triangle\omega+(\nabla_{u}\omega\cdot B))|\chi(f_{h}-f)|^{2}\\
&\quad \quad+\int_{Q_{t+h}} \chi'\chi\omega|(f_{h}-f)|^{2}+2\int_{Q_{t+h}}\omega \chi^2(g_{h}-g)(f_{h}-f).
\end{aligned}
\end{equation}
Since $\chi$ and $\chi'$ are bounded, by using Corollary \ref{coro:LebesgueContinuity} and the estimations on $\omega$ and its derivatives  in Lemma \ref{lem:major_poids}, all the terms  above with an integral in time tend to 0 when $h$ goes to $0$. Hence  we have, for a fixed $t\in(0,T]$,
\begin{equation*}
\lim_{h\rightarrow 0^{+}}\left|\Vert\chi(t)(f(t+h)-f(t))
\Vert_{L^{2}(\omega,\Dd\times\er^{d})}-\Vert\chi(0)(f(h)-f_{0})\Vert_{L^{2}(\omega,\Dd\times\er^{d})}\right|=0.
\end{equation*}
We get the right continuity at time $t$, by choosing
$\chi(0) =0$ and $\chi(t)=1$. The continuity at time $t=0$
is given by taking $\chi(0)=1$ and $\chi(t)=0$. The left continuity
is proved in an analogous way.
\end{proof}
\subsection{The Maxwellian bounds for the linear Vlasov-Fokker-Planck equation}\label{subsec:LinearPDE}
We state the existence of lower and upper-bounds for the solution in $\VoneT$ to the linear problem
\begin{equation}\label{VFPlineaire_Dirichlet}
\left\{
\begin{aligned}
&\trans(S) + (B^{}\cdot\nabla_{u}S)
-\frac{\sigma^{2}}{2}\triangle_{u}S = 0,~\mbox{ in }~\Hh'(Q_{T}),\\
&S(0,x,u) = \rho_{0}(x,u),~\mbox{ on }\Dd \times\mathbb{R}^{d},\\
&\gamma^-(S)(t,x,u) = q(t,x,u),~\mbox{ on }\Sigma^{-}_{T}.
\end{aligned}
\right.
\end{equation}
\begin{proposition}\label{prop:EncadrLineaire} Assume \hypedpii~ and \hypedpiii.
For $B^{}\in L^{\infty}((0,T)\times\Dd;\er^{d})$, for $(\underline{P}_{0},\overline{P}_{0})$ as in \hypedpiii, let $\left(\underline{p},\overline{p}\right)$ be
a couple of Maxwellian distributions with parameters
$\left(\underline{a},\underline{\mu},\underline{P}_{0}\right)$ and
$\left(\overline{a},\overline{\mu},\overline{P}_{0}\right)$
satisfying:
\item[$(a{1})$] $\underline{\mu}>1$,  and~ $\overline{\mu}\in (\frac{1}{2},1)$.
\item[$(a{2})$] ${\displaystyle \underline{a}\leq
\frac{-\underline{\mu}}{2\sigma^{2}(\underline{\mu}-1)}
\Vert B^{}\Vert_{L^{\infty}((0,T)\times\Dd;\er^{d})}^{2}}$, and~
${\displaystyle \overline{a}\geq \frac{\overline{\mu}}{2\sigma^{2}(1-\overline{\mu})}\Vert B^{}\Vert_{L^{\infty}((0,T)\times\Dd;\er^{d})}^{2}}$.
\item[~] Then the following properties hold:
\item[$(d{1})$] ${\displaystyle \sup_{t\in [0,T]}\int_{\er^{d}}
(1+|u|)\omega(u)\left|\overline{p}(t,u)\right|^{2}\,du<+\infty}$, \quad
and  \quad${\displaystyle \inf_{t\in [0,T]}\int_{\er^{d}}\underline{p}(t,u)du >0}$.
\item[$(d{2})$] Let $S$ be the unique weak solution of
\eqref{VFPlineaire_Dirichlet} with inputs $\rho_{0},~q$ and
$B^{}$. If $\underline{p}\leq q\leq \overline{p}$, $\lambda_{\Sigma_{T}}$-a.e.
on $\Sigma^{-}_{T}$, then $\underline{p}\leq S \leq \overline{p}$, a.e.
on $Q_{T}$, and $\underline{p}\leq \gamma^{+}(S)\leq \overline{p}$, $\lambda_{\Sigma_{T}}$-a.e.
on $\Sigma^{+}_{T}$.
\end{proposition}
For the proof of Proposition \ref{prop:EncadrLineaire}, we will use  the notions of
 super-solution and sub-solution of Maxwellian type related to the operator
\begin{equation}\label{Def:LinearDiffOperator}
\Ll_{B^{}}(\psi)=\trans(\psi)+(B^{}\cdot \nabla_{u}\psi)
-\frac{\sigma^{2}}{2}\triangle_{u}\psi.
\end{equation}
\begin{definition}
\label{sursol_soussol_def}
Let $P$ be a Maxwellian distribution with parameters $(a,\mu,P_{0})$.
For $B^{}\in L^{\infty}((0,T)\times\Dd;\er^{d})$, we say that:
\begin{enumerate}
\item $P$ is a super-solution of Maxwellian type for $\Ll_{B^{}}$
if $0 \leq \Ll_{B^{}}(P)< +\infty$, a.e. on $Q_{T}$.
\item $P$ is a sub-solution of Maxwellian type for
$\Ll_{B^{}}$ if $-\infty <\Ll_{B^{}}(P)\leq 0$, a.e. on $Q_{T}$.
\end{enumerate}
\end{definition}
The proof of Proposition \ref{prop:EncadrLineaire} proceeds as follows.
In Step 1, we exhibit a class of Maxwellian distributions satisfying  $(d{1})$
and some regularity properties (see Lemma \ref{lem:GoodMaxwellian}). In Step 2, we establish a comparison
principle between super-solutions and sub-solutions of Maxwellian type  and the weak solution to \eqref{VFPlineaire_Dirichlet}
(see Lemma \ref{lem:LinearMaxwellianBounds}). We thus deduce a particular class of Maxwellian distributions satisfying
the properties $(d{1})$ and $(d{2})$.
In Step 3, we identify the class of
super-solutions and sub-solutions  of Maxwellian type
for the operator $\Ll_{B^{}}$ for all $B^{}$ fixed
in $L^{\infty}((0,T)\times\Dd;\er^{d})$ (see Lemma \ref{lem:ConstructMaxwellianBounds}). In Step 4, combining these results, we conclude on Proposition \ref{prop:EncadrLineaire}.
\paragraph{Step 1.}
We start by emphasizing some technical properties of the Maxwellian distributions.
\begin{lemma}\label{lem:GoodMaxwellian}
Let $\mathit{p}$ be a Maxwellian distribution equipped with the parameters $\left(a,\mu,\mathit{p}_{0}
\right)$ such that $2\mu>1$, $p_0$ is not identically equal to zero, and
\begin{align}\label{Integrabilite_Cond_initial}
\int_{\er^{d}}(1+|u|)\omega(u){p}_{0}^2(u)\,du<+\infty.
\end{align}
Then the following properties hold:
\item[$(i{1})$] $\sup_{t\in [0,T]}\int_{\er^{d}}(1+|u|)\omega(u)\left|\mathit{p}(t,u)\right|^{2}\,du<+\infty$.

\item[$(i{2})$]
$\inf_{t\in [0,T]}\int_{\er^{d}}\mathit{p}(t,u)du >0$, if $\mu>1$.

\item[$(i{3})$] There exists a sequence of positive reals $\{\epsilon_{k}~;~k\in\mathbb{N}\}$ such that
\[\lim_{\substack{k\rightarrow +\infty}}\epsilon_{k}=0~\mbox{and}
\lim_{\substack{k\rightarrow +\infty}}{\mathit{p}}(\epsilon_{k},\cdot)= \mathit{p}_{0}(\cdot), ~\mbox{ in}~L^{2}(\er^{d}).\]

\item[$(i{4})$] For all $\delta>0$, $\partial_{t}{\mathit{p}}$  belongs to $L^{2}((\delta,T)\times\er^{d})$.

\item[$(i{5})$] $\mathit{p}
\in \Hh(Q_T)$.
\end{lemma}
The proof of this lemma  relies on some well-known properties of Gaussian distributions  and is postponed in the Appendix \ref{sec:appendix-proof}.
Lemma \ref{lem:GoodMaxwellian} enables us to identify the class of Maxwellian distributions satisfying
 $(d{1})$ in Proposition \ref{prop:EncadrLineaire}.
The properties $((i{3}),(i{4}),(i{5}))$ emphasize regularities that we will need in the sequel.

\paragraph{Step 2. Comparison principle and Maxwellian bounds.}
\begin{lemma}\label{lem:LinearMaxwellianBounds}
Let $B^{}\in L^{\infty}((0,T)\times\Dd;\er^{d})$ be fixed and let
$\underline{p},~\overline{p}$ be two Maxwellian distributions, sub-solution and super-solution for $\Ll_{B^{}}$ respectively
with parameters
$\left(\underline{a},\underline{\mu},\underline{\mathit{p}}_{0}\right)$ and
$\left(\overline{a},\overline{\mu},\overline{\mathit{p}}_{0}\right)$
such that $2\underline{\mu} \wedge 2\overline{\mu}>1$ and
$\underline{\mathit{p}}_{0},~\overline{\mathit{p}}_{0}$ satisfying \eqref{Integrabilite_Cond_initial}.
If $\underline{p}_{0}\leq \rho_{0}\leq \overline{p}_{0}$,
a.e. on $\Dd\times\er^{d}$, and
$\underline{p}\leq q\leq \overline{p}$, $\lambda_{\Sigma_{T}}$-a.e.
on $\Sigma^{-}_{T}$, then we have
\begin{align}\label{Major_Minor_Maxwell_step_0}
&\underline{p}\leq S\leq \overline{p},~\mbox{a.e. on}~Q_{T},\\
&\underline{p}\leq\gamma^{+}(S)\leq \overline{p},~\lambda_{\Sigma_{T}}\mbox{-a.e. on}~\Sigma^{+}_{T},
\end{align}
for $S$ the weak solution  in $\VoneT$ to \eqref{VFPlineaire_Dirichlet} with inputs $(\rho_0,q,B)$.
\end{lemma}
\begin{proof}
Let us first prove the implication for the upperbounds
\begin{equation}\label{MajorMaxwell}
\begin{aligned}
\left\{
\begin{array}{l}
\rho_{0}\leq \overline{p}_{0},~\mbox{a.e. on}~(0,T)\times\Dd,
\\
q\leq\overline{p},~\lambda_{\Sigma_{T}}\mbox{-a.e. on}~\Sigma^{-}_{T}.
\end{array}
\right.
\quad \Longrightarrow \quad
\left\{
\begin{array}{l}
S\leq \overline{p},~\mbox{a.e. on}~Q_{T},
\\
\gamma^{+}(S)\leq \overline{p},~\lambda_{\Sigma_{T}}\mbox{-a.e. on}~\Sigma^{+}_{T}.
\end{array}
\right.
\end{aligned}
\end{equation}
Defining ${F}$, $\gamma^{+}({F})$  with
\begin{align*}
&{F}(t,x,u)=\left(\overline{p}(t,u)-S(t,x,u)\right),
\\
&\gamma^{+}({F})(t,x,u)=\left(\overline{p}(t,u)-
\gamma^{+}(S)(t,x,u)\right),
\end{align*}
 \eqref{MajorMaxwell} is equivalent to the inequality: $\forall\,t\in(0,T]$,
\begin{equation}\label{pr:LinearMaxwellBounds}
\begin{aligned}
\Vert\left({F}(t)\right)^{-}\Vert^{2}_{L^{2}(\Dd\times\er^{d})}
+\Vert\left(\gamma^{+}({F})\right)^{-}\Vert^{2}_{L^{2}(\Sigma^{+}_{t})}
\leq \Vert \left(\overline{p}_{0}-\rho_{0}\right)^{-}
\Vert^{2}_{L^{2}(\Dd\times\er^{d})}+
\Vert \left(\overline{p}-{q}\right)^{-}\Vert^{2}_{L^{2}(\Sigma^{-}_{t})}.
\end{aligned}
\end{equation}
In order to obtain \eqref{pr:LinearMaxwellBounds}, we  establish a Green identity
on a smooth approximation of $({F})^{-}$ (where $(F)^-$ refers to the negative part of $F$). For fixed $t$ in $(0,T]$,
by Lemma \ref{lem:Density},
there exists a sequence of $\Cc^{\infty}_{c}(\overline{Q_{t}})$-functions $\left\{f_{n};\,n\in\mathbb{N}\right\}$, such that
\begin{align}\label{Limits_inegalite}
\begin{array}{ll}
&{\displaystyle \lim_{\substack{n\rightarrow +\infty}}f_{n}=
S,\,\mbox{ in}~\mathcal{H}(Q_{t}),\quad ~\lim_{\substack{n\rightarrow +\infty}}\trans(f_{n})=
\trans(S), ~\mbox{ in}~\mathcal{H}'(Q_{t})},\\
&{\displaystyle\lim_{\substack{n}\rightarrow+\infty}\int_{\Dd\times\er^{d}}
\left|f_{n}(s,x,u)-S(s,x,u)\right|^{2} \,dx\,du=0,~\forall~s\in [0,t],}\\
&{\displaystyle\lim_{\substack{n}\rightarrow+\infty}
\int_{\Sigma^{\pm}_{t}}|(u\cdot n_{\Dd}(x))|\left|f_{n}(s,x,u)-\gamma^{\pm}(S)(s,x,u)\right|^{2} d\lambda_{\Sigma_{T}}(s,x,u)=0.}
\end{array}
\end{align}
In addition, according to Lemma \ref{lem:GoodMaxwellian}, $\overline{p}$
satisfies $(i{1})$ to $(i{5})$. Let us define
the sequence of $\Cc^{2}_{b}(Q_{t})$-functions $\left\{{F}_{n};\,n\in\mathbb{N}\right\}$ by
\begin{align*}
&{F}_{n}(s,x,u)=\overline{p}(s,x,u)-f_{n}(s,x,u).
\end{align*}
Then by definition of $\Ll_{B}$ in \eqref{Def:LinearDiffOperator}, for a.e. $(s,x,u)$ in $Q_{t}$,
\begin{align}\label{inegalite_comparaison_step1}
\trans({F}_{n})(s,x,u) & = \Ll_{B^{}}(\overline{p}-f_{n})(s,x,u)-\left({B}^{}(s,x)\cdot\nabla_{u}{F}_{n}(s,x,u)\right)
+\frac{{\sigma}^{2}}{2}\triangle_{u}{F}_{n}(s,x,u)\nonumber \\
& \geq -\Ll_{B^{}}(f_{n})(s,x,u)-\left({B}^{}(s,x)\cdot \nabla_{u}{F}_{n}(s,x,u)\right)+\frac{{\sigma}^{2}}{2}
\triangle_{u}{F}_{n}(s,x,u)
\end{align}
since $\Ll_{B^{}}(\overline{p}-f_{n}) \geq -\Ll_{B^{}}(f_{n})$ as $\overline{p}$ is a super-solution for $\Ll_{B}$.
Using the sequence $\left\{\epsilon_{k};~k\in\mathbb{N}\right\}$
given by $(i{3})$ and by taking $k$ such that $0<\epsilon_{k} \leq t$,  $(i{4})$ and $(i{5})$ ensure that
 $\trans\left({F}_{n}\right)\in L^{2}(Q_{\epsilon_{k},t})$
and ${F}_{n}\in \mathcal{H}\left(Q_{\epsilon_{k},t}\right)$
for $Q_{\epsilon_{k},t}:=(\epsilon_{k},t)\times\Dd\times\er^{d}$. These properties are also true for  $\left({F}_{n}\right)^{-}$ (see Theorem \ref{thm:SobolevRegNegativePart}).
Multiplying both sides of \eqref{inegalite_comparaison_step1} by
$\left({F}_{n}\right)^{-}$, and integrating the resulting expression on $Q_{\epsilon_{k},t}$, we obtain
\begin{align}
\label{inegalite_comparaison_step2}
\int_{Q_{\epsilon_{k},t}}\trans({F}_{n})\left({F}_{n}\right)^{-}\geq-\int_{Q_{\epsilon_{k},t}}\left({B}^{}\cdot \nabla_{u}{F}_{n}\right)\left({F}_{n}\right)^{-}
+\frac{{\sigma}^{2}}{2}
\int_{Q_{\epsilon_{k},t}}
\left(\triangle_{u}{F}_{n}\right)\left({F}_{n}\right)^{-} -\int_{Q_{\epsilon_{k},t}}\Ll_{B^{}}(f_{n})
\left({F}_{n}\right)^{-}.
\end{align}
For the l.h.s. in \eqref{inegalite_comparaison_step2}, an integration by parts yields
\begin{align*}
&\int_{Q_{\epsilon_{k},t}}\trans({F}_{n})\left({F}_{n}\right)^{-} = - \int_{Q_{\epsilon_{k},t}}\trans(({F}_{n})^-)\left({F}_{n}\right)^{-}
 =-\int_{Q_{\epsilon_{k},t}}
\left\{\partial_{t}\left({F_{n}}\right)^{-} + \left(u\cdot\nabla_{x}
\left({F_{n}}\right)^{-}\right)\right\}\left({F}_{n}\right)^{-}
\\
&=\frac{1}{2}\Vert \left({F}_{n}(\epsilon_{k})\right)^{-}\Vert^{2}_{L^2(\Dd \times\er^{d})}
-\frac{1}{2} \Vert \left({F}_{n}(t)\right)^{-}\Vert^{2}_{L^2(\Dd \times\er^{d})}
- \frac{1}{2}\Vert \left({F}_{n}\right)^{-}\Vert^{2}_{L^2(\Sigma^{+}_{\epsilon_{k},t})}
+ \frac{1}{2}\Vert \left({F}_{n}\right)^{-}\Vert^{2}_{L^2(\Sigma^{-}_{\epsilon_{k},t})},
\end{align*}
with $\Sigma^{\pm}_{\epsilon_{k},t}:=\left\{(s,x,u)\in\Sigma^{\pm}_{t}~;~s\in(\epsilon_{k},t)\right\}$.
In the r.h.s. in \eqref{inegalite_comparaison_step2}, as $B$ depends only on $x$,  we get
\begin{align*}
-\int_{Q_{\epsilon_{k},t}}\left({B}^{}\cdot \nabla_{u}{F}_{n}\right)\left({F}_{n}\right)^{-}
 +\frac{\sigma^{2}}{2}\int_{Q_{\epsilon_{k},t}}
\left(\triangle_{u}{F}_{n}\right)\left({F}_{n}\right)^{-}
=\frac{{\sigma}^{2}}{2}\int_{Q_{\epsilon_{k},t}}
\left|\nabla_{u}\left({F}_{n}\right)^{-}\right|^{2}\geq 0.
\end{align*}
Coming back to \eqref{inegalite_comparaison_step2}, it follows that
\begin{align*}
&\Vert\left({F}_{n}(\epsilon_{k})\right)^{-}\Vert^{2}_{L^{2}(\Dd \times\er^{d})} +
\Vert \left({F}_{n}\right)^{-}\Vert^{2}_{L^2(\Sigma^{-}_{\epsilon_{k},t})}
\\
& \geq \Vert\left({F}_{n}(t)\right)^{-}\Vert^{2}_{L^{2}(\Dd \times\er^{d})}
+\Vert \left({F}_{n}\right)^{-}\Vert^{2}_{L^2(\Sigma^{+}_{\epsilon_{k},t})}
- 2\int_{Q_{\epsilon_{k},t}}\Ll_{B^{}}(f_{n})\left({F}_{n}\right)^{-}.
\end{align*}
Taking the limit $k\rightarrow +\infty$, $(i{3})$
implies that ${\displaystyle \lim_{\substack{k\rightarrow+\infty}} \left|({F}_{n}(\epsilon_{k}))^{-}\right|=\left|\left(\overline{p}_{0}-f_{n}(0)\right)^{-}\right|}$.
Thus
\begin{equation}
\label{Limits_inegalite_1}
\begin{aligned}
&\Vert\left(\overline{p}_{0}-f_{n}(0)\right)^{-}
\Vert^{2}_{L^{2}(\Dd \times\er^{d})}
+ \Vert\left({F}_{n}\right)^{-}\Vert^{2}_{L^{2}(\Sigma^{-}_{t})}
\\
&\geq\Vert\left({F}_{n}(t)\right)^{-}\Vert^{2}_{L^{2}(\Dd \times\er^{d})}
+ \Vert\left({F}_{n}\right)^{-}\Vert^{2}_{L^{2}(\Sigma^{+}_{t})}
-2\int_{Q_{t}}\Ll_{B^{}}(f_{n})\left({F}_{n}\right)^{-}
\end{aligned}
\end{equation}
It remains to study the limit w.r.t. $n$. By \eqref{Limits_inegalite},
\begin{equation*}
\label{Limits_inegalite_2}
\begin{aligned}
&\lim_{\substack{n\rightarrow+\infty}} \left(\Vert \left(\overline{p}_{0}-f_{n}(0)\right)^{-}\Vert^{2}_{L^{2}(\Dd \times\er^{d})}
+ \Vert \left({F}_{n}\right)^{-}\Vert^{2}_{L^{2}(\Sigma^{-}_{t})}\right)=\Vert \left(\overline{p}_{0}-\rho_{0}\right)^{-}\Vert^{2}_{L^{2}(\Dd\times\er^{d})}+\Vert \left(\overline{p}-{g}\right)^{-}\Vert^{2}_{L^{2}(\Sigma^{-}_{t})},
\end{aligned}
\end{equation*}
\begin{align*}
\label{Limits_inegalite_3}
\lim_{\substack{n\rightarrow+\infty}}\left(
\Vert \left({F}_{n}(t)\right)^{-}\Vert^{2}_{L^{2}(\Dd \times\er^{d})}+
\Vert \left({F}_{n}\right)^{-}\Vert^{2}_{L^{2}(\Sigma^{+}_{t})}\right)
 =\Vert\left({F}(t)\right)^{-}\Vert^{2}_{L^{2}(\Dd\times\er^{d})}
+ \Vert \left(\gamma^{+}({F})\right)^{-}\Vert^{2}_{L^{2}(\Sigma^{+}_{t})}.
\end{align*}
For the last term in \eqref{Limits_inegalite_1}, an integration by parts yields
\begin{align*}
& \int_{Q_{t}}
 \Ll_{B^{}}(f_{n})\left(F_{n}\right)^{-}\\
 &=\left(\trans(f_{n}),
 \left(\overline{p}-f_{n}\right)^{-}\right)_{\mathcal{H}'(Q_{t}),\mathcal{H}(Q_{t})}+\int_{Q_{t}}
\left(B^{}\cdot\nabla_{u}f_{n}\right) \left(\overline{p}-f_{n}\right)^{-}
+\frac{\sigma^{2}}{2}\int_{Q_{t}} \left(\nabla_{u}f_{n}
\cdot\nabla_{u}\left(\overline{p}-f_{n}\right)^{-}\right) .
\end{align*}
As $\lim_{n\rightarrow +\infty}\left(\overline{p}-f_{n}\right)^{-}
=\left(\overline{p}-S\right)^{-}$ in $\mathcal{H}(Q_{T})$ and
$\lim_{n\rightarrow +\infty}\trans(f_{n}) =\trans(S)$ in
$\mathcal{H}'(Q_{T})$, we get
\begin{align*}
& \lim_{\substack{n\rightarrow+\infty}}\int_{Q_{t}}
 \Ll_{B^{}}(f_{n})\left(F_{n}\right)^{-} \\
&=\left(\trans(S),
 \left(\overline{p}-S\right)^{-}\right)_{\mathcal{H}'(Q_{t}),
\mathcal{H}(Q_{t})}
+\int_{Q_{t}}
\left(B^{}\cdot\nabla_{u}S\right) \left(\overline{p}-S\right)^{-}
+\frac{\sigma^{2}}{2}\int_{Q_{t}}\left(\nabla_{u}S
\cdot\nabla_{u} \left(\overline{p}-S\right)^{-} \right) = 0
\end{align*}
since  ${\displaystyle \trans(S)+\left(B^{}\cdot \nabla_{u}S\right)-\frac{\sigma^{2}}{2}\triangle_{u}S=0}$ in $\mathcal{H}'(Q_{T})$.
Coming back to \eqref{Limits_inegalite_1}, we take the limit $n\rightarrow +\infty$ to get   \eqref{pr:LinearMaxwellBounds}. The
implication
for the lower-bounds
\begin{equation}\label{MinorMaxwell}
\begin{aligned}
\left\{
\begin{array}{l}
\underline{p}_{0}\leq \rho_{0},~\mbox{a.e. on}~(0,T)\times\Dd,\\
\underline{p}\leq q,~\lambda_{\Sigma_{T}}\mbox{-a.e. on}~\Sigma^{-}_{T}.
\end{array}
\right.
\quad \Longrightarrow \quad
\left\{
\begin{array}{l}
\underline{p}\leq S,~\mbox{a.e. on}~Q_{T},\\
\underline{p}\leq \gamma^{+}(S),~\lambda_{\Sigma_{T}}\mbox{-a.e. on}~\Sigma^{+}_{T}.
\end{array}
\right.
\end{aligned}
\end{equation}
is proved in the same way. By defining
\begin{align*}
&J_{}(t,x,u):= \left(S(t,x,u)-\underline{\mathit{p}}(t,u)\right), \\
&\gamma^{+}(J_{})(t,x,u):=
\left(\gamma^{+}(S)(t,x, u)-
\underline{\mathit{p}}(t, u)\right),
\end{align*}
we can then  establish that
\begin{equation}\label{MinorMaxwell_bis}
\begin{aligned}
&\Vert(\rho_{0}-\underline{\mathit{p}}_{0})^{-}\Vert^{2}_{L^{2}(\Dd\times\er^{d})}+\Vert(q-\underline{\mathit{p}}_{})^{-}
\Vert^{2}_{L^{2}(\Sigma^{-}_{t})}\geq
\Vert\left(J_{}(t)\right)^{-}
\Vert^{2}_{L^{2}(\Dd\times\er^{d})}
+ \Vert\left(\gamma^{+}(J_{})\right)^{-}\Vert^{2}_{L^{2}(\Sigma^{+}_{t})},
\end{aligned}
\end{equation}
from which we conclude on \eqref{MinorMaxwell}.
The inequality \eqref{MinorMaxwell_bis} is proved
by using the sequence
\begin{equation*}
J_{n}(s,x,u):= f_{n}(s,x, u)-
\underline{p}(s, u),
\end{equation*}
with $\left\{f_{n};\,n\in\mathbb{N}\right\}$
satisfying \eqref{Limits_inegalite},
and using the fact that $\underline{p}$ is a sub-solution of Maxwellian type,
we obtain: for all $t\in(0,T]$ fixed, for a.e. $(s,x,u)$ in $Q_{t}$, it holds
\begin{equation*}
\trans(J_{n})(s,x,u) \geq
\Ll_{B^{}}(f_{n})(s,x, u)-\left({B}^{}(s,x)\cdot \nabla_{u}J_{n}(s,x,u)\right)+\frac{{\sigma}^{2}}{2}
\triangle_{u}J_{n}(s,x,u).
\end{equation*}
Replicating the arguments in \eqref{MajorMaxwell}, we  get
\begin{align*}
&\Vert(f_{n}(0)-\underline{p}_{0})^{-}\Vert^{2}_{L^{2}(\Dd \times\er^{d})}
+\Vert\left(J_{n}\right)^{-}\Vert^{2}_{L^{2}
(\Sigma^{-}_{t})}\\
& \geq \Vert\left(J_{n}(t)\right)^{-}
\Vert^{2}_{L^{2}(\Dd\times\er^{d})}+\Vert\left(J_{n}\right)^{-}
\Vert^{2}_{L^{2}(\Sigma^{+}_{t})}
+2\int_{Q_{\epsilon_{k},t}} \Ll_{B^{}}(f_{n})
\left(J_{n} \right)^{-}.
\end{align*}
We then obtain \eqref{MinorMaxwell_bis} by taking the limit $n\rightarrow+\infty$.
\end{proof}
\paragraph{Step 3. Existence of sub- and super-solutions of Maxwellian type.}
\begin{lemma}\label{lem:ConstructMaxwellianBounds}
Let $\mathit{p}$ be a Maxwellian distribution with parameters
$(a,\mu,\mathit{p}_{0})$. For $B^{}\in
L^{\infty}((0,T)\times\Dd;\er^{d})$, let $\mathcal{L}_{B^{}}$ be the operator defined in \eqref{Def:LinearDiffOperator}.
Then the following properties hold.
\begin{description}
\item[$(i)$] If $\mu\in (0,1)$ and ${\displaystyle a\geq \frac{\mu}{2\sigma^{2}(1-\mu)}\|B^{}\|_{L^{\infty}((0,T)\times\Dd;\er^{d})}^{2}}$,
then $\mathit{p}$ is a super-solution  for $\Ll_{B^{}}$.
\item[$(ii)$] If $\mu~>~1$ and ${\displaystyle a\leq \frac{-\mu}{2\sigma^{2}(\mu-1)}\|B^{}\|_{L^{\infty}((0,T)\times\Dd;\er^{d})}^{2}}$,
then $\mathit{p}$ is a sub-solution  for $\Ll_{B^{}}$.
\end{description}
\end{lemma}
\begin{proof}
By using the explicit form  \eqref{defMaxwell} of the considered Maxwellian distribution, we have
\begin{align*}
&\Ll_{B^{}}(\mathit{p})(t,x,u)
\\
&= a\exp\left\{a t\right\}m^{\mu}(t,u)
+\mu\exp\left\{a t\right\}\left(\partial_{t}m(t,u)
-\frac{\sigma^{2}}{2}\triangle_{u}m(t,u)\right)m^{\mu-1}(t,u)
\\
&\quad + \mu \exp\left\{a t\right\}\left(B^{}(t,x)\cdot \nabla_{u}m(t,u)\right)
m^{\mu-1}(t,u) -\frac{\sigma^{2}}{2}\mu(\mu-1)\exp\left\{a t\right\}
\left|\nabla_{u}m(t,u)\right|^{2}m^{\mu-2}(t,u).
\end{align*}
Since $m$ is a classical solution of the heat equation, the previous equality reduces to
\begin{align*}
&\Ll_{B^{}}(\mathit{p})(t,x,u)\\
&=
\exp\left\{a t\right\}m^{\mu-2}(t,u)\left[a\,m^{2}(t,u)-\frac{\sigma^{2}}{2}\mu(\mu-1)|\nabla_{u}m(t,u)|^{2}
+\frac{}{}\mu\,m(t,u)\left(B^{}(t,x)\cdot \nabla_{u}m(t,u)\right)\right].
\end{align*}
The sign of $\Ll_{B^{}}(\mathit{p})$ is thus determined by the function:
\begin{equation}
\label{lem:ConstructMaxwellianBounds_step1}
J(t,x,u):=a \,m^2(t,u)-\frac{\sigma^{2}\mu(\mu-1)}{2}|\nabla_{u}m(t,u)|^{2}
+\mu m(t,u)\left(B^{}(t,x)\cdot \nabla_{u}m(t,u)\right).
\end{equation}
$\bullet$ When $a$ and $\mu$ satisfy $(i)$, using the identity $\left(u_{1} \cdot u_{2}\right) = \frac{1}{2}
\left|\epsilon u_{1} + \frac{u_{2}}{\epsilon}\right|^{2}
-\frac{\epsilon^{2}\left|u_{1}\right|^{2}}{2}-\frac{\left|u_{2}\right|^{2}}{2\epsilon^{2}}$, for $u_{1},~u_{2}\in\er^{d}$, $\epsilon >0$,  we have
\begin{align*}
&\mu\,m(t,u)\left(B^{}(t,x)\cdot \nabla_{u}m(t,u)\right)\\
&=\frac{1}{2}\left|\epsilon \nabla_{u}m(t,u) + \frac{1}{\epsilon} m(t,u)
B^{}(t,x)\right| ^{2}
-\frac{1}{\epsilon^2}m^2(t,u) |B^{}(t,x)|^{2}
-\epsilon^2 |\nabla_{u}m(t,u)|^{2}.
\end{align*}
Inserting this equality into \eqref{lem:ConstructMaxwellianBounds_step1}, with ${\displaystyle \epsilon = \frac{\sigma\sqrt{1-\mu}}{\sqrt{\mu}}}$ ($>0$ since $0<\mu<1$), it follows that
\begin{align*}
&J(t,x,u)\\
&=\left(a-\frac{\mu}{2\sigma^{2}(1-\mu)}|B^{}(t,x)|^{2}\right)\left|m(t,u)\right|^{2}
+\frac{1}{2}\left|\sigma\sqrt{\mu(1-\mu)}\nabla_{u}m(t,u) +\frac{\sqrt{\mu}}{\sigma\sqrt{1-\mu}}m(t,u)B^{}(t,x)
\right| ^{2},
\end{align*}
where, under $(i)$,
\begin{equation*}
a-\frac{\mu}{2\sigma^{2}(1-\mu)}|B^{}(t,x)|^{2}\geq a-\frac{\mu}{2\sigma^{2}(1-\mu)}\Vert B\Vert^{2}_{L^{\infty}((0,T)
\times\Dd;\er^{d})}\geq 0.
\end{equation*}
We thus deduce that $J$ (and consequently $\Ll_{B}(\mathit{p})$) is nonnegative in the situation $(i)$.
\newline\noindent
$\bullet$ When $a$ and $\mu$ satisfy $(ii)$, using the identity $
\left(u_{1}\cdot u_{2}\right) = -\frac{1}{2}\left|\epsilon u_{1} -\frac{u_{2}}{\epsilon}\right|^{2}+\frac{\epsilon^{2}\left|u_{1}\right|^{2}}{2}+
\frac{\left|u_{2}\right|^{2}}{2\epsilon^{2}}$, for $u_{1},~u_{2}\in\er^{d}$, $\epsilon >0$, we get
\begin{align*}
&\mu\,m(t,u)\left(B^{}(t,x)\cdot \nabla_{u}m(t,u)\right)\\
&= \frac{1}{2\epsilon^{2}}|B^{}(t,x)|^{2} \left|m(t,u)\right|^{2}
+\frac{\epsilon^{2}}{2}\mu^{2}|\nabla_{u}m(t,u)|^{2}
 -\frac{1}{2}\left| \epsilon
m(t,u) B^{}(t,x)-
\frac{\mu}{\epsilon}\nabla_{u}m(t,u)\right| ^{2}.
\end{align*}
Taking ${\epsilon = \frac{\sigma\sqrt{\mu-1}}{\sqrt{\mu}}}$ it follows that
\begin{align*}
J(t,x,u)=\left(a+\frac{\mu}{2\sigma^{2}(\mu-1)}|B^{}(t,x)|^{2}\right)
m^2(t,u)-\frac{1}{2}\left|  \frac{\sigma\sqrt{\mu-1}}{\sqrt{\mu}}
m(t,u)B^{}(t,x)-
\frac{\mu\sqrt{\mu}}{\sigma\sqrt{\mu-1}}\nabla_{u}m(t,u)\right| ^{2}.
\end{align*}
As $(ii)$ ensures that
\begin{equation*}
a+\frac{\mu}{2\sigma^{2}(1-\mu)}|B^{}(t,x)|^{2}
\leq a+\frac{\mu}{2\sigma^{2}(1-\mu)}\Vert B^{}\Vert^{2}_{L^{\infty}((0,T)\times\Dd;\er^{d})}\leq 0,
\end{equation*}
 we conclude that $\Ll_{B}(\mathit{p})$ is non-positive.
\end{proof}
\paragraph{Step 4. Proof of Proposition \ref{prop:EncadrLineaire}.}
Let $\left(\underline{\mathit{p}},\overline{\mathit{p}}\right)$ be a couple of Maxwellian distributions of parameters $(\underline{a},\underline{\mu},\underline{\mathit{P}}_{0})$ and $(\overline{a},\overline{\mu},\overline{\mathit{P}}_{0})$
such that $(\underline{\mathit{P}}_{0},\overline{\mathit{P}}_{0})$ satisfy  \hypedpiii~ and such that
 $(\underline{a},\underline{\mu})$ and $(\overline{a},\overline{\mu})$ satisfy the
properties $(a{1})$ and $(a{2})$ in Proposition
\ref{prop:EncadrLineaire}. Applying Lemma \ref{lem:ConstructMaxwellianBounds},
$\left(\underline{\mathit{p}},\overline{\mathit{p}}\right)$ are  respectively
a sub-solution and a super-solution of Maxwellian type for the linear operator. Moreover, recalling
that $\underline{\mathit{P}}_{0}$, $\overline{\mathit{P}}_{0}\in
L^{2}\left(\omega,Q_{T}\right)$ are positives,  these Maxwellian distributions satisfy the conditions of Lemma \ref{lem:GoodMaxwellian} and Lemma \ref{lem:LinearMaxwellianBounds}, and then
 $\left(\underline{\mathit{p}},\overline{\mathit{p}}\right)$
satisfy the properties $(d{1}),(d{2})$ of Proposition \ref{prop:EncadrLineaire}.
\subsection{Construction of a weak solution to the conditional McKean-Vlasov-Fokker-Planck equation}\label{subsec:NonLinMKFP}
\subsubsection{Introduction of the specular boundary condition}
We consider now a linear problem
endowing a given convection term
$B^{}\in L^{\infty}((0,T)\times\Dd;\er^{d})$ and submitted to the specular boundary condition:
\begin{equation}\label{eq:LinearVFPSpec}
\left\{
\begin{aligned}
&\trans(S) + (B^{}\cdot\nabla_{u}S)
-\frac{\sigma^{2}}{2}\triangle_{u}S = 0,\,\mbox{ in}~\Hh'(Q_{T}),\\
&S(0,x,u)=\rho_{0}(x,u),~\mbox{on}~\Dd \times\mathbb{R}^{d},\\
&\gamma^-(S)(t,x,u)=\gamma^+(S)(t,x,u-2(u\cdot n_{\Dd}(x))n_{\Dd}(x)),~\mbox{on}~\Sigma^{-}_{T}.
\end{aligned}
\right.
\end{equation}
\begin{proposition}\label{prop:VFPlineairespec}
Assume \hypedpii,~\hypedpiii, and $B^{}\in L^{\infty}((0,T)\times\Dd;\er^{d})$.
Let $\left(\underline{P},\overline{P}\right)$ be a couple Maxwellian distributions
 with parameters
$\left(\underline{a},\underline{\mu},\underline{P}_{0}\right)$
and $\left(\overline{a},\overline{\mu},\overline{P}_{0}\right)$ respectively,
satisfying the hypotheses of Proposition
\ref{prop:EncadrLineaire}. Then there exists a unique weak solution
$\SSS$ in $\VoneT$
of \eqref{eq:LinearVFPSpec} such that
\begin{align*}
&\underline{P}(t,u)\leq \SSS (t,x,u)\leq\overline{P}(t,u),\mbox{ for a.e. }(t,x,u)\in Q_{T},\\
&\underline{P}(t,u)\leq \gamma^{\pm}(\SSS)(t,x,u)\leq\overline{P}(t,u),\quad \mbox{for }\,\lambda_{\Sigma_{T}}\mbox{-a.e. }(t,x,u)\in\Sigma^{\pm}_{T}.
\end{align*}
\end{proposition}
\begin{proof} For the existence claim, let us introduce the functional space
\begin{equation*}
E=\left\{\psi\in  V_{1}\left(\omega,Q_{T}\right)~;~\psi~\mbox{admits trace functions}~\gamma^{\pm}(\psi)~\mbox{on}~\Sigma^{\pm}_{T}~\mbox{belonging to}\,L^{2}({\omega,}\Sigma^{\pm}_{T})
\right\},
\end{equation*}
equipped with the norm
\begin{equation*}
\Vert\psi\Vert^{2}_{E}=\Vert\psi\Vert^{2}_{\VoneT} + \Vert \gamma^{+} (\psi)\Vert^{2}_{L^{2}(\omega,\Sigma^{+}_{T})}+
\Vert \gamma^{-}(\psi)\Vert^{2}_{L^{2}(\omega,\Sigma^{-}_{T})}.
\end{equation*}
For all $f\in E$, we denote by $S^{}(f)$ the unique weak solution (in the sense of Lemma \ref{lem:ExistVFPlineaire}) to the linear Vlasov-Fokker-Planck equation 
\begin{equation}\label{VFPlineaire_specular_step0}
\left\{
\begin{aligned}
&\trans(S^{}(f))+(B^{}\cdot\nabla_{u}S^{}(f))
-\frac{\sigma^{2}}{2}\triangle_{u}S^{}(f) = 0, ~\mbox{ in}~\Hh'(Q_{T}),\\
&S^{}(f)(0,x,u) = \rho_{0}(x,u),\,\mbox{ on}~\Dd \times\mathbb{R}^{d},\\
&\gamma^-(S(f))(t,x,u) = \gamma^{+}(f)(t,x,u-2\left(u\cdot n_{\Dd }(x)\right)n_{\Dd }(x)),\,\mbox{ on}~\Sigma^{-}_{T}.
\end{aligned}\right.
\end{equation}
Lemma \ref{lem:ExistVFPlineaire}
ensures that ${S^{}}(f)\in \VoneT$, and that the trace functions
$\gamma^{\pm}(S^{}(f))$ belong to $L^{2}(\omega,\Sigma^{\pm}_{T})$.
Therefore, we can define the mapping
\begin{align*}
S^{}:f\in E \longrightarrow S^{}(f)\in E.
\end{align*}
If $S^{}$ admits a fixed point $\SSS_{}$, then it naturally satisfies the specular boundary condition
\begin{equation*}
\gamma^{-}(\SSS_{})(t,x,u)
=\gamma^{+}(\SSS_{})(t,x,u-2\left(u\cdot n_{\Dd }(x)\right)n_{\Dd }(x)),~
\mbox{ on }\Sigma^{-}_{T},
\end{equation*}
implying that $\SSS_{}$ is a weak solution to \eqref{eq:LinearVFPSpec}.
In order to establish the existence of this fixed point, let us observe the following properties of $S^{}$
(the proof of which are postponed at the end on this section):
\begin{description}
\item[$\mathbf{(p1)}$] $S^{}$ is Lipschitz-continuous on $E$;
\item[$\mathbf{(p2)}$] $S^{}$ is an increasing mapping on $E$ w.r.t. the following order relation:
\begin{equation*}
f_{1}\leq f_{2}~\mbox{on}~E~\Longleftrightarrow~
\left\{
\begin{array}{l}
f_{1}(t,x,u)\leq f_{2}(t,x,u),~\mbox{for a.e.}~(t,x,u)\in Q_{T},\\
\gamma(f_{1})(t,x,u)\leq \gamma(f_{2})(t,x,u),\,
\mbox{for }\lambda_{\Sigma_{T}}{-a.e.}~(t,x,u)\in\Sigma_{T};
\end{array}\right.
\end{equation*}
\item[$\mathbf{(p3)}$] If $\underline{P}\leq \gamma^{+}(f)\leq \overline{P}$,
$\lambda_{\Sigma_{T}}$-a.e. on $\Sigma^{+}_{T}$, then
$\underline{P}\leq S(f)\leq \overline{P}$ on $E$.
\end{description}
The fixed point of $\SSS_{}$ arises from the convergence of the sequence
$\{S_{n};\,n\in\mathbb{N}\}$ defined iteratively by $S^{}_{0}=\underline{P}$ and $S_{n+1}=S^{}(S_{n})$.

Indeed, the monotone property $\mathbf{(p2)}$ of $S$ implies that
$\{S_{n};\,n\in\mathbb{N}\}$ is increasing in $E$. In addition, since $\underline{P}= S_{0}\leq \overline{P}$,
 Proposition \ref{prop:EncadrLineaire} ensures that $\underline{P}\leq S_{1}\leq \overline{P}$ and
$\underline{P}\leq \gamma^{+}(S_{1})\leq \overline{P}$. Using repeatedly $\mathbf{(p3)}$, it holds that
\begin{equation}\label{VFPlineaire_specular_step3_bis}
\forall~n\in\mathbb{N},\quad\underline{P}\leq S_{n}\leq \overline{P},~\mbox{on}~E.
\end{equation}
The sequence $\{S_{n};\,n\in\mathbb{N}\}$ being increasing and uniformly bounded on $E$, we deduce that
$\{S_{n};\,n\in\mathbb{N}\}$ and
$\{\gamma^{\pm}(S_{n});\,n\in\mathbb{N}\}$ converge. Set
\begin{equation*}
\SSS_{}(t,x,u):=\lim_{n\rightarrow +\infty}S_{n}(t,x,u),\,\,\mbox{ for}\,(t,x,u)\in Q_{T},
\end{equation*}
and
\begin{equation*}
\gamma^{\pm}(\SSS_{})(t,x,u)=\lim_{n\rightarrow +\infty}\gamma^{\pm}(S_{n})(t,x,u),~\mbox{for}~(t,x,u)\in\Sigma^{\pm}_{T}.
\end{equation*}
According to \eqref{VFPlineaire_specular_step3_bis},
$\underline{P}\leq\SSS_{}\leq \overline{P}$ on $E$.
Since $\overline{P}\in L^{2}(\omega,Q_{T})$ and $\overline{P}\in L^{2}(\omega,\Sigma_{T})$, by dominated convergence,
we obtain that
\begin{equation*}
\begin{aligned}
&\lim_{n\rightarrow +\infty}\Vert \SSS_{}-S_{n}\Vert_{L^{2}(\omega,Q_{T})}=0,\\
&\lim_{n\rightarrow +\infty}\Vert\gamma^{\pm}(\SSS_{})-\gamma^{\pm}(S_{n})\Vert_{L^{2}(\omega,\Sigma^{\pm}_{T})}=0.
\end{aligned}
\end{equation*}
Owing to the continuity of $S$ given in $\mathbf{(p1)}$, we deduce that
\begin{equation*}
\begin{aligned}
&\SSS_{}=\lim_{n\rightarrow +\infty}S_{n+1}
=\lim_{n\rightarrow +\infty}S^{}(S_{n})=S^{}\left(\SSS_{}\right),
\\
&\gamma^{\pm}\left(\SSS_{}\right)=\lim_{n\rightarrow +\infty}\gamma^{\pm}\left(S_{n+1}\right)
=\lim_{n\rightarrow +\infty}\gamma^{\pm}\left(S^{}(S_{n})\right)=
\gamma^{\pm}\left(S^{}\left(\SSS_{}\right)\right).
\end{aligned}
\end{equation*}
We furthermore observe that, owing to \eqref{VFPlineaire_specular_step3_bis}, \eqref{eq:LinearGreenIdentity} ensures that
\begin{equation*}
\sup_{n}\Vert \nabla_{u}S_{n}\Vert^{2}_{L^{2}(\omega,Q_{t})}\leq C\left(\Vert \rho_{0}\Vert^{2}_{L^{2}(\omega,\Dd\times\er^{d})}+\sup_{t\in(0,T)}\Vert \overline{P}(t)\Vert^{2}_{L^{2}(\omega,\er^{d})}\right)
\end{equation*}
 where $C$ is some constant depending only on $d,T,\sigma, \Vert b\Vert_{L^{\infty}(\er^{d};\er^{d})},\alpha$ and the Lebesgue measure of $\Dd$. It follows that $\SSS\in \Hh(\omega,Q_{T})$ with $\nabla_{u}\SSS=\lim_{n\rightarrow +\infty}\nabla_{u}S_{n}$ in $L^{2}(\omega,Q_{T})$.
We thus conclude that $S_{}$ has a fixed point $\SSS^{}$ in $E$.

For the uniqueness claim,
consider two weak solutions $\SSS^{1},~\SSS^{2}$
to  \eqref{eq:LinearVFPSpec}. Set
\begin{align*}
&R^{}(t,x,u):=\SSS^{1}(t,x,u)-\SSS^{2}(t,x,u)~\mbox{for}~(t,x,u)\in Q_{T},
\\
&\gamma^{\pm}(R^{})(t,x,u):=\left(\gamma^{\pm}(\SSS^{1})
-\gamma^{\pm}(\SSS^{2})\right)(t,x,u)~\mbox{for}~(t,x,u)\in\Sigma^{\pm}_{T}.
\end{align*}
Then, $R^{}$ and $\gamma^{\pm}(R^{})$ satisfy
\begin{align*}
\left\{\begin{array}{l}
\trans(R^{})+(B^{}\cdot\nabla_{u}R^{})
-\frac{\sigma^{2}}{2}\triangle_{u}R^{} = 0,\,\mbox{ in}~\Hh'(Q_{T}),
\\
R^{}(0,x,u) = \SSS^{1}(0)-\SSS^{2}(0)=0,\,\mbox{ on}~\Dd \times\mathbb{R}^{d},\\
\gamma^{-}(R^{})(t,x,u) = \left(\gamma^{+}(\SSS^{1})-\gamma^{+}(\SSS^{2})\right)(t,x,u-2(u\cdot n_{\Dd }(x))n_{\Dd }(x)),\,\mbox{ on}~\Sigma^{-}_{T}.
\end{array}\right.
\end{align*}
Using Lemma \ref{lem:ExistVFPlineaire}, one has for all $t\in(0,T]$,
\begin{align*}
\Vert R^{}(t)\Vert^{2}_{L^{2}(\omega,\Dd\times\er^{d})}
+{\sigma^{2}}\Vert  \nabla_{u}R_{} \Vert^{2}_{L^{2}(\omega,Q_{T})}+\Vert\gamma^{+}(R_{})\Vert^{2}_{L^{2}(\omega,\Sigma^{+}_{t})}\\
=\Vert\gamma^{-}(R_{})\Vert^{2}_{L^{2}(\omega,\Sigma^{-}_{t})} + \int_{Q_{t}}
\left(\frac{\sigma^{2}}{{2}}\triangle\omega+\left(B_{}\cdot\nabla_{u}\omega\right)\right)\left|R_{}\right|^{2}.
\end{align*}
Since $\SSS^{1}$ and $\SSS^{2}$ satisfy the specular boundary condition,
\begin{align*}
&\int_{\Sigma^{-}_{t}} (u\cdot n_{\Dd }(x))\omega(u)
\left|\gamma^{-}(R_{})(s,x,u)\right|^{2}
\,d\lambda_{\Sigma_{T}}(s,x,u)
\\
&=-\int_{\Sigma^{+}_{t}} (u\cdot n_{\Dd }(x))
\omega(u)\left|\gamma^{+}(R_{})(s,x,u)\right|^{2}\,d\lambda_{\Sigma_{T}}(s,x,u),
\end{align*}
so that, using Lemma \ref{lem:major_poids}, the previous inequality is reduced to
\begin{equation*}
\min(1,\sigma^{2})\Vert R^{}\Vert^{2}_{V_{1}(\omega,Q_{t})}
\leq \left(  \frac{\sigma^{2}}{2}\alpha(\alpha-2+d)+\alpha
\Vert B\Vert_{L^{\infty}((0,T)\times\Dd;\er^{d})} \right)
\int_{0}^{t}\Vert R_{} \Vert^{2}_{V_{1}(\omega,Q_{s})}\,ds
\end{equation*}
which ensures the uniqueness result by applying Gronwall's Lemma.

\paragraph{Proof of $\mathbf{(p1)}$.} For $f_{1},~f_{2}\in E$, we set
\begin{align*}
&R^{}(t,x,u)=S^{}(f_{1})(t,x,u)-S^{}(f_{2})(t,x,u),
~\mbox{for}~(t,x,u)\in Q_{T},\\
&\gamma^{\pm}(R^{})(t,x,u)=\left(\gamma^{\pm}(S(f_{1}))
-\gamma^{\pm}(S(f_{2}))\right)(t,x,u),
~\mbox{for}~(t,x,u)\in\Sigma^{\pm}_{T}.
\end{align*}
Then, $R^{}$ and $\gamma^{\pm}(R^{})$ satisfy
\begin{align*}
\left\{\begin{array}{l}
\trans(R^{})+(B^{}\cdot\nabla_{u}R^{})
-\frac{\sigma^{2}}{2}\triangle_{u}R^{} = 0,~\mbox{
in}~\Hh'(Q_T),\\
R^{}(0,x,u) = S(f_{1})(0)-S(f_{2})(0)=0,~\mbox{on}~\Dd \times\mathbb{R}^{d},\\
\gamma^{+}(R^{})(t,x,u) = \left(\gamma^{+}(f_{1})-\gamma^{+}(f_{2})\right)(t,x,u-2(u\cdot n_{\Dd }(x))n_{\Dd }(x)),~\mbox{on}~\Sigma^{-}_{T}.
\end{array}\right.
\end{align*}
Replicating the proof of the uniqueness result for \eqref{eq:LinearVFPSpec}, one has
\begin{equation*}\label{VFPlineaire_specular_step2}
\min(1,\sigma^{2})\Vert R^{}\Vert^{2}_{V_{1}(\omega,Q_{t})}+\Vert\gamma^{+}(R_{})\Vert^{2}_{L^{2}(\omega,\Sigma^{+}_{t})}
\leq
\Vert\gamma^{-}(R_{}) \Vert^{2}_{L^{2}(\omega,\Sigma^{-}_{t})} +C\int_{0}^{t}\Vert R^{}\Vert^{2}_{V_{1}(\omega,Q_{s})} \,ds
\end{equation*}
for $C:= \frac{\sigma^{2}}{2}\alpha(\alpha-2+d)+\alpha
\Vert B\Vert_{L^{\infty}((0,T)\times\Dd;\er^{d})}$. Applying Gronwall's Lemma, it follows that
\begin{equation*}
\Vert S^{}(f_{2})-S^{}(f_{1})\Vert^{2}_{E}
\leq C\Vert\gamma^{-}(S^{}(f_{2}))-\gamma^{-}(S^{}(f_{2}))\Vert^{2}_{L^{2}(\omega,\Sigma^{-}_{T})}=C
\Vert\gamma^{+}(f_{2})-\gamma^{+}(f_{1})\Vert^{2}_{L^{2}(\omega,\Sigma^{+}_{T})},
\end{equation*}
which enables us to deduce
\begin{equation*}
\label{Lipschitz_property}
\Vert S^{}(f_{2})-S^{}(f_{1})\Vert^{2}_{E}
\leq C\Vert f_{2}- f_{1}\Vert^{2}_{E},
\end{equation*}
and thus that $S$ is Lipschitz-continuous.
\paragraph{Proof of $\mathbf{(p3)}$.} Assume that $f\in E$ is such that
\begin{equation}\label{VFPlineaire_specular_step3}
\underline{P}(t,u)\leq
\gamma^{+}(f)(t,x,u)
\leq \overline{P}(t,u),~\mbox{for }\lambda_{\Sigma_{T}}\mbox{-a.e.}~(t,x,u)\in\Sigma^{+}_{T}.
\end{equation}
Remark \ref{Cond_spec_pour_maxwellienne}
implies that
\begin{equation*}
\begin{aligned}
&\underline{P}(t,u)=\underline{P}(t,u-2(u\cdot n_{\Dd}(x))n_{\Dd}(x)),\\
&\overline{P}( t,u)=\overline{P}(t,u-2(u\cdot n_{\Dd}(x))n_{\Dd}(x)).
\end{aligned}
\end{equation*}
Hence \eqref{VFPlineaire_specular_step3} is equivalent to
\begin{equation}\label{VFPlineaire_specular_step2_bis}
\underline{P}(t,u)\leq \gamma^{+}(f)(t,x,u-2(u\cdot n_{\Dd}(x))n_{\Dd}(x))= \gamma^-(S)(t,x,u)\leq \overline{P}(t,u),~\lambda_{\Sigma_{T}}\mbox{-a.e. on }\Sigma^{-}_{T}.
\end{equation}
Applying Proposition \ref{prop:EncadrLineaire}, it follows that
\begin{align*}
&\underline{P}\leq S(f)\leq \overline{P},~\mbox{a.e. on}~Q_{T},\\
\mbox{and}~&\underline{P}\leq \gamma^{+}(S(f))\leq \overline{P},~\lambda_{\Sigma_{T}}\mbox{-a.e. on}~\Sigma^{+}_{T}.
\end{align*}
\paragraph{Proof of $\mathbf{(p2)}$.} Let $f_{1},~f_{2}$
be such that $f_{1}\leq f_{2}$ on $E$.
The difference $S^{}(f_{2})-S^{}(f_{1})$ is then a weak solution to the
linear Vlasov-Fokker-Planck equation \eqref{VFPlineaire_Dirichlet}
for $\rho_{0}=0$ and
$$g(t,x,u)=\left(\gamma^{+}(f_{2})-\gamma^{+}(f_{1})\right)(t,x,u-2(u\cdot
n_{\Dd}(x))n_{\Dd}(x));$$ namely
\begin{equation*}
\left\{
\begin{aligned}
&\trans(S^{}(f_{2})-S^{}(f_{1})) + (B^{}\cdot\nabla_{u}(S^{}(f_{2})-S^{}(f_{1})))
-\frac{\sigma^{2}}{2}\triangle_{u}(S^{}(f_{2})-S^{}(f_{1})) = 0,\,\mbox{ in}~\Hh'(Q_{T}),\\
& S^{}(f_{2})(0,x,u) -S^{}(f_{1})(0,x,u)= 0,~\mbox{on}~\Dd \times\mathbb{R}^{d},\\
&\gamma^-\left(S^{}(f_{2})-S^{}(f_{1})\right)(t,x,u) = \left(\gamma^{+}(f_{2})-
\gamma^{+}(f_{1})\right)(t,x,u-2\left(u\cdot n_{\Dd }(x)\right)n_{\Dd }(x)),~\mbox{on}~\Sigma^{-}_{T}.
\end{aligned}
\right.
\end{equation*}
Therefore, applying Lemma \ref{lem:ExistVFPlineaire}, we obtain that
$S^{}(f_{2})-S^{}(f_{1})=S^{}(f_{2}-f_{1})\geq 0$ and that
$\gamma^{+}(S^{}(f_{2}))-\gamma^{+}(S^{}(f_{1}))=\gamma^{+}(S^{}(f_{2})-S^{}(f_{1}))=\gamma^{+}(S^{}(f_{2}-f_{1}))\geq 0$.
We conclude that $S^{}$ is nondecreasing on $E$ since $\gamma^{-}(S^{}(f_{2}))(t,x,u)-\gamma^{-}(S^{}(f_{1}))(t,x,u)=
\left(\gamma^{+}(f_{2})-\gamma^{+}(f_{1})\right)
(t,x,u-2(u\cdot n_{\Dd}(x))n_{\Dd}(x))
\geq 0$.
\end{proof}
\subsubsection{Introduction of the nonlinear drift (end of the proof of Theorem \ref{thm:Existence_result})}
Hereafter we end the proof of Theorem \ref{thm:Existence_result}
by introducing the nonlinear coefficient  $B[\cdot~;~\cdot]$  in the equation \eqref{eq:LinearVFPSpec}.
To this aim, we consider
 $(\underline{P},\overline{P})$,
a couple of Maxwellian distributions with parameters
$(\underline{a},\underline{\mu},\underline{P}_{0})$ and
$(\overline{a},\overline{\mu},\overline{P}_{0})$
such that
\begin{equation}\label{Cond_pour_maxwellienne}
\begin{aligned}
&2\underline{\mu}>1,~2\overline{\mu}>1,\\
&\underline{a}\leq
\frac{-\underline{\mu}}{2\sigma^{2}(\underline{\mu}-1)}
\Vert b\Vert_{L^{\infty}(\er^{d};\er^{d})}^{2},\,\,\overline{a}\geq
\frac{\overline{\mu}}{2\sigma^{2}(1-\overline{\mu})}\Vert
 b\Vert_{L^{\infty}(\er^{d};\er^{d})}^{2}.
\end{aligned}
\end{equation}
We also consider the sequence $\{\rho^{(n)};\,n\in\nat\}$ in $V_{1}(\omega,Q_T)$, defined by
\begin{itemize}
\item $\rho^{(0)}=\rho_{0}$ on $Q_{T}$;
\item For $n\geq 1$ and $\rho^{(n-1)}$ given, we define $B\left[x;\rho^{(n-1)}(t)\right]$ as in \eqref{eq:DriftDefinitionEDP}. Under \hypedpi, we have   $\left|B\right|\leq
\Vert b\Vert_{L^{\infty}(\er^{d};\er^{d})}$.
We define $\rho^{(n)}$ as the unique solution in $V_{1}(\omega,Q_T)$ of
\begin{align}\label{eq:VFPnonlineaire_iteratif}
\left\{\begin{array}{l}
\Tt(\rho^{(n)}) + \left(B[\cdot;\rho^{(n-1)}]\cdot\nabla_{u}\rho^{(n)}\right) - \frac{\sigma^{2}}{2}\triangle_{u}\rho^{(n)} = 0 ,~\mbox{in}~{\Hh'(Q_T)},\\
\gamma^{-}(\rho^{(n)})(t,x,u)=\gamma^{+}(\rho^{(n)})(t,x,u-2(u\cdot n_{\Dd}(x))n_{\Dd}(x)),~\mbox{on}~\Sigma^{-}_{T},\\
\rho^{(n)}(0,x,u)=\rho_{0}(x,u),~\mbox{on}~\Dd\times\er^{d}.
\end{array}\right.
\end{align}
\end{itemize}
According to Proposition \ref{prop:VFPlineairespec}, for all $n\geq 1$,
\begin{equation}\label{VFPnonlineaire_iteratif_step1}
\begin{aligned}
&\underline{P}\leq \rho^{(n)}\leq \overline{P},~\mbox{a.e. on}~Q_{T},\\
&\underline{P}\leq \gamma^{\pm}(\rho^{(n)})\leq \overline{P},~\lambda_{\Sigma_{T}}
\mbox{-a.e. on}~\Sigma^{\pm}_{T},
\end{aligned}
\end{equation}
so that $\rho^{(n)}$ is positive on $Q_{T}$ and
\begin{equation*}
\label{drift_definition_iterative_construction}
B[x;\rho^{(n)}(t)]
=\frac{\displaystyle \int_{\mathbb{R}^{d}} b(v)\rho^{(n)}(t,x,v)\,dv}
{\displaystyle \int_{\mathbb{R}^{d}}\rho^{(n)}(t,x,v)\,dv},~\mbox{for a.e.}~
(t,x)\in (0,T)\times\Dd.
\end{equation*}
\begin{proposition}\label{prop:FixedPts}
Assume \hypedp. Let $(\underline{P},
\overline{P})$ be a couple of Maxwellian distributions with respective parameters
$(\underline{a},\underline{\mu},\underline{P}_{0})$ and
$(\overline{a},\overline{\mu},\overline{P}_{0})$ satisfying
\eqref{Cond_pour_maxwellienne}. Then the sequence $\{\rho^{(n)};\,n\in\nat\}$ converges
 in $\VoneT$ to a weak solution $\rho$ to the non-linear equations \eqref{eq:VFPnonlineaire_interieur}-\eqref{eq:specularcondition}.
Moreover this solution satisfies
\begin{equation*}
\begin{aligned}
&\underline{P}\leq \rho\leq \overline{P},~\mbox{a.e. on}~Q_{T},\\
&\underline{P}\leq \gamma^{\pm}(\rho)\leq \overline{P},~
\lambda_{\Sigma_{T}}\mbox{-a.e. on}~\Sigma^{\pm}_{T}.
\end{aligned}
\end{equation*}
\end{proposition}
\begin{proof}
First, we establish the following  uniform estimation for $\{\rho^{(n)};\,n\in\nat\}$:
\begin{equation}\label{Majorgradient}
\sup_{n\geq 1}\Vert \rho^{(n)}\Vert^{2}_{\VoneT}\leq
K\left(
\Vert \rho_{0}\Vert^{2}_{L^{2}(\omega,\Dd\times\er^{d})} +\Vert \overline{P}
\Vert^{2}_{L^{2}(\omega,(0,T)\times\er^{d})}\right),
\end{equation}
for some constant $K>0$. Indeed, using the energy estimate \eqref{eq:LinearGreenIdentity},
we obtain that, for all $t\in (0,T]$,
\begin{align*}
&\min\left({1,\sigma^{2}}\right)
\Vert \rho^{(n)}\Vert^{2}_{V_{1}(\omega,Q_{t})}\\
&\leq \Vert\rho_{0}\Vert^{2}_{L^{2}(\omega,\Dd\times\er^{d})}
+\left(\frac{\sigma^2}{2}\alpha \left(\alpha - 2 + d\right) + \alpha \Vert b\Vert_{L^{\infty}(\er^{d};\er^{d})}\right)  \int_{0}^{t}\Vert\rho^{(n)}(s)\Vert^{2}_{L^{2}(\omega,\Dd\times\er^{d})}\,ds.
\end{align*}
Using the Maxwellian upper-bound \eqref{VFPnonlineaire_iteratif_step1}, we obtain \eqref{Majorgradient}.

Since $(\VoneT,\Vert~\Vert_{\VoneT})$  is a Banach space,
it is sufficient to establish that $\{\rho^{(n)};\,n\in\nat\}$
is a Cauchy sequence in $\VoneT$. For $n$, $m> 1$, the functions
\begin{equation*}
\begin{aligned}
&R^{(n,n+m)}:=\rho^{(n+m)}-\rho^{(n)},~\mbox{on}~Q_{T},\\
&\gamma^{\pm}(R^{(n,n+m)}):=\gamma^{\pm}(\rho^{(n+m)})-\gamma^{\pm}(\rho^{(n+m)}),~\mbox{on}~\Sigma^{\pm}_{T},
\end{aligned}
\end{equation*}
satisfy
\begin{equation*}
\left\{
\begin{aligned}
&\trans(R^{(n,n+m)})-\frac{\sigma^{2}}{2}\triangle_{u}R^{(n,n+m)} =\nabla_{u}\cdot\left( B[\cdot;\rho^{(n-1)}]\rho^{(n)}-B[\cdot;\rho^{(n+m-1)}]\rho^{(n+m)}\right), ~\mbox{in}~\Hh'(Q_{T}),\\
&R^{(n,n+m)}(0,x,u)=0,~\mbox{on}~\Dd \times\mathbb{R}^{d},\\
&\gamma^{-}(R^{(n,n+m)})(t,x,u)=\gamma^{+}(R^{(n,n+m)})(t,x,u-2(u\cdot n_{\Dd}(x))n_{\Dd}(x)),~\mbox{on}~\Sigma^{-}_{T}.
\end{aligned}\right.
\end{equation*}
Then, according to \eqref{eq:LinearGreenIdentity}, it follows that for all $t\in(0,T]$,
\begin{equation*}
\begin{aligned}
&\int_{\Dd\times\er^{d}}\omega
\left|R^{(n,n+m)}(t,x,u)\right|^{2}\,dx\,du
+{\sigma^{2}}\int_{Q_{t}}\omega
\left|\nabla_{u}R^{(n,n+m)}\right|^{2}\\
&\leq \frac{\sigma^{2}}{{2}}\int_{Q_{t}}\triangle\omega\left|R^{(n,n+m)}\right|^{2}
+{2}\int_{Q_{t}}\omega R^{(n,n+m)} \nabla_{u}\cdot \left(B[\cdot~;\rho^{(n-1)}]\rho^{(n)}-B[\cdot~;\rho^{(n+m-1)}]\rho^{(n+m)} \right).
\end{aligned}
\end{equation*}
But
\begin{align*}
&2\int_{Q_{t}}\omega R^{(n,n+m)} \nabla_{u}\cdot \left(B[\cdot~;\rho^{(n-1)}]\rho^{(n)}-B[\cdot~;\rho^{(n+m-1)}]\rho^{(n+m)} \right)\\
&\leq 2\int_{Q_{t}}  \omega\left(\alpha |  R^{(n,n+m)} | +  |\nabla_{u}R^{(n,n+m)}|\right)\, \left|B[\cdot~;\rho^{(n-1)}]\rho^{(n)}-B[\cdot~;\rho^{(n+m-1)}]\rho^{(n+m)} \right|\\
&\leq  \int_{Q_{t}} {\alpha^2}\omega (R^{(n,n+m)})^2 + \int_{Q_{t}}\frac{\sigma^2}{2}\omega |\nabla_u R^{(n,n+m)}|^2 \\
& \qquad + ( 1+ \frac{2}{\sigma^2}) \int_{Q_t} \omega \left|B[\cdot~;\rho^{(n-1)}]\rho^{(n)}-B[\cdot~;\rho^{(n+m-1)}]\rho^{(n+m)} \right|^2
\end{align*}
and hence, using Lemma \ref{lem:major_poids},
\begin{align*}
&\int_{\Dd\times\er^{d}}\omega
\left|R^{(n,n+m)}(t,x,u)\right|^{2}\,dx\,du
+\frac{\sigma^{2}}{2}\int_{Q_{t}}\omega
\left|\nabla_{u}R^{(n,n+m)}\right|^{2}\\
&\leq\left(\frac{\sigma^2}{2}\alpha(\alpha -2 +d) + \alpha^2\right)\int_{0}^{t}\Vert R^{(n,n+m)}\Vert^{2}_{V^{1}(\omega,Q_{s})}\,ds\\
&\qquad+( 1+ \frac{2}{\sigma^2}) \int_{Q_t} \omega \left|B[\cdot~;\rho^{(n-1)}]\rho^{(n)}-B[\cdot~;\rho^{(n+m-1)}]\rho^{(n+m)} \right|^2.
\end{align*}
Let us now observe that
\begin{align*}
&\int_{Q_{t}}\omega\left|B[\cdot~;\rho^{(n-1)}]\rho^{(n)}-B[\cdot~;\rho^{(n+m-1)}]\rho^{(n+m)} \right|^2\\
&\leq\frac{1}{2}\int_{Q_{t}}\omega\left|\rho^{(n+m)}- \rho^{(n)}\right|^{2}
\left|B[\cdot~;\rho^{(n+m-1)}]\right|^{2}
+\frac{1}{2}\int_{Q_{t}}\omega\left|\rho^{(n)}\right|^{2}
\left|B[\cdot~;\rho^{(n+m-1)}]-B[\cdot~;\rho^{(n-1)}]\right|^{2}.
\end{align*}
For the first term, we have
\begin{equation}
\label{blabla_EDP}
\int_{Q_{t}}\omega\left|\rho^{(n+m)}- \rho^{(n)}\right|^{2}
\left|B[\cdot~;\rho^{(n+m-1)}]\right|^{2}\leq
\Vert b\Vert^{2}_{L^{\infty}(\er^{d};\er^{d})}\int_{Q_{t}}\omega\left|\rho^{(n+m)}- \rho^{(n)}\right|^{2}.
\end{equation}
For the second term, let us set
\begin{equation*}
\underline{M}:=\inf_{\substack{t\in(0,T)}}\int_{\er^{d}}\underline{P}(t,u)du,\,\,\overline{M}:=\sup_{\substack{t\in (0,T)}}\int_{\er^{d}}\omega(u)
\left|\overline{P}(t,u)\right|^{2}\,du.
\end{equation*}
According to the definition of $B[\cdot;\cdot]$ in \eqref{eq:DriftDefinitionEDP},
\begin{align*}
&\left| B[x;\rho^{(n+m-1)}(s)]-B[x;\rho^{(n-1)}(s)]\right| ^{2}\\
&=\left|
\frac{\int_{\er^{d}} b(v)\rho^{(n+m-1)}(s,x,v)dv}{\int_{\er^{d}} \rho^{(n+m-1)}(s,x,v)dv}
-\frac{\int_{\er^{d}} b(v)\rho^{(n-1)}
(s,x,v)dv}
{\int_{\er^{d}} \rho^{(n-1)}(s,x,v)dv} \right| ^{2}\\
&\leq \frac{1}{2}\frac{\left|\int_{\er^{d}} b(v)\left(\rho^{(n+m-1)}- \rho^{(n-1)}\right)
(s,x,v)dv \right|^{2}}{\left|\int_{\er^{d}} \rho^{(n+m-1)}(s,x,v)dv\right|^{2}}\\
& \qquad 
+\frac{1}{2}\frac{\left|\int_{\er^{d}}
\left(\rho^{(n+m-1)}- \rho^{(n-1)}\right)(s,x,v)dv \right|^{2}
\left|\int_{\er^{d}} b(v)\rho^{(n-1)}(s,x,v)dv \right|^{2}}{\left|\int_{\er^{d}} \rho^{(n+m-1)}(s,x,v)dv\int_{\er^{d}} \rho^{(n-1)}(s,x,v)dv\right|^{2}}\\
&\leq \frac{\Vert b\Vert^{2}_{L^{\infty}(\er^{d};\er^{d})}\left(\int_{\er^{d}}
\left|\rho^{(n+m-1)}- \rho^{(n-1)}\right|(s,x,v)dv \right)^{2}}
{(\underline{M})^{2}}\\
&\leq \frac{\overline{w}\Vert b\Vert^{2}_{L^{\infty}(\er^{d};\er^{d})}}
{(\underline{M})^{2}}
\int_{\er^{d}}\omega(v)\left|\left(\rho^{(n+m-1)}-\rho^{(n-1)}\right)(s,x,v)\right|^{2}\,dv,
\end{align*}
since $\overline{w}:=\displaystyle \int_{\er^{d}}\frac{1}{\omega(v)}\,dv< +\infty$  (see Lemma \ref{lem:major_poids}), from which it follows that
\begin{align*}
\int_{Q_{t}}\omega|\rho^{(n)}|^{2}
\left|B[\cdot~;\rho^{(n+m-1)}]-B[\cdot~;\rho^{(n-1)}]\right|^{2}
\leq \Vert b\Vert^{2}_{L^{\infty}(\er^{d};\er^{d})}\overline{w}\frac{\overline{M}^2
}{\underline{M}^{2}}
\int_{Q_{t}}\omega\left|\rho^{(n+m-1)}- \rho^{(n-1)}\right|^{2}.
\end{align*}
Combining this inequality with \eqref{blabla_EDP}, we obtain that
\begin{align*}
&\int_{Q_{t}}\omega\left|\rho^{(n+m)}B[\cdot~;\rho^{(n+m-1)}]
- \rho^{(n)}B[\cdot~;\rho^{(n-1)}]\right|^{2}\\
&\leq \Vert b\Vert^{2}_{L^{\infty}(\er^{d};\er^{d})}\int_{Q_{t}}
\omega\left|R^{(n,n+m)}\right|^{2}{+\Vert b\Vert^{2}_{L^{\infty}(\er^{d};\er^{d})}\overline{w}\frac{\overline{M}^2
}{\underline{M}^{2}}
\int_{Q_{t}}\omega\left|R^{(n-1,n+m-1)}\right|^{2}}.
\end{align*}
Then for all $t$ in $(0,T]$,
\begin{align*}
&\Vert R^{(n,n+m)}\Vert^{2}_{V_{1}(\omega,Q_{t})}
\leq  {C_{1}}\int_{0}^{t}\Vert R^{(n,n+m)}\Vert^{2}_{
V_{1}(\omega,Q_{s})}\,ds +   C_{2}\int_{0}^{t}\Vert R^{(n+m-1,n-1)}\Vert^{2}_{
V_{1}(\omega,Q_{s})}\,ds,
\end{align*}
with
\begin{align*}
C_{1}=\frac{\frac{\sigma^2}{2}\alpha ( \alpha -2 + d)  + \alpha ^2  + \Vert b\Vert^{2}_{L^{\infty}(\er^{d};\er^{d})}(
1 + \frac{2}{\sigma^2})}{\min\left(1,\frac{\sigma^{2}}{2}\right)}, \quad
C_{2}=\frac{(1 + \frac{2}{\sigma^2})
\Vert b\Vert^{2}_{L^{\infty}(\er^{d};\er^{d})}
\overline{w}\frac{\overline{M}^2}
{\underline{M}^{2}}}{\min\left(1,\frac{\sigma^{2}}{2}\right)}.
\end{align*}
Therefore, applying  Gronwall's Lemma, we deduce that
\begin{equation*}
\Vert R^{(n,n+m)}\Vert^{2}_{V_{1}(\omega,Q_{t})}
\leq C_{{2}}\int_{0}^{t}\left(1+ \exp\left\{C_{1}(t-s)\right\}\right)
\Vert R^{(n+m-1,n-1)}\Vert^{2}_{V_{1}(\omega,Q_{s})}\,ds.
\end{equation*}
Iterating $n-1$ times this inequality, we obtain
\begin{align*}
&\Vert R^{(n,n+m)}\Vert^{2}_{\VoneT}\\
&\leq (C_{2})^{n-1} \int_{0}^{T}
\left(1+ \exp\{C_{1}(T-t_n)\}\right)\int_{0}^{t_{n-1}}
\dotsi \\
&\qquad\qquad \qquad\qquad \dotsi \int_{0}^{t_{2}}\left(1+ \exp\{C_{1}(T-t_{1})\}\right)\Vert \rho^{(m+1)}-\rho^{(1)}
\Vert^{2}_{V_{1}(\omega,Q_{t_{1}})}\,dt_{1}\dotsi\,dt_{n-1}\,dt_{n}
\\
&\leq \frac{\left(C_{T}\right)^{n-1}}{(n-1) !}\Vert \rho^{(m+1)}-\rho^{(1)}\Vert^{2}_{\VoneT}
\end{align*}
with $C_{T}:=C_{{2}}\int_{0}^{T}(1+
\exp\{C_{1} (T-t)\})dt$.
Using the estimation \eqref{Majorgradient}, it follows that
\begin{align*}
\sup_{\substack{m\in\mathbb{N}}}\Vert \rho^{(n+m)} -\rho^{(n)}\Vert^{2}_{\VoneT}
\leq  \frac{\left(C_{T}\right)^{n}}{n !}2K\left(
\Vert \rho_{0}\Vert^{2}_{L^{2}(\omega,\Dd\times\er^{d})} +\overline{M}\right).
\end{align*}
Since ${\displaystyle \sum_{\substack{n\in\mathbb{N}}}
\frac{\left(C_{T}\right)^{n}}{n !}
=\exp\left\{C_{T}\right\}<+\infty}$,
${\displaystyle\frac{\left(C_{T}\right)^{n}}{n !}}$ tends to $0$ as $n$ tends to infinity.
Therefore $\left\{\rho^{(n)};\,n\in\NN\right\}$ is a Cauchy sequence
in  $V_{1}\left(\omega,Q_{T}\right)$.

Let us denote by $\rho$, the limit of
$\left\{\rho^{(n)};~n\geq 1\right\}$ in
$\VoneT$. According to \eqref{VFPnonlineaire_iteratif_step1}, we have
\begin{equation*}
\underline{P}\leq \rho\leq \overline{P},~\mbox{a.e. on}~Q_{T}.
\end{equation*}
Now we check that $\rho$ is a weak solution to \eqref{eq:VFPnonlineaire_interieur}-\eqref{eq:specularcondition}.
Since $\{\rho^{n};~n\geq 1\}$ tends to $\rho$ in
$\VoneT$, and so in
$\Hh(\omega,Q_{T})$,
we can consider a subsequence still denoted by
$\{\rho^{(n)};~n\geq 1\}$ such that
\begin{align*}
&\lim_{n\rightarrow+\infty}\rho^{(n)}(t,x,u)=\rho(t,x,u),~\mbox{for a.e.}
~(t,x,u)\in Q_{T},\\
&\lim_{n\rightarrow+\infty}
\nabla_{u}\rho^{(n)}(t,x,u)=\nabla_{u}\rho(t,x,u),~\mbox{for a.e.}
~(t,x,u)\in Q_{T},
\end{align*}
and
\begin{equation*}
\sup_{n\geq 1}\left(\left|\rho^{(n)}(t,x,u)\right|+
\left|\nabla_{u}\rho^{(n)}(t,x,u)\right|\right)\in L^{2}(\omega,Q_{T}).
\end{equation*}
Therefore, by dominated convergence, we deduce the convergence of the coefficients
\begin{equation*}
\lim_{n\rightarrow +\infty}B[x;\rho^{(n)}(t)]=B[x;\rho(t)],
~\mbox{a.e. on }(0,T)\times\Dd.
\end{equation*}
We further observe that for
$\psi\in \Cc^{\infty}_{c}(Q_{T})$,
\begin{align*}
|(\trans(\sqrt{\omega}\rho),\psi)_{\Hh'(Q_T),\Hh(Q_T)}| =
&\left|\int_{Q_{T}}\trans(\psi)\sqrt{\omega}\rho\right|
\leq \limsup_{n\rightarrow +\infty}
\left|\int_{Q_{T}}\trans(\psi)\sqrt{\omega}\rho^{(n)}\right|
\\
&\leq \limsup_{n\rightarrow +\infty}
\left|\int_{Q_{T}}\sqrt{\omega} \psi
\left(B^{}\left[\cdot~;\rho^{(n-1)}\right]\cdot\nabla_{u}\rho^{(n)}\right)-
\frac{\sigma^{2}}{2}\left(\nabla_{u}(\sqrt{\omega}\psi)\cdot
\nabla_{u}\rho^{(n)}\right)\right|
\\
&\leq \left(\Vert b^{}\Vert_{L^{\infty}(\er^{d};\er^{d})}
+\frac{\sigma^{2}}{2}\left(1+\frac{\alpha}{{2}}\right)\right)
\Vert\psi\Vert_{\mathcal{H}(Q_{T})}\limsup_{n\rightarrow +\infty}\Vert\rho^{(n)}\Vert_{\VoneT}.
\end{align*}
Owing to \eqref{Majorgradient}, the right-hand side is uniformly bounded, so that
$\trans(\sqrt{\omega}\rho)\in\mathcal{H}'(Q_{T})$ and, from \eqref{eq:VFPnonlineaire_iteratif},
\begin{equation*}
\trans(\rho)+\left(B[\cdot~;\rho]\cdot\nabla_{u}\rho\right)
-\frac{\sigma^{2}}{2}\triangle_{u}\rho=0\,,~\mbox{in}~\Hh'(Q_{T}).
\end{equation*}
According to Lemma \ref{lem:Derive_Carrillo}, $\rho$ admits
traces functions along the frontier $\Sigma^{\pm}_{T}$, which belongs to
$L^{2}(\omega,\Sigma^{\pm}_{T})$. It  remains to check the specular boundary condition
and the Maxwellian bounds  for $\rho$:
\begin{equation}\label{VFPnonlineaire_iteratif_step_final}
\begin{aligned}
&\gamma^{-}(\rho)(s,x,u)=\gamma^{+}(\rho)(s,x,u-2(u\cdot
n_{\Dd}(x))n_{\Dd}(x)),
~\lambda_{\Sigma_{T}}~\mbox{-a.e. on}~\Sigma^{-}_{T},\\
&\underline{P}\leq \gamma^{\pm}(\rho)\leq \overline{P},~\lambda_{\Sigma_{T}}~\mbox{-a.e. on}~\Sigma^{\pm}_{T}.
\end{aligned}
\end{equation}
For all
$\psi\in\Cc^{\infty}_{c}(\overline{Q_{T}})$, we have
\begin{equation*}
\begin{aligned}
&\int_{Q_{T}}\trans(\psi)\left(\rho-\rho^{(n)}\right)
+\int_{Q_{T}}\psi \left(\left(B[\cdot~;\rho]\cdot\nabla_{u}\rho\right)
- \left(B\left[\cdot~;\rho^{(n-1)}\right]\cdot\nabla_{u}\rho^{(n)}\right)\right)
\\
&+\frac{\sigma^{2}}{2}\int_{Q_{T}}\left(\nabla_{u}\psi\cdot\nabla_{u}
\left(\rho-\rho^{(n)}\right)\right)
=-\int_{\Dd\times\er^{d}}\psi(T,x,u)\left(\rho(T,x,u)-\rho^{(n)}(T,x,u)
\right)\,dx\,du\\
&\quad-\int_{\Sigma^{+}_{T}}(u\cdot n_{\Dd}(x))\psi(s,x,u)
\left(\gamma^{+}(\rho)-\gamma^{+}(\rho^{(n)})\right)(s,x,u)d\lambda_{\Sigma_{T}}(s,x,u)\\
&\quad-\int_{\Sigma^{-}_{T}}(u\cdot n_{\Dd}(x))\psi(s,x,u)
\left(\gamma^{-}(\rho)(s,x,u)-\gamma^{-}(\rho^{(n)})(s,x,u)\right)
d\lambda_{\Sigma_{T}}(s,x,u).
\end{aligned}
\end{equation*}
Hence, for all $\psi\in \Cc^{\infty}_{c}(\overline{Q_{T}})$ vanishing
on $\left\{T\right\}\times\Dd\times\er^{d}$ and $\Sigma^{+}_{T}$, we have
\begin{equation*}
\lim_{n\rightarrow +\infty}\int_{\Sigma^{-}_{T}}(u\cdot n_{\Dd}(x))\psi(s,x,u)
\left(\gamma^{-}(\rho)(s,x,u)-\gamma^{-}(\rho^{(n)})(s,x,u)\right)
\,d\lambda_{\Sigma_{T}}(s,x,u)=0.
\end{equation*}
It follows that
\begin{equation*}
\lim_{n\rightarrow +\infty}
\Vert\gamma^{-}(\rho)-\gamma^{-}(\rho^{(n)})\Vert_{L^{2}(\Sigma^{-}_{T})}=0.
\end{equation*}
Since, for all $n\geq 1$, $\rho^{(n)}$ satisfies the specular boundary condition and the Maxwellian bounds
\eqref{VFPnonlineaire_iteratif_step1}, we deduce
\eqref{VFPnonlineaire_iteratif_step_final}.
\end{proof}
\section{Well-posedness for the nonlinear Lagrangian stochastic model with specular boundary condition}\label{sec:LawReconstruction}
\subsection{Construction of a solution} \label{sec:LawConstruction}
Under the hypotheses \hyp, we construct a solution to the Lagrangian stochastic model with specular boundary
condition \eqref{eq:NonlinearConfinedLangevin} that  satisfies the mean no permeability condition \eqref{eq:nopermeability}.

Let us consider  a probability $\PP$  given by Theorem  \ref{thm:LinearSolutionByStraightening}, on the sample space $\TTT$, such that under $\PP$ the canonical process $((x(t),u(t));\,t\in[0,T])$ of $\TTT$ satisfies
\begin{equation*}
\left\{
\begin{aligned}
x(t) &= x(0) + \int_{0}^{t} u(s)\,ds,\\
u(t) &= u(0) + \sigma \widetilde{w}(t) -\sum_{\substack{0 < s \leq t}} 2\left(u(s^{-})\cdot n_{\Dd}(x(s))\right)n_{\Dd}(x(s))
\ind_{\displaystyle \left\{x(s)\in \partial \Dd\right\}},
\end{aligned}
\right.
\end{equation*}
where $(\widetilde{w}(t);\,t\in[0,T])$ is an  $\er^d$-Brownian motion and $\PP\circ(x(0), u(0))^{-1} = \mu_0$.
Next, we consider a solution $\rhofp\in V_{1}(\omega,Q_T)$ to the conditional McKean-Vlasov-Fokker-Planck equation \eqref{eq:ConditionalMcK-VFP}
that we have constructed in Theorem \ref{thm:Existence_result}. We also denote by $\gamma^{\pm}(\rhofp)\in L^{2}(\omega,Q_{T})$ its trace functions. Due to the Maxwellian bounds \eqref{EncadrMaxwell}, one can check that the function
\begin{equation*}
\gamma(\rhofp)(t,x,u):=\left\{
\begin{aligned}
&\gamma^{+}(\rhofp)(t,x,u)~\mbox{ on }\,\Sigma^{+}_{T},\\
&\gamma^{-}(\rhofp)(t,x,u)~\mbox{ on }\,\Sigma^{-}_{T},
\end{aligned}
\right.
\end{equation*}
is a trace of $\rhofp$ in the sense of Definition \ref{def:Trace_density_process}. In particular,  the integrability and positivity requirement \eqref{nopermeabilitycondition_2} is an  immediate consequence of the Maxwellian bounds  \eqref{EncadrMaxwell}, and the initial bounds in  \hypedpiii. For \eqref{nopermeabilitycondition_1}, using the estimate \eqref{MaxwellUpperBound}, and since $\|\frac{1}{\omega}\|_{L^1(\er^d)}<+\infty$,
\begin{align*}
\int_{\Sigma_{T}}|(u\cdot n_{\Dd}(x))|\gamma(\rhofp)(t,x,u)\,d\lambda_{\Sigma_{T}}(t,x,u)&=2\int_{\Sigma^{+}_{T}}|(u\cdot n_{\Dd}(x))|\gamma^{+}(\rhofp)(t,x,u)\,d\lambda_{\Sigma_{T}}(t,x,u)\\
&\leq C\int_{(0,T)\times\er^d} |u| \overline{P}(t,u) du\;  dt\\
&\leq C\left(\int_{(0,T)\times\er^{d}}|u|\omega(u)\overline{P}^{2}(t,u)dt\;du\right)^{\frac{1}{2}}<+\infty,
\end{align*}
where $C$ is a  constant depending only on $T$, $\partial\Dd$, and $\|\frac{1}{\omega}\|_{L^1(\er^d)}$.  The mean no-permeability condition \eqref{eq:nopermeability} is then  satisfied, since $\rhofp$
satisfies the specular boundary condition.

We now introduce the probability measure $\QQ$ defined by
\begin{equation}\label{eq:SuitableCandidate}
\frac{d\QQ}{d\PP}=\exp\left\{\frac{1}{\sigma}\int_{0}^{T}B[x(t);\rhofp(t)]\,d\widetilde{w}(t)-
\frac{1}{2\sigma^{2}}\int_{0}^{T}\left|B[x(t);\rhofp(t)]\right|^2\,dt\right\}.
\end{equation}
Then, according to Girsanov Theorem, $((x(t),u(t));\,t\in[0,T])$ satisfies the confined Langevin equation with the additional drift $(t,x) \mapsto B[x;\rhofp(t)]$; namely, $\QQ$-a.s.,
\begin{equation*}
\left\{
\begin{aligned}
x(t) &= x(0) + \int_{0}^{t} u(s)\,ds,\\
u(t) &= u(0) + \int_0^t B[x(s);\rhofp(s)]\,ds+ \sigma w(t)
 -\sum_{\substack{0 < s \leq t}} 2\left(u(s^{-})\cdot n_{\Dd}(x(s))\right)n_{\Dd}(x(s))
\ind_{\displaystyle \left\{x(s)\in \partial \Dd\right\}},
\end{aligned}
\right.
\end{equation*}
where $(w(t):=\widetilde{w}(t)-\int_0^t B[x(s);\rhofp(s)]\,ds;\,t\in[0,T])$ is an $\er^d$-valued $\QQ$-Brownian motion, and $\QQ(x(0)\in dx,u(0)\in du)=\rho_{0}(x,u)dx\,du$.
To prove that $\QQ$ is a solution in law to \eqref{eq:NonlinearConfinedLangevin}, we check that the time-marginals $\QQ\circ(x(t),u(t))^{-1}$ satisfy a mild equation. More precisely, we show that $\QQ\circ(x(t),u(t))^{-1}$ admits a density function equal to $\rhofp(t)$, so that $B[x(t);\rhofp(t)]$ is equal to
$\EE_{\QQ}[b(u(t))|x(t)]$. For this purpose,  we introduce the following linear mild equation:
for all $t\in(0,T]$, for all $\psi\in\Cc_{c}(\Dd\times\er^{d})$,
\begin{equation}\label{eq:mild-bis}
\left<\psi,\rho(t)\right>=\left<\Gamma^{\psi}(t),\rho_{0}\right>+\int_{0}^{t}\left<\nabla_{u}\Gamma^{\psi}(t-s),
B[\cdot;\rhofp(s)]\rho(s)\right>\,ds,
\end{equation}
where $\left<\cdot,\cdot\right>$ stands for the inner product in $L^{2}(\Dd\times\er^d)$ and $\Gamma^{\psi}(t,x,u)$ is defined as in \eqref{eq:EqMildKernel}.
According to Corollary \ref{coro:KernelEstimate}, for all $\psi\in\Cc_{c}(\Dd\times\er^{d})$,
$\Gamma^{\psi}$ belongs to $L^{2}((0,T)\times\Dd;H^1(\er^d))$ hence \eqref{eq:mild-bis} is well defined.
Furthermore, we have
\begin{proposition}\label{prop:weakuniqueness}
There exists  at most one  solution in $\Cc([0,T];L^{2}(\Dd\times\er^d))$ to the linear mild equation \eqref{eq:mild-bis}.
\end{proposition}
\begin{proof} Let $\rho_1,\rho_2\in\Cc([0,T];L^2(\Dd\times\er^d))$ be two mild solutions to
\eqref{eq:mild-bis}. Then, for all $t\in[0,T]$,
\begin{align*}
\Vert\rho_1(t)-\rho_2(t)\Vert^{2}_{L^2(\Dd\times\er^d)} &=\sup_{\substack{\psi\in \Cc_{c}(\Dd\times\er^d);\\\Vert \psi\Vert_{L^2(\Dd\times\er^d)}=1}}\left(\left<\psi,(\rho_1-\rho_2)(t)\right>\right)^{2} \\
&=\sup_{\substack{\psi\in \Cc_{c}(\Dd\times\er^d);\\\Vert \psi\Vert_{L^2(\Dd\times\er^d)}=1}}\left|\int_{Q_t}\left(\nabla_{u}\Gamma^{\psi}(t-s,x,u),
B[x;\rhofp(s)]\right)(\rho_1-\rho_2)(s,x,u)\,ds\,dx\,du\right|^2 \\
&\leq \Vert b\Vert^2_{L^{\infty}(\er^d;\er^{d})}
\sup_{\substack{\psi\in \Cc_c(\Dd\times\er^d);\\\Vert \psi\Vert_{L^2(\Dd\times\er^d)}=1}} \Vert\nabla_{u}\Gamma^{\psi}\Vert^2_{L^2(Q_{t})}
\int_{0}^{t}\Vert \rho_1(s)-\rho_2(s)\Vert^2_{L^2(\Dd\times\er^{d})}\,ds.
\end{align*}
By using the estimate \eqref{ineq:SobolevEstimate_L2} in Corollary \ref{coro:KernelEstimate} on the decomposition $(\psi)^+$ and $(\psi)^-$, we obtain that  $\Vert\nabla_{u}\Gamma^{\psi}\Vert_{L^{2}(\omega,Q_{T})}\leq 1/\sigma^2$. It follows that
\begin{equation*}
\Vert\rho_1(t)-\rho_2(t)\Vert^{2}_{L^2(\Dd\times\er^d)} \leq \frac{\Vert b\Vert^{2}_{L^\infty(\er^d;\er^{d})}}{\sigma^{2}}\int_{0}^{t}
\Vert \rho_1(s)-\rho_2(s)\Vert^{2}_{L^2(\Dd\times\er^d)}\,ds.
\end{equation*}
By Gronwall's Lemma, we conclude on the uniqueness of solutions to \eqref{eq:mild-bis}.
\end{proof}
\begin{proposition}\label{thm:LinkMild&MVFK}
\item{(i)} The solution $(\rhofp(t); t\in[0,T])\in \VoneT$ of \eqref{eq:ConditionalMcK-VFP},  constructed in Theorem  \ref{thm:Existence_result} is solution to the mild equation \eqref{eq:mild-bis}.\\
\item{(ii)} For $\QQ$ defined in \eqref{eq:SuitableCandidate}, for all $t\in[0,T]$, the time marginal $\QQ\circ(x(t),u(t))^{-1}$ admits a density $\rho(t)\in L^{2}(\omega,\Dd\times\er^{d})$ which is solution to the mild equation \eqref{eq:mild-bis}.
\end{proposition}
Combining Proposition \ref{thm:LinkMild&MVFK}  and Proposition \ref{prop:weakuniqueness}, we conclude on the equality $\QQ\circ(x(t),u(t))^{-1}=\rhofp(t,x,u)dx,du$. We also conclude  that $\QQ$
is a solution in law to \eqref{eq:NonlinearConfinedLangevin} in $\Pi_{\omega}$. The set of time marginal densities of $\QQ$ naturally inherits the
trace functions of $\rhofp$ so that the specular and the mean no-permeability boundary conditions are both satisfied.
\begin{proof}[Proof of Proposition \ref{thm:LinkMild&MVFK}]
We first prove $(i)$. According to Theorem \ref{thm:Existence_result}, $\rhofp\in \VoneT \subset \Cc([0,T];L^2(\Dd\times\er^d))$ satisfies: for all $t\in[0,T]$, $\Psi\in\Cc^{\infty}_{c}(\overline{Q_{t}})$,
\begin{align}\label{eq:KVFPVariational}
\begin{aligned}
&\int_{Q_{t}}\left(\rhofp(s,x,u)\left(\partial_{s}\Psi+(u\cdot\nabla_{x}\Psi)+\frac{\sigma^{2}}{2}\triangle_{u}\Psi\right)(s,x,u)
+ \left(B[x;\rhofp(s)]\cdot\nabla_{u}\Psi(s,x,u)\right)\rhofp(s,x,u)\right)\,ds\,dx\,du
\\
&=-\int_{\Dd\times\er^{d}}\rhofp(t,x,u)\Psi(t,x,u)\,dx\,du +
\int_{\Dd\times\er^{d}} \rho_{0}(x,u)\Psi(0,x,u)\,dx\,du
\\
&\quad-\int_{\Sigma^{+}_{t}}(u\cdot n_{\Dd}(x))\gamma^{+}(\rhofp)(s,x,u)\left(\Psi(s,x,u)-\Psi(s,x,u-2(u\cdot n_{\Dd }(x))n_{\Dd }(x))\right)d\lambda_{\Sigma_{T}}(s,x,u).
\end{aligned}
\end{align}
Using convolution approximation on the test function and since $\rhofp,\nabla_{u}\rhofp$ and $\gamma^{+}(\rhofp)$ are square-integrable, one can extend the preceding formula for all $\Psi\in\Cc_{b}(\overline{Q_{t}})\cap\Cc^{1,1,2}_{b}(Q_{t})$.
Using Corollary \ref{coro:KernelSmoothness}, we know that
$(s,x,u)\in Q_{t}\mapsto\Gamma^{\psi}_{n}(t-s,x,u)$ is a smooth function that satisfies
\begin{equation}\label{eq:backwardVFP-bis-bis}
\left\{
\begin{aligned}
&\partial_{s}\Gamma^{\psi}_{n}(t-s)+(u\cdot \nabla_{x}\Gamma^{\psi}_{n}(t-s))+\frac{\sigma^{2}}{2}\triangle_{u}\Gamma^{\psi}_{n}(t-s)=0,~\mbox{on}~Q_{t},\\
&\lim_{s\rightarrow t}\Gamma^{\psi}_{n}(t-s,x,u)=\psi(x,u),~\mbox{on}~\Dd\times\er^{d},\\
&\Gamma^{\psi}_{n}(t-s,x,u)=\Gamma^{\psi}_{n-1}(t-s,x,u-2(u\cdot n_{\Dd}(x))n_{\Dd}(x)),~\mbox{on}~\Sigma^{+}_{t}.
\end{aligned}
\right.
\end{equation}
Hence, in the case $\Psi(s,x,u)=\Gamma^{\psi}_{n}(t-s,x,u)$, \eqref{eq:KVFPVariational} reduces to
\begin{equation}
\label{proofst1:LinkMild&MVFK}
\begin{aligned}
&\int_{\Dd\times\er^{d}}\rhofp(t,x,u)\psi(x,u)\,dx\,du\\
&=\int_{\Dd\times\er^{d}} \rho_{0}(x,u)\Gamma^{\psi}_{n}(t,x,u)\,dx\,du
+\int_{Q_{t}}\left(B[x;\rhofp(s)]\cdot\nabla_{u}\Gamma^{\psi}_{n}(t-s,x,u)\right)\rhofp(s,x,u)\,ds\,dx\,du\\
&\quad-\int_{\Sigma^{-}_{t}}(u\cdot n_{\Dd}(x))\gamma^{-}(\rhofp)(s,x,u)\left(\Gamma^{\psi}_{n}(t-s,x,u)-\Gamma^{\psi}_{n-1}(t-s,x,u)\right)d\lambda_{\Sigma_{T}}(s,x,u).
\end{aligned}
\end{equation}
Owing to \eqref{eq:KernelsConvergence}, we obtain \eqref{eq:mild-bis} by taking the limit $n\rightarrow +\infty$.

Now we prove $(ii)$. Let $\QQ$ be defined as in \eqref{eq:SuitableCandidate}. Let us also introduce the time-marginal probability measures
$(\mu^{m}(t) := \QQ\circ(x(t\wedge \tau_{m}),u(t\wedge \tau_{m}))^{-1}; m\geq 1)$ where $\tau_{m}$ is the $m$th-time $x(t)$ hits $\partial\Dd$. Since $(t,x)\mapsto B[x;\rhofp(t)]$ is uniformly bounded, owing to Girsanov transform and Theorem \ref{thm:LinearSolutionByStraightening},
one can easily check that $\QQ$ is absolutely continuous w.r.t. $\PP$ so that
the sequence $(\tau_{m}; m\geq 1)$ is well-defined and grows to $\infty$ under $\QQ$. As a first step, we show the existence of an $L^{\infty}((0,T);L^{2}(\Dd\times\er^{d}))$-density of $\mu^{m}$. Using a Riesz representation argument, it is sufficient to show that there exists $K>0$, possibly depending on $m$, such that
\begin{equation}\label{proofst1:thm:LinkMild&MartProb}
\forall t\in(0,T),~\forall \psi\in \Cc^{\infty}_{c}(\Dd\times\er^{d}),\,\mbox{ nonnegative},\quad\int_{\Dd\times\er^{d}}\psi(x,u)\mu^{m}(t,dx,du)\leq K\Vert \psi\Vert_{L^{2}(\Dd\times\er^d)}.
\end{equation}
To prove  \eqref{proofst1:thm:LinkMild&MartProb}, let us observe that for all nonnegative $\psi\in\Cc^{\infty}_{c}(\Dd\times\er^{d})$, and for all  $\alpha\in(1,+\infty)$, using Girsanov's change of probability and the boundedness of $B[\cdot;{\rhofp}(\cdot)]$,
\begin{equation}\begin{aligned}
\label{proofst2:thm:LinkMild&MartProb}
\int_{\Dd\times\er^{d}}\psi(x,u)\mu^{m}(t,dx,du)&=\EE_{\QQ}\left[\psi(x(t\wedge\tau_{m}),u(t\wedge\tau_{m}))\right]\\
&\leq\exp{\left\{\frac{\Vert b\Vert_{L^{\infty}(\er^{d};\er^{d})}T}{2(\alpha+1)}\right\}}\left(\EE_{\PP}\left[ \psi^{\alpha}(x(t\wedge\tau_{m}),u(t\wedge\tau_{m})) \right]\right)^{\frac{1}{\alpha}}.
\end{aligned}
\end{equation}
We will specify later the appropriate $\alpha$.
Now observe that
\begin{equation*}
\EE_{\PP}\left[\psi^{\alpha}(x(t\wedge\tau_{m}),u(t\wedge\tau_{m}))\right]=\left<\Gamma^{(\psi)^{\alpha}}_{m}(t),\rho_{0}\right>,
\end{equation*}
for $\Gamma^{(\psi)^{\alpha}}_{m}$ given as in \eqref{eq:EqMildKernel}. Let us observe that, according to Corollary \ref{coro:KernelLpEstimate},
for all $\beta\in(1,2)$,
\begin{equation*}
\Vert \Gamma^{(\psi)^{\alpha}}_{m}(t)\Vert^{\beta}_{L^{\beta}(\Dd\times\er^d)}+\Vert \Gamma^{(\psi)^{\alpha}}_{m}\Vert^{\beta}_{L^{\beta}(\Sigma^{-}_{t})}\leq \Vert\psi^{\alpha}\Vert^{\beta}_{L^{\beta}(\Dd\times\er^d)}+\Vert \Gamma^{(\psi^{\alpha})}_{m-1}\Vert^{\beta}_{L^{\beta}(\Sigma^{-}_{t})}.
\end{equation*}
Iterating this inequality $m$ times and since $\Gamma^{(\psi^{\alpha})}_{0}=\psi=0$ on $\Sigma_{T}$, one gets
\begin{equation*}
\Vert \Gamma^{(\psi)^{\alpha}}_{m}(t)\Vert^{\beta}_{L^{\beta}(\Dd\times\er^d)}\leq m \,\Vert \psi^{\alpha}\Vert^{\beta}_{L^{\beta}(\Dd\times\er^d)}.
\end{equation*}
It thus follows that, for $\beta'$ conjugate to $\beta$,
\begin{equation*}
\left|\left<\Gamma^{(\psi^{\alpha})}_{m}(t),\rho_{0}\right>\right|\leq \Vert \Gamma^{(\psi^{\alpha})}_{m}(t)\Vert_{L^{\beta}(\Dd\times\er^d)}\Vert \rho_{0}
\Vert_{L^{\beta^{'}}(\Dd\times\er^d)}\leq {m^{\frac{1}{\beta}}}\Vert\psi^{\alpha}\Vert^{\beta}_{L^{\beta}(\Dd\times\er^d)}\Vert \rho_{0}
\Vert^{\beta^{'}}_{L^{\beta^{'}}(\Dd\times\er^d)}.
\end{equation*}
Coming back to \eqref{proofst2:thm:LinkMild&MartProb}, we deduce that
\begin{equation}\label{proofst3:thm:LinkMild&MartProb}
\int_{\Dd\times\er^{d}}\psi(x,u)\mu^{m}(t,dx,du)\leq m^{\frac{1}{{\beta}\alpha}}\exp{\left\{\frac{\Vert b\Vert_{L^{\infty}}T}{2(\alpha+1)}\right\}}\Vert
\psi\Vert_{L^{\beta\alpha}(\Dd\times\er^d)}\left(\Vert \rho_{0}
\Vert_{L^{\beta^{'}}(\Dd\times\er^d)}\right)^{\frac{1}{\alpha}}
\end{equation}
For the special case where $\alpha$ and $\beta$ are such that $\beta\alpha=2$ and owing to \hypedpiii, we get \eqref{proofst1:thm:LinkMild&MartProb}
for a constant $K$ depending  on $m$, $T$, $\Vert\rho_{0}\Vert_{L^{2}(\Dd\times\er^d)}$ and ${\Vert b\Vert_{L^{\infty}(\er^{d};\er^{d})}}$.

Now applying the It\^{o} formula to $(s,x,u)\mapsto \Gamma^{\psi}_{n}(t-s,x,u)$ and using  Eq. \eqref{eq:backwardVFP-bis-bis}, it follows that, $\QQ$-a.s.,
\begin{align*}
\psi(x(t\wedge \tau_{m}),u(t\wedge \tau_{m}))&=\Gamma^{\psi}_{n}(t,x(0\wedge \tau_{m}),u(0\wedge \tau_{m}))\\
&\quad +\int_{0}^{t\wedge \tau_{m}}\nabla_{u}\Gamma^{\psi}_{n}(t-s,x(s),u(s))\,
\left(dw(s)+B[x(s);\rhofp(s)]\,ds\right)\\
&\quad+\sum_{k=0}^{m}\left(\Gamma^{\psi}_{n}(t-\tau_{k},x(\tau_{k}),u(\tau_{k}))-\Gamma^{\psi}_{n}(t-\tau_{k},x(\tau_{k}),u(\tau_{k}^{-}))\right)
\ind_{\{\tau_{m}\leq t\}}.
\end{align*}
Since $\Gamma^{\psi}_{n}(t-\tau_{k},x(\tau_{k}),u(\tau_{k}^{-}))=\Gamma^{\psi}_{n-1}(t-\tau_{k},x(\tau_{k}),u(\tau_{k}))$, taking the expectation on both sides of the equality yields
\begin{equation}
\label{proofst0:thm:LinkMild&MartProb}
\begin{aligned}
\left<\psi,\mu^{m}(t)\right>
&=\left<\Gamma^{\psi}_{n}(t),\rho_{0}\right>+\EE_{\QQ}\left[\int_{0}^{t\wedge\tau_{m}}\nabla_{u}\Gamma^{\psi}_{n}(t-s,x(s),u(s))B[x(s);\rhofp(s)]\,ds\right]\\
&\quad+\EE_{\QQ}\left[\sum_{k=0}^{m}\left(\Gamma^{\psi}_{n}(t-\tau_{k},x(\tau_{k}),u(\tau_{k}))-\Gamma^{\psi}_{n-1}(t-\tau_{k},x(\tau_{k}),u(\tau_{k}))
\right)\ind_{\{\tau_{k}\leq t\}}\right].
\end{aligned}
\end{equation}
We take the limit $n\rightarrow +\infty$ in \eqref{proofst0:thm:LinkMild&MartProb}. According to Lemma \ref{lem:abs-cont-wrt-surface-mesure}, $\sum_{k=0}^{m}\PP\circ(\tau_{k},x(\tau_{k}),u(\tau_{k}))^{-1}$ is a finite measure on $\Sigma^{-}_{T}$ and is absolutely continuous w.r.t. $\lambda_{\Sigma_{T}}$. By Girsanov Theorem, the same holds true for $\sum_{k=0}^{m}\QQ\circ(\tau_{k},x(\tau_{k}),u(\tau_{k}))^{-1}$. Hence using the $\lambda_{\Sigma_{T}}$-a.e. convergence given in \eqref{eq:KernelsConvergence} we get that
\begin{equation*}
\lim_{n\rightarrow+\infty}
\EE_{\QQ}\left[\sum_{k=0}^{m}\left(\Gamma^{\psi}_{n}(t-\tau_{k},x(\tau_{k}),u(\tau_{k}))-\Gamma^{\psi}_{n-1}(t-\tau_{k},x(\tau_{k}),u(\tau_{k}))\right)\right]=0.
\end{equation*}
Since $\Gamma^{\psi}_{n}$ and $\nabla_{u}\Gamma^{\psi}_{n}$ converge to $\Gamma^{\psi}$ and $\nabla_{u}\Gamma^{\psi}$ in $L^{2}(Q_{T})$
and since $\mu^{m}\in L^{\infty}((0,T);L^2(\Dd\times\er^{d}))$ we get
\begin{equation*}
\left<{\psi},\mu^{m}(t)\right>=\left<\Gamma^{\psi}(t),\rho_{0}\right>+\EE_{\QQ}\left[\int_{0}^{t\wedge\tau_{m}}
\left(\nabla_{u}\Gamma^{\psi}(t-s,x(s),u(s)),B[x(s);\rhofp(s)]\right)\,ds\right]
\end{equation*}
Next, one can observe that, for all $t\in(0,T)$,
\begin{align}\label{proofst4:thm:LinkMild&MartProb}
\left<\psi,\mu^{m}(t)\right>
&=\int_{\Dd\times\er^{d}}|\Gamma^{\psi}(t)|(x,u)|\rho_{0}| (x,u)\,dx\,du+\int_{0}^t\left(\nabla_{u}\Gamma^{\psi}(t,s),B[x;\rhofp(s)]\right)\mu^{m}(s,x,u)\,ds\,dx\,du\\
&\leq \Vert\Gamma^{\psi}(t)\Vert_{L^{2}(\Dd\times\er^{d})}\Vert\rho_{0}\Vert_{L^{2}(\Dd\times\er^{d})}+\Vert b\Vert_{L^{\infty}(\er^{d};\er^{d})}\Vert\nabla_{u}\Gamma^{\psi}\Vert_{L^{2}(Q_{T})}\Vert\mu^{m}\Vert_{L^{2}(Q_{t})}.\nonumber
\end{align}
By Gronwall's Lemma, this shows that $\mu^{m}$ is  bounded in $L^{2}(Q_{T})$ uniformly in $m$. In addition, for all $t\in(0,T]$,
$\mu^{m}(t)$ converges weakly toward $\QQ\circ(x(t),u(t))^{-1}$. Thus, using again a Riesz representation argument, we deduce that $\QQ$ admits a set of time marginal densities $(\rho(t);\,t\in[0,T])$ in $L^{2}(Q_{T})$. Taking the limit $m\rightarrow +\infty$ in \eqref{proofst4:thm:LinkMild&MartProb}, we further conclude that $(\rho(t);\,t\in[0,T])$ in $L^{2}(Q_{T})$ is solution to \eqref{eq:mild-bis}.
\end{proof}
\subsection{Uniqueness}
Under \hyp, let us observe that any solution in law $\QQ\in\Pi_{\omega}$ to \eqref{eq:NonlinearConfinedLangevin} has  time marginal densities $\rho(t)\in L^{2}(\omega,\Dd\times\er^{d})$, solution of the following weighted nonlinear mild equation: for all $t\in(0,T]$, for all $\psi\in L^{2}(\Dd\times\er^{d})$,
\begin{equation}\label{eq:EqNonlinearMildKernel}
\begin{aligned}
\langle \sqrt{\omega}\psi,\rho(t)\rangle &= \langle \Gamma^{\psi}(t),\sqrt{\omega}\rho_{0}\rangle +\int_{0}^{t}\langle \nabla_{u}\Gamma^{\psi}(t-s),\sqrt{\omega}B[\cdot;\rho(s)]\rho(s)\rangle\,ds\\
&\quad+\int_{0}^{t}\langle\Gamma^{\psi}(t-s), (\nabla_{u}\log(\sqrt{\omega})\cdot \sqrt{\omega} B[\cdot;\rho(s)])\rho(s)\rangle \,ds\\
&\quad 5+\frac{\sigma^{2}}{2}\int_{0}^{t}\langle( \triangle_{u}\sqrt{\omega})\Gamma^{\psi}(t-s), \rho(s)\rangle \,ds.
\end{aligned}
\end{equation}
We prove \eqref{eq:EqNonlinearMildKernel} by  replicating  some proof steps of Proposition~\ref{thm:LinkMild&MVFK}-$(ii)$: for fixed nonnegative $\psi\in \Cc^{\infty}_{c}(\Dd\times\er^{d})$, $t\in(0,T]$, and  $m, n\in\nat^*$, using Corollary \ref{coro:KernelSmoothness}, It\^o's formula applied to $(s,x,u)\in Q_{t}\mapsto \sqrt{\omega(u)}\Gamma^{\psi}_{n}(t-s,x,u)$ yields
\begin{align*}
&\EE_{\QQ}\left[\sqrt{\omega(u(t))}\psi(x(t\wedge\tau_{m}),u(t\wedge\tau_{m}))\right]\\
&=\EE_{\QQ}\left[\sqrt{\omega(u(0))}\Gamma_{n}^\psi(t,x(0),u(0))\right]+
\EE_{\QQ}\left[\int_{0}^{t\wedge\tau_{m}}\sqrt{\omega(u(s))}\left(\nabla_{u}\Gamma_{n}^\psi(t-s,x(s),u(s))\cdot B[x(s);\rho(s)]\right)\,ds\right]\\
&\quad +\EE_{\QQ}\left[\int_{0}^{t\wedge\tau_{m}}\Gamma_{n}^\psi(t-s,x(s),u(s))\left(\left(\nabla_{u}\sqrt{\omega(u(s))}\cdot B[x(s);\rho(s)]+\frac{\sigma^{2}}{2}\triangle\sqrt{\omega}(u(s))\right)\right)\,ds\right]\\
&\quad +\EE_{\QQ}\left[\sum_{k=0}^{m}\sqrt{\omega(u(\tau_{k}))}\left(\Gamma^{\psi}_{n}(t-\tau_{k},x(\tau_{k}),u(\tau_{k}))-\Gamma^{\psi}_{n-1}
(t-\tau_{k},x(\tau_{k}),u(\tau_{k}))\right)\ind_{\{\tau_{k}\leq t\}}\right].
\end{align*}
Then taking the limit $n\rightarrow +\infty$, the boundary term vanishes.  Next as $m\rightarrow +\infty$,  we prove  \eqref{eq:EqNonlinearMildKernel}
for any nonnegative test function $\psi\in \Cc_{c}(\Dd\times\er^{d})$. This result extends to any test function $\psi\in \Cc_{c}(\Dd\times\er^{d})$, using the linearity of $\psi\mapsto \Gamma^\psi$ and the decomposition $\psi = (\psi)^{+} -(\psi)^{-}$.
Since the drift coefficient in \eqref{eq:NonlinearConfinedLangevin} is bounded, the fact that two solutions in law of \eqref{eq:NonlinearConfinedLangevin} coincide is equivalent with the equality between the time marginal densities of these two solutions. We thus conclude the proof of the uniqueness part of Theorem \ref{thm:main} with the following.
\begin{lemma}\label{lem:NonlinearUniquenessResult}
Under \hyp, any solution
$\rho(t)\in L^{2}(\omega;\Dd\times\er^{d})$ to the non-linear mild equation \eqref{eq:EqNonlinearMildKernel} is equal to $\rhofp(t)$ for all $t\in[0,T]$.
\end{lemma}
\begin{proof} Using \eqref{eq:EqNonlinearMildKernel}, for all $t\in[0,T]$, we have
\begin{align*}
&\Vert\rhofp(t)-{\rho}(t)\Vert^{2}_{L^2(\omega,\Dd\times\er^d)}\\
&=\Vert\sqrt{\omega}(\rhofp(t)-{\rho}(t))\Vert^{2}_{L^{2}(\Dd\times\er^{d})}=\left(\sup_{\displaystyle\substack{\psi\in \Cc_c(\Dd\times\er^{d});\\ \Vert \psi\Vert_{L^2(\Dd\times\er^d)}=1}}
 \langle \sqrt{\omega}\psi,\rhofp(t)-{\rho}(t)\rangle \right)^{2}\\
&\leq2\sup_{\displaystyle \substack{\psi\in \Cc_c(\Dd\times\er^{d});\\\Vert \psi\Vert_{L^2(\Dd\times\er^d)}=1}}\left(\int_{0}^{t}\langle \nabla_{u}\Gamma^{\psi}(t-s)+\Gamma(t-s)\nabla_{u}\log(\sqrt{\omega}),\sqrt{\omega}\left(B[\cdot;\rhofp(s)]\rhofp(s)-B[\cdot;{\rho}(s)]{\rho}(s)\right)\rangle\,ds\right)^{2}\\
&\quad +\sigma^{2}\sup_{\displaystyle \substack{\psi\in \Cc_c(\Dd\times\er^{d});\\\Vert \psi\Vert_{L^2(\Dd\times\er^d)}=1}} \left(\int_{0}^{t} \langle\Gamma^{\psi}(t-s)\triangle_{u}\sqrt{\omega}, \rhofp(s)-{\rho}(s)\rangle \,ds\right)^{2}.
\end{align*}
Using Cauchy-Schwarz's inequality and Lemma \ref{lem:major_poids}, it follows that
\begin{align*}
&\Vert\rhofp(t)-{\rho}(t)\Vert^{2}_{L^2(\omega,\Dd\times\er^d)}\\
&\leq 2\sup_{\displaystyle \substack{\psi\in \Cc_c(\Dd\times\er^{d});\\\Vert \psi\Vert_{L^2(\Dd\times\er^d)}=1}}\left(\Vert\nabla_{u}\Gamma^{\psi} \Vert^{2}_{L^2(Q_{t})}+\frac{\alpha}{4}\Vert\Gamma^{\psi}\Vert^{2}_{L^2(Q_{t})}
\right)\int_{0}^{t}\Vert  \rhofp(s)B[\cdot;\rhofp(s)]-{\rho}(s)B[\cdot;{\rho}(s)]\Vert^{2}_{L^2(\omega,\Dd\times\er^{d})}\,ds\\
&\quad + \sigma^{2}\left(2\alpha(\frac{\alpha}{2}-1)+\alpha d\right)\left(\sup_{\displaystyle \substack{\psi\in\Cc_c(\Dd\times\er^{d});\\\Vert \psi\Vert_{L^2(\Dd\times\er^d)}=1}}\Vert\Gamma^{\psi}\Vert^{2}_{L^2(Q_{t})}\right)\int_{0}^{t}\Vert\rhofp(s)-{\rho}(s)\Vert^{2}_{L^{2}(\omega,\Dd\times\er^{d})}\,ds.
\end{align*}
Using Corollary \ref{coro:KernelEstimate}, one  can deduce that
\begin{equation}
\label{prst1}
\begin{aligned}
\Vert\rhofp(t)-{\rho}(t)\Vert^{2}_{L^2(\omega,\Dd\times\er^d)}&\leq \left(\frac{2}{\sigma^{2}}+\frac{\alpha T}{4}\right)\int_{0}^{t}
\Vert \rhofp(s)B[\cdot;\rhofp(s)]-{\rho}(s)B[\cdot;{\rho}(s)]\Vert^2_{L^2(\omega,\Dd\times\er^{d})}\,ds\\
&\quad +\left(2\alpha(\frac{\alpha}{2}-1)+\alpha d\right)\int_{0}^{t}\Vert\rhofp(s)-{\rho}(s)\Vert^{2}_{L^{2}(\omega,\Dd\times\er^{d})}\,ds.
\end{aligned}
\end{equation}
Now observe that
\begin{align*}
&\int_{Q_{T}}\omega(u)\left(\rhofp(s,x,u)B[x;\rhofp(s)]-{\rho}(s,x,u)B[x;{\rho}(s)]\right)^{2}\,ds\,dx\,du\\
&\leq\Vert b\Vert^{2}_{L^{\infty}(\er^{d};\er^{d})}\Vert\rhofp-{\rho}\Vert^{2}_{L^{2}(\omega,Q_{T})}\\
&\quad +\int_{(0,T)\times\Dd}\left(\int_{\er^{d}}\omega(u)\left(\rhofp(s,x,u)\right)^{2}\,du\right)
\left(B[x;\rhofp(s)]-B[x;{\rho}(s)]\right)^{2}\,ds\,dx
\end{align*}
and that
\begin{align*}
\left(B[x;\rhofp(s)]-B[x;{\rho}(s)]\right)&=
\dfrac{\int_{\er^{d}}b(v)\left(\rhofp(s,x,v)-{\rho}(s,x,v)\right)\,dv}{\int_{\er^{d}}\rhofp(s,x,v)\,dv}\\
&\quad+\dfrac{\int_{\er^{d}}b(v)\rhofp(s,x,v)\,dv\int_{\er^{d}}\left(\rhofp(s,x,v)-{\rho}(s,x,v)\right)\,dv}{\left(\int_{\er^{d}}\rhofp(s,x,v)\,dv\right)\left(\int_{\er^{d}}{\rho}(s,x,u)\,dv\right)}.
\end{align*}
Using the Maxwellian bounds $\overline{P}$ and $\underline{P}$ of $\rhofp$ given in Theorem \ref{thm:Existence_result}, one has
\begin{align*}
&\int_{\er^{d}}\omega(u)
\left(\rhofp(s,x,u)\right)^{2}\,du\left(\frac{\int_{\er^{d}}b(v)\left(\rhofp(s,x,v)-{\rho}(s,x,v)\right)\,dv}
{\int_{\er^{d}}\rhofp(s,x,v)\,dv}\right)^{2}\\
&\leq \dfrac{\int_{\er^{d}}\omega(u)\left(\overline{P}(s,u)\right)^{2}\,du}{\left(\int_{\er^{d}}\underline{P}(s,u)\,du\right)^{2}}
\left|\int_{\er^{d}}b(v)\left(\rhofp(s,x,v)-{\rho}(s,x,v)\right)\,dv\right|^{2}\\
&\leq 2\Vert b\Vert_{L^{\infty}(\er^{d};\er^{d})}\dfrac{\sup_{t\in[0,T]}\int_{\er^{d}}\omega(u)\left(\overline{P}(t,u)\right)^{2}\,du}{
\inf_{t\in[0,T]}\left(\int_{\er^{d}}\underline{P}(t,u)\,du\right)^{2}}\int_{\er^{d}}\omega^{-1}(v)\,dv
\int_{\er^{d}}\omega(v)\left|\rhofp(s,x,v)-{\rho}(s,x,v)\right|^{2}\,dv,
\end{align*}
and
\begin{align*}
&\int_{\er^{d}}\omega(u)
\left( \rhofp(s,x,u)\right)^{2}\,du \left(\frac{\int_{\er^{d}}b(v)\rhofp(s,x,v)\,dv\int_{\er^{d}}\left(\rhofp(s,x,v)-{\rho}(s,x,v)\right)\,dv}{\left(\int_{\er^{d}}\rhofp(s,x,v)\,dv\right)\left(\int_{\er^{d}}{\rho}(s,x,u)\,dv\right)}
\right)^{2}\\
&\leq \Vert b\Vert_{L^{\infty}(\er^{d};\er^{d})}\frac{\int_{\er^{d}}\omega(u)\left(\rhofp(s,x,u)\right)^{2}\,du}{\left(\int_{\er^{d}}\rhofp(s,x,v)\,dv\right)^{2}}
\left(\int_{\er^{d}}b(v)\left(\rhofp(s,x,v)-{\rho}(s,x,v)\right)\,dv\right)^{2}\\
&\leq \Vert b\Vert_{L^{\infty}(\er^{d};\er^{d})}\frac{\sup_{t\in(0,T)}\int_{\er^{d}}\omega(u)\left(\overline{P}(t,u)\right)^{2}\,du}{\inf_{t\in(0,T)}
\left(\int_{\er^{d}}\underline{P}(t,u)\,du\right)^{2}}
\int_{\er^{d}}\omega(v)\left|\rhofp(s,x,v)-{\rho}(s,x,v)\right|^{2}\,dv,
\end{align*}
Therefore, for some constant $C>0$,
\begin{equation*}
\int_{(0,T)\times\Dd}\left(\int_{\er^{d}}\omega(u)\left(\rhofp(s,x,u)\right)^{2}\,du\right)
\left(B[x;\rhofp(s)]-B[x;{\rho}(s)]\right)^{2}\,ds\,dx\leq 2C\Vert \rhofp-{\rho}\Vert^{2}_{L^{2}(\omega,Q_{t})}
\end{equation*}
Coming back to \eqref{prst1}, we deduce
\begin{align*}
&\Vert\rhofp(t)-{\rho}(t)\Vert^{2}_{L^2(\omega,\Dd\times\er^d)}\\
&\leq C \sup_{\psi\in \Cc_c(\Dd\times\er^d);\Vert \psi\Vert_{L^2(\Dd\times\er^d)}=1}\int_{0}^{t}\Vert\nabla_{u}\Gamma^{\psi}(t-s)\Vert^2_{L^2(\omega,\Dd\times\er^{d})}\,ds\int_{0}^{t}\Vert\rhofp(s)-{\rho}(s)\Vert^{2}_{L^2(\omega,\Dd\times\er^d)}\,ds.
\end{align*}
where $C$ depends  on $\rho_0, d,\sigma$ and $\alpha$. Using Gronwall's Lemma, this ends the proof.
\end{proof}
\paragraph{Acknowledgment. } The authors would like to thank Pr. Pierre-Louis Lions for having pointed to us the PDE approach developed in Section \ref{sec:PDEApproach} and for his helpful remarks.

\appendix
\section{Appendix}
\subsection{Some recalls}
\begin{corollary}[Rana \cite{IKRana2002}]
\label{coro:LebesgueContinuity} If $\phi\in L^{p}(\er^{d})$ for $p\in[1,+\infty)$ then
\begin{equation*}
\lim_{|\delta|\rightarrow 0^{+}}\int | \phi(z +\delta)-\phi(z)|^{p}dz=0.
\end{equation*}
\end{corollary}
\begin{theorem}[Tartar \cite{tartar-78}, Chapter 4]
\label{thm:SobolevRegNegativePart}
Let $\mathcal{V}$ be an open subset of $\er^{d}$ and $\psi\in L^{2}(\mathcal{V})$
such that $\nabla_{v}\psi\in L^{2}(\mathcal{V})$. Then $\nabla_{v}(\psi)^{+}, \nabla_{v}(\psi)^{-}
\in L^{2}(\mathcal{V})$ with
$\partial_{v_{i}} (\psi)^{+}=\partial_{v_{i}}\psi \ind_{\{\psi \geq 0\}}$ and $\partial_{v_{i}} (\psi)^{-}=-\partial_{v_{i}}\psi\ind_{\{\psi \leq 0\}}$.
\end{theorem}
\begin{theorem}[Lions and Magenes \cite{Lions1972}]
\label{Existence_JLLions}
Let $E$ be a Hilbert space with the inner product $\left(~,~\right)_{E}$.
Let $F\subset E$ equipped with the norm $|~|_{F}$
such that the canonical injection of $F$ into $E$ is continuous. Assume that $A:E\times F\rightarrow \er$ is a bilinear
application satisfying:
\begin{enumerate}
\item $\forall~\psi\in F$, the mapping $A(.,\psi):E\rightarrow \er$ is continuous.
\item $A$ is coercive on $F$ that is there exists a constant $c>0$ such that
$A(\psi,\psi)~\geq~c |\psi|^{2}_{F},~\forall~\psi\in F$.
\end{enumerate}
Then for all linear application $L~:~F\rightarrow \er$, continuous on $(F,~|~|_{F})$, there exists
 $S\in E$ such that $A(S,\psi)=L(\psi),~\forall~\psi\in F$.
\end{theorem}

\subsection{Proof of Proposition \ref{prop:Schauder_estimate}}\label{sec:appendix-shauder}
To prove this proposition, it is sufficient to show that, for all $z_0:=(t_0,x_0,u_0)$ in $Q_T$, there exists $r>0$ such that $f$ belongs to  $\Cc^{1,1,2}(B_{z_0}(r))$ where $B_{z_0}(r)\subset Q_T$ is the open ball centered at $z_0$ of radius $r$. To this end, we use the Sobolev embeddings (see e.g. \cite{brezis-11}, Corollary $9.15$):  \\
for $m= \lfloor d/2\rfloor + 2  - \left\lfloor 1 - d/2 - \lfloor d/2 \rfloor \right \rfloor$,  we have\footnote{For $\lfloor x \rfloor$ the nearest integer lower than $x \in \er^+$.}
\begin{equation*}
W^{2,2}((0,T))\subset \Cc^{1}([0,T]),\quad W^{m,2}(B_{x_0}(r))\subset \Cc^{1}(\overline{B_{x_0}(r)}),\quad W^{m+1,2}(B_{u_0}(r))\subset \Cc^{2}(\overline{B_{u_0}(r)}).
\end{equation*}
We thus have to prove that for some $r>0$,
\begin{equation}\label{proof_step1:Schauder estimate}
\Vert\partial^{2}_{t} f\Vert_{L^{2}(B_{z_0}(r))}+\sum_{\eta\in\mathbb{N}^{d};|\eta|\leq m}\Vert D^{\eta}_{x}f\Vert_{L^{2}(B_{z_0}(r))}+\sum_{\kappa\in\mathbb{N}^{d};|\kappa|\leq m+1}\Vert D^{\kappa}_{u}f\Vert_{L^{2}(B_{z_0}(r))}<+\infty,
\end{equation}
where $D^{\eta}_x$ and $D^{\kappa}_{u}$ refer to the differential operators given by
\begin{align*}
D^{\eta}_{x}f = & \partial^{\eta_1}_{x_{1}} \partial^{\eta_2}_{x_2} \cdots \partial^{\eta_d}_{x_{d}} f\mbox{, for }\, \eta=(\eta_1,\eta_2,\cdots,\eta_d) \in \NN^{d},\\
D^{\kappa}_{u}f = &\partial^{\kappa_1}_{u_{1}} \partial^{\kappa_2}_{u_2} \cdots \partial^{\kappa_d}_{u_{d}} f\mbox{, for }\, \kappa=(\kappa_1,\kappa_2,\cdots,\kappa_d) \in \NN^{d}.
\end{align*}

The proof of \eqref{proof_step1:Schauder estimate}, is based on a bootstrap argument that uses the regularity results (in fractional Sobolev spaces)  obtained in Bouchut \cite{bouchut-02}  for the solution to kinetic equation (see Theorem \ref{thm:Bouchut-hypoellipticity}).

\paragraph{Step 1.} Let us start with the regularity along the $(x,u)$-variables.  We proceed by induction on a truncated version of $f$.

For any $r_0>0$ such that $B_{z_{0}}(r_0)\subsetneq Q_{T}$, we denote by $\beta_{r_0}:Q_{T}\rightarrow [0,1]$, a $\Cc^{\infty}_{c}(Q_{T})$-cutoff function such that
\begin{align*}
\left\{\begin{array}{l}
\beta_{r_0}=1 \mbox{ on }\overline{B_{z_0}(\frac{r_0}{2})},\\
\beta_{r_0}=0 \mbox{ on }Q_{T}\setminus B_{z_0}(r_0).
\end{array}\right.
\end{align*}
We further assume that there exists a constant $C$ depending on $r_0$ such that
\begin{equation*}
\sum_{\eta\in\mathbb{N}^{d};|\eta|\leq m + 1;\,\beta\in\mathbb{N}^{d};|\beta|\leq m+2}\Vert \partial^{2}_{t} D^{\eta}_{x} D^{\beta}_{u}\beta_{{r_0}}\Vert_{L^{\infty}(Q_{T})}\leq C.
\end{equation*}
Starting from $f\in L^{2}((0,T)\times\Dd;H^{1}(\er^{d}))$ given in Proposition \ref{propo:VFPDirWellposed},
 the truncated function $f_{{r_0}}:=\beta_{{r_0}} f$ satisfies, in the sense of distributions,
\begin{equation*}
\left\{
\begin{aligned}
&\partial_{t}f_{{r_0}}-(u\cdot \nabla_{x} f_{{r_0}})-\frac{\sigma^2}{2}\triangle_{u}f_{{r_0}}=\Gamma_{r_0} f+(\Psi_{r_0} \cdot\nabla_{u}f),\,\mbox{  on }\,Q_{T},\\
&f_{{r_0}}|_{t=0}=0,\,\mbox{  on }\,\Dd\times\er^{d},\\
&\gamma^{\pm}(f_{{r_0}})=0,\,\mbox{  on }\,\Sigma^{\pm}_{T},
\end{aligned}\right.
\end{equation*}
with $\Gamma_{r_0}:=\partial_t\beta_{{r_0}}-(u\cdot \nabla_x \beta_{{r_0}})-\frac{\sigma^2}{2}\triangle_{u}\beta_{{r_0}}$ and $\Psi_{r_0}:=-\sigma^{2}\nabla_{u}\beta_{{r_0}}$. Extending $f_{r_0}$, $\Gamma_{r_0} f$ and $(\Psi_{r_0}\cdot\nabla_u f)$ on the whole space $\er\times\er^d\times\er^d$ by $0$ outside $B_{z_0}({r_0})$, one has
\begin{equation}\label{eq:TruncatedVlasovFokkerPlanck}
\partial_{t}f_{{r_0}}-(u\cdot \nabla_{x}f_{{r_0}})-\frac{\sigma^2}{2}\triangle_{u}f_{{r_0}}=g_{{r_0}},\,\mbox{  in }\,(\Cc^{\infty}_{c}(\er\times\er^d\times\er^d))'
\end{equation}
where $g_{{r_0}}:=\Gamma_{r_0} f+(\Psi_{r_0}\cdot\nabla_{u}f)$. Let us now recall Theorem $1.5$ (and its proof) in \cite{bouchut-02}: for $\alpha\in (0,1)$, we further denote by $D^{\alpha}_x$ the fractional derivative   w.r.t. $x$-variables,  defined as the fractional Laplace operator of order $\alpha$
\begin{equation*}
D^{\alpha}_x=(-\triangle_x)^{\alpha/2}.
\end{equation*}
\begin{theorem}[Bouchut \cite{bouchut-02}]\label{thm:Bouchut-hypoellipticity} Let $h\in L^2(\er\times\er^{d}\times\er^{d})$. Assume that $\phi\in L^2(\er\times\er^{d}\times\er^{d})$, such that $\nabla_u \phi\in (L^{2}(\er\times\er^{d}\times\er^{d}))^{d}$, satisfies (in the sense of distributions)
\begin{equation}
\label{eq:KineticFokkerPlanck}
\partial_t \phi+(u\cdot \nabla_x \phi)-\frac{\sigma^2}{2}\triangle_u \phi=h,\,\mbox{  on }\,\er\times\er^{d}\times\er^{d}.
\end{equation}
Then there exists a positive constant $C(d)$  depending  on the dimension such that:
\begin{description}
\item[(a)] $\partial_t \phi+(u\cdot\nabla_x \phi)$ and $\triangle_u \phi$ both belong to $L^2(\er\times\er^{d}\times\er^{d})$ with
\begin{equation*}
\Vert \partial_t \phi+(u\cdot\nabla_x \phi)\Vert_{L^2(\er\times\er^{d}\times\er^{d})}+\frac{\sigma^2}{2}\Vert\triangle_u \phi \Vert_{L^2(\er\times\er^{d}\times\er^{d})}\leq C(d)\Vert h\Vert_{L^2(\er\times\er^{d}\times\er^{d})},
\end{equation*}
\item[(b)] $D^{2/3}_x \phi$ and $|\nabla_u D^{1/3}_x \phi|$ belong to  $L^2(\er\times\er^{d}\times\er^{d})$ with
\begin{align*}
&\Vert \nabla_u D^{1/3}_x\phi\Vert^2_{L^2(\er\times\er^d\times\er^d)} + \Vert D^{2/3}_x \phi \Vert_{L^2(\er\times\er^{d}\times\er^{d})}\leq C(d)\Vert h\Vert_{L^2(\er\times\er^{d}\times\er^{d})}.
\end{align*}
\end{description}
\end{theorem}
Since $\Gamma_{r_0} f$ and $(\Psi_{{r_0}}\cdot\nabla_u f)$ are in $L^2(\er\times\er^d\times\er^d)$, Theorem \ref{thm:Bouchut-hypoellipticity}-(b) implies that
$D^{2/3}_x f_{r_0}$, $|\nabla_u D^{1/3}_x f_{r_0}|$, and $\triangle_u f_{r_0}$ are in $L^2(\er\times\er^d\times\er^d)$. As $\beta_{{r_0}} =1$ on $B_{z_0}(\frac{{r_0}}{2})$, this particularly ensures that
\begin{align*}
&\Vert D^{2/3}_x f\Vert_{B_{z_0}(\frac{{r_0}}{2})}=\Vert D^{2/3}_x f_{{r_0}}\Vert_{B_{z_0}(\frac{{r_0}}{2})}\leq \Vert D^{2/3}_x f_{{r_0}}\Vert_{L^{2}(\er\times\er^{d}\times\er^{d})}<+\infty,\\
&\Vert \nabla_u D^{1/3}_x f\Vert_{B_{z_0}(\frac{{r_0}}{2})}=\Vert \nabla_u D^{1/3}_x f_{{r_0}}\Vert_{B_{z_0}(\frac{{r_0}}{2})}\leq \Vert \nabla_uD^{1/3}_x f_{{r_0}}\Vert_{L^{2}(\er\times\er^{d}\times\er^{d})}<+\infty.
\end{align*}
By setting $r_1:=\frac{r_0}{2}$ and $f_{{r_1}}:=\beta_{{r_1}}f$,  it follows that    $D^{1/3}_x f_{{r_1}}\in L^{2}(\er\times\er^d\times\er^d)$ (since\footnote{This can be shown by applying a  Cauchy-Schwarz inequality in the alternative definition of the fractional derivative in $L^2$ via Fourier transform, see e.g. \cite{DiNezza-Palatucci-all-12}.}
$\Vert D^{1/3}_{x}f_{{r_1}}\Vert^2_{L^{2}(\er\times\er^{d}\times\er^{d})}\leq \Vert f_{{r_1}}\Vert_{L^{2}(\er\times\er^{d}\times\er^{d})}\Vert
D^{2/3}_{x}f_{{r_1}}\Vert_{L^{2}(\er\times\er^{d}\times\er^{d})}$.) Furthermore,
\begin{equation*}
D^{1/3}_{x}g_{{r_1}}=\Gamma_{{r_1}}D^{1/3}_{x}f
+f D^{1/3}_{x}\Gamma_{{r_1}}
+\left(\Psi_{{r_1}}\cdot \nabla_{u}D^{1/3}_{x}f\right)
+\left(D^{1/3}_{x}\Psi_{{r_1}}\cdot \nabla_{u}f\right)\in L^{2}(\er\times\er^{d}\times\er^{d}).
\end{equation*}
Applying the differential operator $D^{1/3}_x$ to \eqref{eq:TruncatedVlasovFokkerPlanck}, one can check that
$D^{1/3}_x f_{{r_1}}$  satisfies
\begin{equation}
\label{proof_step2:Schauder estimate}
\partial_{t}D^{1/3}_x f_{{r_1}}-(u\cdot \nabla_{x}D^{1/3}_x f_{{r_1}})-\frac{\sigma^2}{2}\triangle_{u}D^{1/3}_x f_{{r_1}}= D^{1/3}_x g_{{r_1}},\,\mbox{  in }\,(\Cc^{\infty}_c(\er\times\er^d\times\er^d))'.
\end{equation}
From Theorem \ref{thm:Bouchut-hypoellipticity}-(b) again, we obtain that $|\nabla_{x}f_{{r_1}}| \leq C |D^{2/3}_x(D^{1/3}_x f_{{r_1}})| \in L^{2}(\er\times\er^d\times\er^d)$,
 and $|\nabla_u D^{2/3}_x f_{{r_1}}|\in L^{2}(\er\times\er^d\times\er^d)$. Therefore,
$|\nabla_x f|\in L^2(B_{z_0}(\frac{r_1}{2}))$. Applying again $D^{1/3}_x$ to \eqref{proof_step2:Schauder estimate}, applying Theorem \ref{thm:Bouchut-hypoellipticity}-(b) a third time, one can also deduce that
$|\nabla_{u}\nabla_x f|$ is in $L^{2}(B_{z_0}(\frac{r_0}{2^{3}}))$.

We obtain the regularity w.r.t. $u$ by applying the differential operator
$\partial_{u_i}$ to Eq. \eqref{eq:TruncatedVlasovFokkerPlanck}. Hence  $\partial_{u_i} f_{{r_1}}$ satisfies
\begin{equation}\label{proof_step4:Schauder estimate}
\partial_{t}\partial_{u_i}f_{{r_1}}-(u\cdot \nabla_{x}\partial_{u_i}f_{{r_1}})-\frac{\sigma^2}{2}\triangle_{u}\partial_{u_i}f_{{r_1}}=\partial_{u_i}g_{{r_1}} +\partial_{x_i}f_{{r_1}},\,\mbox{  in }\,(\Cc^{\infty}_c(\er\times\er^d\times\er^d))',
\end{equation}
where
\begin{equation*}
\partial_{u_i}g_{{r_1}}=(\partial_{u_i}\Gamma_{{r_1}}) f + \Gamma_{{r_1}} \partial_{u_i} f +  \left(\Psi_{{r_1}}\cdot \nabla_u \partial_{u_i}f\right)+ \left(\partial_{u_i}\Psi_{{r_1}}\cdot \nabla_u f\right).
\end{equation*}
Theorem \ref{thm:Bouchut-hypoellipticity}-(a)  ensures that $\Vert \triangle_u f_{r_0}\Vert_{L^{2}(\er\times\er^d\times\er^d)}<+\infty$. As $ f_{r_0}$ has a compact support, standard arguments give that
\begin{equation*}
\sum_{1\leq i,j\leq d}\Vert\partial^{2}_{u_i,u_j} f_{{r_0}}\Vert^2_{L^{2}(\er\times\er^d\times\er^d)} =
\Vert\triangle_u f_{{r_0}}\Vert^2_{L^{2}(\er\times\er^d\times\er^d)} < +\infty
\end{equation*}
and thus
\begin{equation*}
\sum_{1\leq i,j\leq d}\Vert\partial^{2}_{u_i,u_j} f \Vert^2_{L^{2}(B_{z_0}(r_1))}  = \sum_{1\leq i,j\leq d}\Vert\partial^{2}_{u_i,u_j} f_{r_0}\Vert_{L^{2}(B_{z_0}(r_1))}  < + \infty.
\end{equation*}
Now we set $h = \partial_{u_i}g_{{r_1}} +\partial_{x_i}f_{{r_1}}$ with $\|h\|_{L^{2}(\er\times\er^d\times\er^d)}\leq \|\nabla_u g_{{r_1}}\|_{L^{2}(\er\times\er^d\times\er^d)} + \|\nabla_x f_{{r_1}}\|_{L^{2}(\er\times\er^d\times\er^d)} < + \infty$, since
\begin{equation*}
\Vert \nabla_{u}g_{{r_1}}\Vert_{L^{2}(\er\times\er^d\times\er^d)}\leq C  \big(
\Vert f\Vert_{L^{2}(B_{z_0}({r_1}))} + \Vert \nabla_u f\Vert_{L^{2}(B_{z_0}({r_1}))} +\sum_{1\leq i,j\leq d}\Vert\partial^{2}_{u_i,u_j} f \Vert^2_{L^{2}(B_{z_0}(r_1))}  \big)<+\infty.
\end{equation*}
Theorem \ref{thm:Bouchut-hypoellipticity}-(a)  ensures that $|\nabla_{u}(\triangle_{u} f_{{r_1}})|\in L^{2}(\er\times\er^d\times\er^d)$
 and hence that $|\triangle_{u} \nabla_{u} f|\in L^{2}(B_{z_0}(\frac{r_1}{2}))$.

We sum up the estimations we have obtained as
\begin{equation}\label{proof_step3:Schauder estimate}
\Vert \nabla_x f\Vert_{L^{2}(B_{z_0}(\frac{{r_0}}{2^3}))}+\Vert \nabla_{x}\nabla_{u} f \Vert_{L^{2}(B_{z_0}(\frac{{r_0}}{2^3}))}+
\Vert \triangle_u\nabla_u f\Vert_{L^{2}(B_{z_0}(\frac{{r_0}}{2^3}))}<+\infty.
\end{equation}

We extend \eqref{proof_step3:Schauder estimate} to higher order differentials  through the following induction argument: we have proved  that for $N=1$,
\begin{equation*}
D^{\eta}_{x}f,\;
 |\nabla_{u}D^{\eta}_{x}f|, \;
 |\nabla_x D^{\eta'}_{u} f|, \;
 |\nabla_{u}D^{\eta}_{u}f|\,\mbox{are all in}\,L^{2}(B_{z_0}(R_N)),\,\mbox{  for all }\,\eta\in\mathbb{N}^{d}\,\mbox{such that}\,1\leq |\eta| \leq N,
\end{equation*}
with $R_N = r_0/2^{3N}$ and $\eta'\in \nat^d$ is such that $|\eta'| = |\eta|-1$.

Starting from the induction assumption that $\Vert D^{\eta}_{x}f\Vert_{L^{2}(B_{z_0}(R_N))}+\Vert \nabla_{u}D^{\eta}_{x}f\Vert_{L^{2}(B_{z_0}(R_N))}<+\infty$, for $|\eta| \leq N$, we have that $D^{\eta}_{x}f_{R_{N}}$ satisfies
\begin{equation*}
\partial_{t}D^{\eta}_{x}f_{R_{N}}-(u\cdot \nabla_{x}D^{\eta}_{x}f_{R_{N}})-\frac{\sigma^2}{2}\triangle_{u}D^{\eta}_{x}f_{R_{N}}=D^{\eta}_{x}g_{R_{N}} ,\,\mbox{  in }\,(\Cc^{\infty}_{c}(\er\times\er^d\times\er^d))'.
\end{equation*}
Applying three times  Theorem \ref{thm:Bouchut-hypoellipticity}-(b), we deduce as before that
$|\nabla_{x}D^{\eta}_{x}f|$ and $|\nabla_{x}\nabla_{u}D^{\eta}_{x}f|$ are  in $L^{2}(B_{z_0}(\frac{R_{N}}{2^{3}}))$.

Now,  from the induction assumption  $\Vert \nabla_{u}D^{\eta}_{u}f\Vert_{L^{2}(B_{z_0}(R_N))}+\Vert \nabla_{x}D^{\eta'}_{u}f\Vert_{L^{2}(B_{z_0}(R_N))}<+\infty$, for $\eta$ and  $\eta'$,   $|\eta|\leq N$, we have that $D^{\eta}_{u}f_{R_N}$ satisfies
\begin{equation*}
\partial_{t}D^{\eta}_{u}f_{R_N}-(u\cdot \nabla_{x}D^{\eta}_{u}f_{R_N})-\frac{\sigma^2}{2}\triangle_{u} D^{\eta}_{u}f_{R_N} =  D^{\eta}_{u} g_{R_{N}}
+ \left(D^{\eta}_{u} (u\cdot \nabla_{x} f_{R_N}) - (u\cdot \nabla_{x}D^{\eta}_{u}f_{R_N}) \right),
\end{equation*}
in $(\Cc^{\infty}_{c}(\er\times\er^d\times\er^d))'$.
Since 
\[\|D^\eta_{u} (u\cdot \nabla_{x} f_{R_N}) - (u\cdot \nabla_{x}D^\eta_{u}f_{R_N}) \|\leq \sum_{\eta'; |\eta'| = N-1}  \Vert \nabla_{x}D^{\eta'}_{u}f\Vert_{L^{2}(B_{z_0}(R_N))}<+\infty,\]
applying Theorem \ref{thm:Bouchut-hypoellipticity}-(a), we deduce as before that $\triangle_u D^\eta_{u} f\in L^{2}(B_{z_0}(\frac{R_N}{2}))$,
which ensures that $\|\nabla_{u} D^\eta_{u}f|\in L^{2}(B_{z_0}(\frac{R_N}{2}))$. By applying Theorem \ref{thm:Bouchut-hypoellipticity}-(b) three times, we obtain that  $|\nabla_x D^\eta_{u}f|\in L^{2}(B_{z_0}(\frac{R_N}{2^3}))$. This ends the proof of the  induction $N+1$.

We iterate $m$ times this induction  and   conclude that, for $r:=\frac{r_0}{2^{3m}}$,
\begin{equation*}
\sum_{\eta\in\mathbb{N}^{d};|\eta|\leq m}\Vert D^\eta_{x}f\Vert_{L^{2}(B_{z_0}(r))}+\sum_{\kappa\in\mathbb{N}^{d};|\kappa|\leq m+1}\Vert D^{\kappa}_{u}f\Vert_{L^{2}(B_{z_0}(r))}<+\infty.
\end{equation*}

\paragraph{Step 2.} Finally, we estimate $\Vert \partial^{2}_{t}f\Vert_{L^{2}(B_{z_0}(r))}$.
Since $\nabla_x f$ and $g_{\frac{r_0}{2^3}}$ are in $L^{2}(B_{z_0}(\frac{r_0}{2^3}))$, according to Theorem \ref{thm:Bouchut-hypoellipticity}-(a), we have
\begin{align*}
\Vert \partial_{t}f_{\frac{r_0}{2^3}}\Vert_{L^{2}(B_{z_0}(\frac{r_0}{2^3}))}
&\leq \Vert\partial_{t}f_{\frac{r_0}{2^3}} + (u\cdot \nabla_{x}f_{\frac{r_0}{2^3}})\Vert_{L^{2}(B_{z_0}(\frac{r_0}{2^3}))}
+\Vert (u\cdot \nabla_{x}f_{\frac{r_0}{2^3}})\Vert_{L^{2}(B_{z_0}(\frac{r_0}{2^3}))}\\
&\leq C\Vert g_{\frac{r_0}{2^3}}\Vert_{L^{2}(B_{z_0}(\frac{r_0}{2^3}))}+\frac{r_0}{2^3}\Vert \nabla_{x}f_{\frac{r_0}{2^3}}\Vert_{L^{2}(B_{z_0}(\frac{r_0}{2^3}))} < +\infty.
\end{align*}
Moreover  $\partial_{u_i}f_{\frac{r_0}{2^3}}$ satisfies \eqref{proof_step4:Schauder estimate} and $\partial_{x_i}f_{\frac{r_0}{2^3}}$ satisfies
\begin{equation*}
\partial_{t}\partial_{x_i}f_{\frac{r_0}{2^3}}-(u\cdot \nabla_{x}\partial_{x_i}f_{\frac{r_0}{2^3}})-\frac{\sigma^2}{2}\triangle_{u}\partial_{x_i}f_{\frac{r_0}{2^3}}=\partial_{x_i}g_{\frac{r_0}{2^3}}\,\mbox{ in }\,(\Cc^{\infty}_{c}(\er\times\er^d\times\er^d))',
\end{equation*}
with  $|\nabla_{x}f_{\frac{r_0}{2^3}}|$, $|\nabla_u f_{\frac{r_0}{2^3}}|$, $\partial^{2}_{x_i,x_j}f_{\frac{r_0}{2^3}}$,
$\partial^{2}_{u_i,u_j}f_{\frac{r_0}{2^3}}$ and $\partial^2_{x_i, u_j}f_{\frac{r_0}{2^3}}$ in $L^{2}(B_{z_0}(\frac{r_0}{2^3}))$ for $1\leq i,j\leq d$.
We easily deduce that $\partial_{x_i}g_{\frac{r_0}{2^3}}$ and $\partial_{u_i}g_{\frac{r_0}{2^3}}$ are also in $L^{2}(B_{z_0}(\frac{r_0}{2^3}))$. From Theorem \ref{thm:Bouchut-hypoellipticity}-(a) again it follows that
\begin{align*}
\Vert \partial_{t}\partial_{x_i}f_{\frac{r_0}{2^3}}\Vert_{L^{2}(B_{z_0}(\frac{r_0}{2^3}))}
&\leq C \Vert \partial_{x_i}g_{\frac{r_0}{2^3}}\Vert_{L^{2}(B_{z_0}(\frac{r_0}{2^3}))}+
\frac{r_0}{2^3}\Vert \nabla_{x}\partial_{x_i}f_{\frac{r_0}{2^3}}\Vert_{L^{2}(B_{z_0}(\frac{r_0}{2^3}))},\\
\Vert \partial_{t}\partial_{u_i}f_{\frac{r_0}{2^3}}\Vert_{L^{2}(B_{z_0}(\frac{r_0}{2^3}))}
&\leq C \Vert \partial_{u_i}g_{\frac{r_0}{2^3}}\Vert_{L^{2}(B_{z_0}(\frac{r_0}{2^3}))}+
\frac{r_0}{2^3}\Vert \nabla_{x}\partial_{u_i}f_{\frac{r_0}{2^3}}\Vert_{L^{2}(B_{z_0}(\frac{r_0}{2^3}))}.
\end{align*}
so that $|\partial_{t}\nabla_x f_{\frac{r_0}{2^3}}|$ and $|\partial_{t} \nabla_u f_{\frac{r_0}{2^3}}|$ are in
$L^{2}(B_{z_0}(\frac{r_0}{2^3}))$. Now we observe that $\partial_{t}f_{\frac{r_0}{2^3}}$ satisfies
\begin{equation*}
\partial^2_t f_{\frac{r_0}{2^3}}-(u\cdot \nabla_x \partial_t f_{\frac{r_0}{2^3}})-\frac{\sigma^2}{2}\triangle_u \partial_t f_{\frac{r_0}{2^3}}=\partial_t g_{\frac{r_0}{2^3}}\,\mbox{ in }\,(\Cc^{\infty}_{c}(\er\times\er^d\times\er^d))',
\end{equation*}
with
\begin{equation*}
\partial_{t}g_{\frac{r_0}{2^3}}=\Gamma_{\frac{r_0}{2^3}}\partial_t f+(\Psi_{\frac{r_0}{2^3}}\cdot \nabla_u\partial_{t}f)+(\partial_t \Gamma_{\frac{r_0}{2^3}})f + \left(\partial_t \Psi_{\frac{r_0}{2^3}}\cdot \nabla_u f\right)\in L^2(B_{z_0}(\frac{r_0}{2^3})).
\end{equation*}
It follows that $\partial^{2}_{t} f\in L^{2}(B_{z_0}(\frac{r_0}{2^4}))$ since
\begin{equation*}
\Vert\partial^2_{t}f_{\frac{r_0}{2^3}}\Vert_{L^{2}(B_{z_0}(\frac{r_0}{2^3}))}\leq  C \Vert\partial_{t}g_{\frac{r_0}{2^3}}\Vert_{L^{2}(B_{z_0}(\frac{r_0}{2^3}))}+
\frac{r_0}{2^3}\Vert \nabla_x(\partial_t f)\Vert_{L^{2}(B_{z_0}(\frac{r_0}{2^3}))}<+\infty.
\end{equation*}
This enables us to conclude on \eqref{proof_step1:Schauder estimate}.

\subsection{End of the proof of Lemma \ref{lem:ExistVFPlineaire}}
\label{subsec:Annexe_EDP_4}

We establish that there exists a unique solution $f\in \mathcal{H}(\omega,Q_{T})$ to Eq. \eqref{eq:ToyEquation}.
When $\rho_{0}\in L^{2}(\Dd\times\er^{d})$, $q\in L^{2}(\Sigma^{-}_{T})$, $g\in L^{2}(Q_{T})$ and  $B^{}\in L^{\infty}((0,T)\times\Dd;\er^d)$, Carrillo \cite{carrillo-98} has established the existence and uniqueness of a weak solution $S$ to \eqref{eq:ToyEquation} in $\mathcal{H}(Q_{T})$, as well as the nonnegative properties stated in Lemma
\ref{lem:ExistVFPlineaire}. Since $\Hh(\omega,Q_{T})\subset\mathcal{H}(Q_{T})$, the uniqueness and the nonnegative property are obviously preserved. For the existence, we adapt the original proof of \cite{carrillo-98}.
We set
\begin{align*}
E:=& \Hh(\omega,Q_{T})\,\mbox{equipped with the scalar product}\,\left(\phi,\psi\right)_{E}:=\int_{Q_{T}}\omega\psi\phi+\int_{Q_{T}}\omega\left(\nabla_{u}\psi\cdot\nabla_{u}\phi\right),
\end{align*}
\begin{align*}
F:=&\left\{\psi\in\mathcal{C}^{\infty}_{c}(\overline{Q_{T}})~\mbox{s.t.}~\psi = 0~\mbox{on }\left(\{T\}\times\Dd \times\er^{d}\right)\cup\Sigma_{T}^{+}\right\}~ \mbox{equipped with the norm }|\psi|_{F}~\mbox{s.t.}
\end{align*}
\begin{align*}
&|\psi|^{2}_{F}=\int_{Q_{T}}\omega \left|\psi\right|^{2} +\int_{\Sigma^{-}_{T}}|(u\cdot
n_{\Dd })|\omega(u)\left|\psi(s,x,u)\right|^{2} d\lambda_{\Sigma_{T}}(s,x,u)
+\int_{Q_{T}}\omega|\nabla_{u}\psi|^{2},
\end{align*}
\begin{align*}
A(\phi,\psi):=&- \int_{Q_{T}}\phi\trans(\omega\psi)
+\eta \int_{Q_{T}}\omega \psi \phi
+\int_{Q_{T}}\omega\psi\left(B^{}\cdot \nabla_{u}\phi\right)
+\frac{\sigma^{2}}{2}\int_{Q_{T}}\left(\nabla_{u}(\omega\psi)
\cdot\nabla_{u}\phi\right),
\end{align*}
\begin{align*}
L(\psi)
:= & \int_{\Dd \times\er^{d}}\omega(u)\rho_{0}(x,u)\psi(0,x,u)dx\,du \\
& \quad -\int_{\Sigma^{-}_{T}}\left(u\cdot n_{\Dd}(x)\right)\omega(u)\widetilde{q}(t,x,u)\psi(t,x,u)d\lambda_{\Sigma_{T}}(t,x,u)+\int_{Q_{T}}\omega \widetilde{g}\psi,
\end{align*}
where $\widetilde{q}(t,x,u)=\exp\left\{-\eta t\right\}q(t,x,u)$, $\widetilde{g}(t,x,u)=\exp\left\{-\eta t\right\}g(t,x,u)$, and
$\eta$ is a positive real parameter that we explicit later. Setting $\Vert
\psi\Vert_{\substack{E}}:=\sqrt{\left(\psi,\psi\right)_{E}}$ for the norm
of $E$, we observe that $\Vert \psi\Vert_{E}\leq \left|\psi\right|_{F}$
for $\psi\in F$. The continuity of the injection $J:F
\rightarrow E$ obviously holds true, as well as the continuity
of the application $A(\cdot,\psi):E\rightarrow \er$ for $\psi\in F$ fixed.
For the coercivity of $A$, we check that, for all $\psi\in F$,
\begin{align*}
&A(\psi,\psi)\\
&=\frac{1}{2}\Vert \psi(0)\Vert^{2}_{L^{2}(\omega,Q_{T})}
+\frac{1}{2}\Vert\psi\Vert^{2}_{L^{2}(\omega,\Sigma^{-}_{T})}+\frac{\sigma^{2}}{2}\Vert \nabla_{u}\psi\Vert^{2}_{L^{2}(\omega,Q_{T})}
+ \int_{Q_{T}}\left(\eta\omega
-\frac{1}{2}\left(B^{}\cdot\nabla_{u}\omega\right)
-\frac{\sigma^{2}}{4}\triangle_{u}\omega\right)\left|\psi\right|^{2}.
\end{align*}
Since, by Lemma \ref{lem:major_poids}, for all $(t,x,u)\in Q_{T}$,
\begin{align*}
-\frac{1}{2}\left(B^{}(t,x)\cdot\nabla_{u}\omega(u)\right)
-\frac{\sigma^{2}}{4}\triangle_{u}\omega(u)
\geq \left(-\frac{\alpha\Vert B^{}\Vert_{L^{\infty}((0,T)\times\Dd;\er^{d})}}{2}-\frac{\sigma^{2}}{4}
\left(2\alpha\left(\frac{\alpha}{2}-1\right)+\alpha d\right)
\right)\omega(u),
\end{align*}
the coercivity of $A$ on $F$ is established by choosing $\eta$ large enough. Theorem \ref{Existence_JLLions} then ensures the existence
of $\widetilde{f}\in E$ such that, for all $\phi$ in $F$,
\begin{align*}
&-\int_{Q_{T}}\widetilde{f}\trans(\omega\psi)
+\eta \int_{Q_{T}}\omega\widetilde{f}\psi
+\int_{Q_{T}}\omega \psi\left(B^{}\cdot\nabla_{u}\widetilde{f}\right)
+\frac{\sigma^{2}}{2}\int_{Q_{T}}\left(\nabla_{u}(\omega\psi)
\cdot\nabla_{u}\widetilde{f}\right)-\int_{Q_{T}}\omega \widetilde{g}\psi
\\
&=\int_{\Dd\times\er^{d}}\omega(u)\rho_{0}(x,u)\psi(0,x,u)\,dx\,du-
\int_{\Sigma^{-}_{T}}(u\cdot n_{\Dd}(x))\omega(u) \widetilde{q}(s,x,u)\psi(s,x,u)
\,d\lambda_{\Sigma_{T}}(s,x,u).
\end{align*}
In the case ${\displaystyle \psi=\frac{\phi}{\sqrt{\omega}}}$ for
$\phi\in\Cc^{\infty}_{c}(Q_{T})$, the preceding expression writes as
\begin{equation}
\label{eq:SemiVFP}
\begin{aligned}
-\int_{Q_{T}}\sqrt{\omega}\widetilde{f}\trans(\phi)&=
-\eta \int_{Q_{T}}\sqrt{\omega}\widetilde{f}\phi
-\int_{Q_{T}}\sqrt{\omega} \phi\left(B^{}\cdot\nabla_{u}\widetilde{f}\right)
-\frac{\sigma^{2}}{2}\int_{Q_{T}}\sqrt{\omega}\left(\nabla_{u}\phi
\cdot\nabla_{u}\widetilde{f}\right)\\
&\quad-\frac{\sigma^{2}}{2}\int_{Q_{T}}\phi\left(\nabla_{u}\sqrt{\omega}
\cdot\nabla_{u}\widetilde{f}\right)+\int_{Q_{T}}\sqrt{\omega}\widetilde{g}\phi,
\end{aligned}
\end{equation}
from which we deduce that
\begin{equation*}
\Vert \trans(\sqrt{\omega}\widetilde{f})\Vert_{\Hh'(Q_{T})}
\leq \left(\eta+\Vert B^{}\Vert_{L^{\infty}((0,T)\times\Dd;\er^{d})}
+\frac{\sigma^{2}}{2}\left(1+\frac{\alpha}{2}\right)\right)\Vert\widetilde{f}\Vert_{E}+\Vert\widetilde{g}\Vert_{L^{2}(\omega,Q_{T})}.
\end{equation*}
According to Lemma \ref{lem:Derive_Carrillo}, $\widetilde{f}$ admits
traces $\gamma^{\pm}(\widetilde{f})$ on $\Sigma^{\pm}_{t}$ satisfying the Green formula \eqref{Greenformula}. In
particular, for all $\psi\in\Cc^{\infty}_{c}(\overline{Q_{T}})$ vanishing on $(\{T\}\times\Dd\times\er^{d})\cup\Sigma^{+}_{T}$,
\begin{align*}
&\left(\trans(\widetilde{f}),
\psi\right)_{\mathcal{H}'(Q_{T}),\mathcal{H}(Q_{T})}+ \left(\trans(\psi),
\widetilde{f}\right)_{\mathcal{H}'(Q_{T}),\mathcal{H}(Q_{T})}
\\
&=-\int_{\Dd\times\er^{d}}\widetilde{f}(0,x,u)\psi(0,x,u)\,dx\,du
+\int_{\Sigma^{-}_{T}}(u\cdot n_{\Dd}(x))\psi(s,x,u)
\gamma^{-}(\widetilde{f})(s,x,u)
\,d\lambda_{\Sigma_{T}}(s,x,u).
\end{align*}
From this expression, replicating the arguments of \cite{carrillo-98},
we establish that $\widetilde{f}(0,\cdot)=\rho_{0}$ on $\Dd\times\er^{d}$ and
$\gamma^{-}(\widetilde{f})=\widetilde{q}$ on $\Sigma^{-}_{T}$. Combined with \eqref{eq:SemiVFP}, we hence obtain that
\begin{equation*}
\left\{
\begin{aligned}
&\trans(\widetilde{f})=
\frac{\sigma^{2}}{2}\triangle_{u}\widetilde{f} -
\left(B^{}\cdot \nabla_{u}\widetilde{f}\right)+
 \widetilde{g}-\eta\widetilde{f},~\mbox{in}~\Hh'(Q_{T}),\\
&\widetilde{f}(0,x,u)=\rho_{0}(x,u),~\mbox{on}~\Dd\times\er^{d},\\
&\gamma^{-}(\widetilde{f})(t,x,u)=\widetilde{q}(t,x,u),~\mbox{on}~\Sigma^{-}_{T}.
\end{aligned}
\right.
\end{equation*}
Taking $f(t,x,u)=\exp\left\{\eta t\right\}\widetilde{f}(t,x,u)$,
and observing that
\begin{equation*}
\Vert \Tt(\sqrt{\omega}f)\Vert_{\Hh'(Q_{T})}\leq \left(\frac{\sigma^{2}}{2}(1+\frac{\alpha}{{2}})+\Vert B\Vert_{L^{\infty}((0,T)\times\Dd;\er^d)}\right)\Vert \nabla_{u}f\Vert_{L^{2}(\omega,Q_{t})}+\Vert g\Vert_{L^{2}(\omega,Q_{T})}<+\infty,
\end{equation*}
we deduce that $f\in \Hh(\omega,Q_{T})$ is a weak solution to \eqref{eq:ToyEquation}.
\subsection{Proof of Lemma \ref{lem:GoodMaxwellian}}\label{sec:appendix-proof}
\paragraph{Proof of $(i{1})$.}
Since $2\mu>1$, by Jensen's inequality,  $m^{2\mu}(t,u)\leq (G({\sigma^{2}}t)*  \mathit{p}^{2}_{0})(u)$.
By setting $f(u):=(1+|u|)\omega(u)$, we thus have, by Lemma \ref{lem:major_poids}-$(i)$,
\begin{align*}
\int_{\er^{d}}(1+|u|)\omega(u)\left|\mathit{p}(t,u)\right|^{2}\,du
&\leq  \exp\{2at\}\int_{\er^{d}}f(u)\left(G({\sigma^{2}}t)*  \mathit{p}_{0}^{2}\right)(u)\,du\\
&\leq 2^{\frac{\alpha}{2}}\exp\{at\}\left(\Vert\mathit{p}_{0}\Vert^{2}_{L^{2}(\er^{d})}
\int_{\er^{d}}f(u)G({\sigma^{2}}t,u)\,du
+\int_{\er^{d}}f(u')\mathit{p}^2_{0}(u')\,du'\right).
\end{align*}
Since the Gaussian density has all its moments finite, $\int_{\er^{d}}f(u)G({\sigma^{2}}t,u)\,du<+\infty$ and $(i{1})$ follows from {the assumption} \eqref{Integrabilite_Cond_initial}.
\paragraph{Proof of $(i{2})$.}
Let us remark that
\begin{align*}
\int_{\er^{d}}\mathit{p}(t,u)\,du\geq \exp\{-|a|T\}\int_{\er^{d}}\left(G({\sigma^{2}}t)*\mathit{p}^{\frac{1}{\mu}}_{0}(u)\right)^{\mu}\,du.
\end{align*}
In addition, since $\mu>1$, H\"older's inequality yields
\begin{align*}
\left(\int_{\er^{d}}G(\nu,u)
\left(G({\sigma^{2}}t)* \mathit{p}^{\frac{1}{\mu}}_{0}(u)\right)\,du\right)^{\mu}\leq \Vert G({\sigma^{2}}\nu)\Vert^{\mu}_{L^{\mu'}(\er^{d})}
\int_{\er^{d}}\left(G({\sigma^{2}}t)
*  \mathit{p}^{\frac{1}{\mu}}_{0}(u)\right)^{\mu}\,du,
\end{align*}
where $\mu'$ is the conjugate of $\mu$, $\nu$ is a positive {constant} and
$\Vert G(\nu)\Vert^{\mu}_{L^{\mu'}(\er^{d})} = (2\pi{\sigma^{2}}\nu)^{\frac{-d}{2}}(\mu')^{\frac{-d}{2(\mu'-1)}}>0$.
Setting $C_{\mu,\nu}:=\exp\{|a|T\}\Vert G({\sigma^{2}}\nu)\Vert^{\mu}_{L^{\mu'}(\er^{d})}$,
we then observe that
\begin{align*}
\inf_{t\in(0,T)}\int_{\er^{d}}\mathit{p}(t,u) \,du
&\geq C^{-1}_{\mu,\nu}\inf_{t\in(0,T)}
\left(\int_{\er^{d}}G({\sigma^{2}}\nu,u)
G({\sigma^{2}}t)* \mathit{p}^{\frac{1}{\mu}}_{0}(u)\,du\right)^{\mu}
\\
&\geq C^{-1}_{\mu,\nu}\inf_{t\in(0,T)}
\left(\int_{\er^{d}}G({\sigma^{2}}\nu)* G(t,u)
\mathit{p}^{\frac{1}{\mu}}_{0}(u_{0})\,du_{0}\right)^{\mu}
\\
&\geq C^{-1}_{\mu,\nu}\left(\frac{1}{2\pi{\sigma^{2}}(\nu+t)}
\right)^{\frac{d\mu}{2}}\left(\int_{\er^{d}}
\exp\left\{\frac{-|u_{0}|^{2}}{2{\sigma^{2}}\left(\nu+t\right)}\right\}\mathit{p}^{\frac{1}{\mu}}_{0}(u_{0})\,du_{0}\right)^{\mu},
\end{align*}
by the explicit convolution product between Gaussian functions. We thus get the lower-bound
\begin{align*}
\inf_{t\in(0,T)}\int_{\er^{d}} \mathit{p}(t,u)\,du
\geq C^{-1}_{\mu,\nu}\left(\frac{1}{2\pi {\sigma^{2}}(T+\nu)}
\right)^{\frac{d\mu}{2}}\left(\int_{\er^{d}}
\exp\left\{\frac{-|u_{0}|^{2}}{2{\sigma^{2}}\nu}\right\}\mathit{p}^{\frac{1}{\mu}}_{0}(u_{0})\,du_{0}\right)^{\mu}.
\end{align*}
Since $\mathit{p}_{0}$ is assumed to be not identically zero on $\er^{d}$, we conclude $(i{2})$.
\paragraph{Proof of $(i{3})$.}
 For all $\mu>0$, it is enough to show that for some real sequence $\left\{\epsilon_{k};~k\in\mathbb{N}\right\}$ decreasing to $0$,
\begin{align}
\label{Converg_Maxwell_temps}
\lim_{\substack{k\rightarrow +\infty}}\int_{\er^{d}}
\left|\mathit{p}^{\frac{1}{\mu}}(\epsilon_{k},u)-\mathit{p}^{\frac{1}{\mu}}_{0}(u)
\right|^{2\mu}\,du =0.
\end{align}
Indeed, in the case $\mu\leq 1$, recalling that:
\begin{equation*}
\left||c|-|b|\right|^{q}\leq \left||c|^{q}-|b|^{q}\right|,
~\mbox{for}~c,b~\in\er,~q\geq 1,
\end{equation*}
it holds, for all $k$,
\begin{equation*}
\left|\mathit{p}(\epsilon_{k},u)-\mathit{p}_{0}(u)\right|^{2}
=\left(\left|\mathit{p}(\epsilon_{k},u)-
\mathit{p}_{0}(u)\right|^{\frac{1}{\mu}}\right)^{2\mu}
\leq |\mathit{p}^{\frac{1}{\mu}}(\epsilon_{k},u)-\mathit{p}^{\frac{1}{\mu}}_{0}(u)|^{2\mu}.
\end{equation*}
Then \eqref{Converg_Maxwell_temps} will imply $(i{3})$
for all $\left\{\epsilon_{k};~k\in\mathbb{N}\right\}$ considered above.

In the case $\mu > 1$, \eqref{Converg_Maxwell_temps} yields the
convergence $\mathit{p}(\epsilon_{k})\rightarrow
\mathit{p}_{0}(\cdot)$ in $L^{2\mu}(\er^{d})$. Applying [Theorem $4.9$, Brezis \cite{brezis-11}], we deduce
the existence of a subsequence of $\{\mathit{p}(\epsilon_{k});~k\in\mathbb{N}\}$ such that
\begin{align*}
\lim_{\substack{k\rightarrow +\infty}}\mathit{p}(\epsilon_{k},u)=\mathit{p}_{0}(u)~\mbox{a.e.}~u\in\er^{d},
\,\mbox{and}\,\,\sup_{k\in\mathbb{N}}\left|\mathit{p}(\epsilon_{k},u)\right|^{\frac{1}{\mu}}\in L^{2\mu}(\er^{d}).
\end{align*}
Since
$\sup_{\substack{k}\in\mathbb{N}}\left|\mathit{p}(\epsilon_{k},u)\right|^{2}\leq
(\sup_{\substack{k}\in\mathbb{N}}\left|\mathit{p}(\epsilon_{k},u)\right|^{\frac{1}{\mu}})^{2\mu}$,
$(i3)$ follows from the Lebesgue Dominated Convergence Theorem.

We now check that \eqref{Converg_Maxwell_temps} holds true: By definition
$\int_{\er^{d}}|\mathit{p}^{\frac{1}{\mu}}(t,u)-\mathit{p}^{\frac{1}{\mu}}_{0}(u)|^{2\mu}\,du$ is bounded from above by
\begin{equation*}
\left|\exp\left\{2 a \epsilon_{k}\right\}-1\right|\int_{\er^{d}}|\mathit{p}_{0}(u)|^{2}\,du+
\exp\left\{2 a \epsilon_{k}\right\}\int_{\er^{d}}|m(\epsilon_{k},u)-\mathit{p}^{\frac{1}{\mu}}_{0}(u)|^{2\mu}\,du.
\end{equation*}
According to assumption \eqref{Integrabilite_Cond_initial},  $\int_{\er^{d}} {p}^2_{0}(u)\,du$ is finite so
the first term in the expression above tends to $0$ when $k$ goes to infinity.
For the second term, a change of variables and H\"older's inequality give
\begin{align*}
\int_{\er^{d}}|m(\epsilon_{k},u)-\mathit{p}^{\frac{1}{\mu}}_{0}(u)|^{2\mu}\,du
&=\int_{\er^{d}}\left|\int_{\er^{d}}G({\sigma^{2}},u_{0})\mathit{p}^{\frac{1}{\mu}}_{0}
(u-\sqrt{\epsilon_{k}}u_{0})\,du_{0}-\mathit{p}^{\frac{1}{\mu}}_{0}(u)\right|^{2\mu}\,du
\\
&\leq\int_{\er^{d}}G({\sigma^{2}},u_{0})\left(\int_{\er^{d}}|\mathit{p}^{\frac{1}{\mu}}_{0}
(u-\sqrt{\epsilon_{k}}u_{0})-\mathit{p}^{\frac{1}{\mu}}_{0}(u)|^{2\mu}\,du\right) \,du_{0}.
\end{align*}
Since $\mathit{p}^{\frac{1}{\mu}}_{0}\in L^{2\mu}(\er^{d})$, Corollary \eqref{coro:LebesgueContinuity} implies that
\begin{equation*}
\lim_{k\rightarrow +\infty}\int_{\er^{d}}|\mathit{p}^{\frac{1}{\mu}}_{0}
(u-\sqrt{\epsilon_{k}}u_{0})-\mathit{p}^{\frac{1}{\mu}}_{0}(u)|^{2\mu}\,du=0.
\end{equation*}
We check also that
\begin{equation*}
\int_{\er}G({\sigma^{2}},u_{0})\sup_{k}\left(\int_{\er^{d}}|\mathit{p}^{\frac{1}{\mu}}_{0}
(u-\sqrt{\epsilon_{k}}u_{0})-\mathit{p}^{\frac{1}{\mu}}_{0}(u)|^{2\mu}\,du\right)\,du_{0}\leq
2\int_{\er^{d}} {p}^2_{0}(u)\,du<+\infty,
\end{equation*}
in order to conclude that
$\lim_{t\rightarrow 0^{+}}\int_{\er^{d}}|m(t,u)-
\mathit{p}^{\frac{1}{\mu}}_{0}(u)|^{2\mu}\,du=0$
by dominated convergence.
\paragraph{Proof of $(i{4})$.}
For all $\delta>0$, $(t,u)\in[\delta,T]\times\er^{d} \mapsto \mathit{p}(t,u)$ is $\Cc^{\infty}$ and
\begin{equation*}
\partial_{t}\mathit{p}(t,u)=a\mathit{p}(t,u)+
\mu\exp\left\{at\right\}\partial_{t}m(t,u)m^{\mu-1}(t,u),
\end{equation*}
from which we get, for all $t>0$,
\begin{equation*}
\int_{\er^{d}} \partial_{t}
{p}^2(t,u)\,du
\leq 2|a|\int_{\er^{d}}{p}^{2}(t,u)\,du
+2\mu^{2}\exp\left\{2at\right\}\int_{\er^{d}}
\left|\partial_{t}m(t,u)\right|^{2}m^{2(\mu-1)}(t,u)\,du.
\end{equation*}
By $(i{1})$, $\int_{\er^{d}}{p}^2(t,u)\,du$ is finite. For the second term, observe that
\begin{equation*}
\left|\partial_{t}m(t,u)\right|\leq
\int_{\er^{d}}\left|\frac{d}{dt}G({\sigma^{2}}t,u-u_{0})\right|
\mathit{p}^{\frac{1}{\mu}}_{0}(u_{0})\,du_{0},
\end{equation*}
where
\begin{equation*}
\frac{d}{dt}G({\sigma^{2}}t,u-u_{0})=
\left(\frac{1}{2\pi{\sigma^{2}} t}\right)^{\frac{d}{2}}\left(
\frac{\left|u-u_{0}\right|^{2}-{\sigma^{2}}d t}{2 {\sigma^{2}}t^{2}}
\right)\exp\left\{-\frac{\left|u-u_{0}\right|^{2}}{
2{\sigma^{2}} t}\right\}
\end{equation*}
Using the inequality
\begin{equation}
\label{Borne_deriv_Maxwell_step2}
|z|^{p}\exp\left\{\frac{-|z|^{2}}{4\theta}\right\}
\leq \left(2 p \theta\right)^{\frac{p}{2}}
~\mbox{for}~\theta>0,~p\geq 1,~z\in\er,
\end{equation}
it follows that
\begin{equation*}
\left|\partial_{t}m(t,u)\right|\leq \frac{d+1}{2\delta}
\int_{\er^{d}}
\left(\frac{1}{2\pi{\sigma^{2}} t}\right)^{\frac{d}{2}}
\exp\left\{\frac{-\left|u-u_{0}\right|^{2}}{4{\sigma^{2}}t}\right\}\mathit{p}^{\frac{1}{\mu}}_{0}(u_0)\,du_{0}=
\frac{2^{\frac{d}{2}}(d+1)}{2\delta}G(2{\sigma^{2}}t)*\mathit{p}^{\frac{1}{\mu}}_{0}(u).
\end{equation*}
Using the upper-bound $m(t,u)\leq 2^{\frac{d}{2}}G(2{\sigma^{2}}t)*\mathit{p}^{\frac{1}{\mu}}_{0}(u)$,
it follows that
\begin{align*}
&\int_{\left(\delta,T\right)\times\er^{d}}
\exp\left\{2at \right\}|\partial_{t}m(t,u)|^{2}m^{2(\mu-1)}(t,u)\,dt\,du
\\
&\leq 2^{d\mu}(\frac{d+1}{2\delta})^{2\mu}\int_{\left(\delta,T\right)\times\er^{d}}
\exp\left\{2at \right\}\left(G(2{\sigma^{2}}t)*\mathit{p}^{\frac{1}{\mu}}_{0}(u)\right)^{2\mu}\,dt \,du
\\
&\leq 2^{d\mu}(\frac{d+1}{2\delta})^{2\mu}T\int_{\er^{d}}
{p}^2_{0}(u_{0})\,du_{0}
\int_{0}^{T}\exp\left\{2at \right\}\,dt.
\end{align*}
Since $\mathit{p}_{0}\in L^{2}(\er^{d})$, we deduce $(i{4})$.
\paragraph{Proof of $(i{5})$.}
To prove that $\nabla_{u}\mathit{p}\in L^{2}(\left(0,T\right)\times\er^{d})$, we use
 the sequence $\{\epsilon_{k};~k\in\mathbb{N}\}$ given in $(i{3})$. For all $k$, we have
\begin{equation}\label{I_proof_step2}
\begin{aligned}
\int_{\left(\epsilon_{k},T\right)\times\er^{d}}|\nabla_{u}
\mathit{p}(t,u)|^{2}dt\,du
&=
\mu^{2}\int_{\left(\epsilon_{k},T\right)\times\er^{d}}\exp\{2at\}
\left|\nabla_{u}m(t,u)\right|^{2}
m^{2\mu-2}(t,u)\,dt\,du\\
&=\frac{1}{2\mu-1}
\int_{\left(\epsilon_{k},T\right)\times\er^{d}}\exp\{2at\}
\left(\nabla_{u}m(t,u)\cdot \nabla_{u}\left(m^{2\mu-1}(t,u)\right)\right)dt\,du.
\end{aligned}
\end{equation}
Since $2\mu-1>0$ and $\left|\nabla_{u}m(t,u)\right|\leq
({\sigma^{2}}t)^{-1}\int_{\er^{d}}
|u-u_{0}|G(t,u-u_{0})\mathit{p}^{\frac{1}{\mu}}_{0}(u_{0})\,du_{0}$, the smoothness of $G$ and the assumptions
\eqref{Integrabilite_Cond_initial} yield
\begin{equation*}
{\displaystyle \lim_{\substack{|u|\rightarrow +\infty}}
\left|\nabla_{u}m(t,u)
\right|m^{2\mu-1}(t,u) =0},~\mbox{for all}~t\in [\epsilon_{k},T].
\end{equation*}
By integrating by parts the right member of \eqref{I_proof_step2} and using the heat equation
$\triangle_{u}m =\frac{2}{\sigma^{2}}\partial_{t}m$, we get
\begin{align*}
&\int_{\left(\epsilon_{k},T\right)\times\er^{d}}\exp\{2at\}
\left|\nabla_{u}m(t,u)\right|^{2}
m^{2\mu-2}(t,u)\,dt\, du\\
& =  \frac{-1}{2\mu-1}\int_{\left(\epsilon_{k},T\right)\times\er^{d}}\exp\{2at\}
\triangle_{u}m(t,u)m^{2\mu-1}(t,u)\,dt\,du
\\
&= \frac{-2}{\sigma^{2}(2\mu-1)}\int_{\left(\epsilon_{k},T\right)\times\er^{d}}\!\exp\{2at\}\partial_{t}m(t,u)m^{2\mu-1}(t,u)\,dt\,du
\\
&=\frac{-1}{\sigma^{2}\mu(2\mu-1)}\int_{\left(\epsilon_{k},T\right)\times\er^{d}}\!\exp\{2at\}
\partial_{t}\left(m^{2\mu}(t,u)\right)\,dt\,du
\end{align*}
Using again an integration by parts enables us to obtain the equality
\begin{align*}
&\int_{\left(\epsilon_{k},T\right)\times\er^{d}}\exp\{2at\}
\left|\nabla_{u}m(t,u)\right|^{2}
m^{2\mu-2}(t,u)\,dt\, du
\\
&=\frac{-1}{\sigma^{2}\mu(2\mu-1)}\left(\int_{\er^{d}}{p}^{2}(T,u)\,du -
\int_{\er^{d}}
{p}^{2}(\epsilon_{k},u)\,du\right) +\frac{2a}{\sigma^{2}\mu(2\mu-1)}
\int_{(\epsilon_{k},T)\times\er^{d}}{p}^{2}(t,u)\,dt\,du.
\end{align*}
Coming back to \eqref{I_proof_step2} and letting $k$ increase to $+\infty$, it follows that
\begin{align*}
&\int_{(0,T)\times\er^{d}}\left|\nabla_{u}\mathit{p}(t,u)\right|^{2}\,dt\,du \\
&= \frac{-\mu}{\sigma^{2}(2\mu-1)}\left(\int_{\er^{d}}
{p}^{2}(T,u)\,du -
\int_{\er^{d}}{p}^{2}_{0}(u)\,du \right)
 +\frac{2a\mu}{\sigma^{2}(2\mu-1)}
\int_{\left(0,T\right)\times\er^{d}}{p}^{2}(t,u)\,dt\,du.
\end{align*}
Thanks to the assumption \eqref{Integrabilite_Cond_initial},
 $\mathit{p}_{0}\in L^{2}(\er^{d})$ and $(i1)$, the right member is finite. We conclude $(i5)$.

\end{document}